\documentclass[11pt]{article}

\usepackage{amsthm,amsfonts,amssymb,amsmath,color}
\usepackage{indentfirst}
\allowdisplaybreaks[4]

\usepackage{hyperref}
\hypersetup{
  colorlinks = true, 
  urlcolor = blue, 
  linkcolor = blue, 
  citecolor = blue 
}

\numberwithin{equation}{section}

\renewcommand\d{\partial}
\renewcommand\a{\alpha}

\newcommand\s{\sigma}

\newcommand\R{\mathbb R}

\newcommand\Z{\mathbb Z}

\def\de{\delta}

\def\O{\Omega}
\def\th{\theta}

\def\l{\lambda}

\def\e{\varepsilon}

\def\weak{\rightharpoonup}

\newcommand\br{\begin{rem}}
\newcommand\er{\end{rem}}
\newcommand\bp{\begin{pmatrix}}
\newcommand\ep{\end{pmatrix}}
\newcommand\be{\begin{equation}}
\newcommand\ee{\end{equation}}
\newcommand\ba{\begin{equation}\begin{aligned}}
\newcommand\ea{\end{aligned}\end{equation}}

\newcommand\nn{\nonumber}

\newcommand{\tc}{\color{red}}
\let \tc \relax

\usepackage{layout}

\newcommand{\RR}{{\mathbb R}}

\newcommand{\Id}{{\rm Id }}

\newcommand{\uu}{{\vc u}}
\newcommand{\vv}{{\vc v}}
\newcommand{\ff}{{\vc f}}

\newcommand{\vu}{\vc{u}}

\newcommand{\vc}[1]{{\bf #1}}

\newcommand{\dx}{\, {\rm d} {x}}
\newcommand{\dt}{\, {\rm d} t }

\newcommand{\vU}{\vc{U}}

\newcommand{\dive}{{\rm div\,}}

\newcommand{\eps}{\e}

\newtheorem{defi}{Definition}[section]
\newtheorem{theorem}[defi]{Theorem}
\newtheorem{proposition}[defi]{Proposition}
\newtheorem{lemma}[defi]{Lemma}

\newtheorem{remark}[defi]{Remark}

\numberwithin{equation}{section}

\definecolor{darkgreen}{rgb}{0,0.5,0}

\begin{document}

\title{Qualitative/quantitative homogenization of some non-Newtonian flows in perforated domains}

\author{
Richard M. H\"ofer \footnote{Faculty of Mathematics, University Regensburg, Universit\"atsstraße 31, 93053 Regensburg, Germany, richard.hoefer@mathematik.uni-regensburg.de}
\and
Yong Lu \footnote{School of Mathematics, Nanjing University, Nanjing 210093, China, luyong@nju.edu.cn}
\and
Florian Oschmann \footnote{Mathematical Institute of the Czech Academy of Sciences, \v{Z}itn\'a 609/25, 140 00 Praha, Czech republic, oschmann@math.cas.cz}
}
\date{}

\maketitle

\begin{abstract}

In this paper, we consider the homogenization of stationary and evolutionary  incompressible viscous non-Newtonian flows of Carreau-Yasuda type in domains perforated with a large number of periodically distributed small holes in $\R^{3}$,  where the mutual distance between the holes is measured by a small parameter $\e>0$ and the size of the holes is $\e^{\a}$ with {\tc $\a \in (1, 3)$}. The Darcy's law is recovered in the limit, thus generalizing the results from [\url{https://doi.org/10.1016/0362-546X(94)00285-P}] and  [\url{https://doi.org/10.1016/j.jde.2024.08.021}] for $\alpha=1$. Instead of using their restriction operator to derive the estimates of the pressure extension by duality, we use the Bogovski\u{\i} type operator in perforated domains (constructed in [\url{https://doi.org/10.1051/cocv/2016016}]) to deduce the uniform estimates of the pressure directly.  Moreover, quantitative convergence rates are given. 

\end{abstract}

\renewcommand{\refname}{References}

\tableofcontents


\section{Introduction}

In this paper, we consider the homogenization of stationary and instationary incompressible viscous non-Newtonian flows in three dimensional perforated domains. Starting with the  steady case, we focus on the {\tc Navier-Stokes equations (NSE)} in the perforated domain $\O_\e$:
\begin{equation}\label{NonNew-0}
\begin{cases}
-{\rm div}\,\big(\eta_r(D\uu_\e )D\uu_\e  \big)+{\tc \e^{\lambda'}} \dive (\uu_\e \otimes \uu_\e)  +\nabla p_\e=\vc{f} & {\rm in}\ \O_\e, \\
{\rm div} \,\uu_\e =0 & {\rm in}\  \O_\e,\\
\uu_\e =0 & {\rm on}\ \d\O_\e.
\end{cases}
\end{equation}

Here, $\uu_\e $ is the fluid's velocity, $\nabla \uu_\e $ is the gradient velocity tensor, $D\uu_\e =\frac{1}{2}(\nabla \uu_\e +\nabla^{\rm T}\uu_\e )$ denotes the rate-of-strain tensor, $p_\e$ denotes the fluid's pressure, and $\vc{f} \in L^{2}(\O;\R^{3})$ is the density of the external force, which is assumed to be independent of $\e$ for simplicity. The case $\vc{f}_{\e} \to \vc{f}$ strongly in $L^{2}(\O)$ can be dealt with in the same manner. {\tc The parameter $\e^{\lambda'}$ (with $\l' \in \R$ possibly negative) accounts for the strength of inertia.} 

\medskip

{\tc The non-Newtonian viscosity coefficient $\eta_r(D\uu_\e )$ with $r>1$ is assumed to satisfy the following \emph{$r$-structure} properties including \emph{continuity and strict monotonicity (coercivity)}: there exists a constant $C>0$, such that for any $A,B \in \R_{sym}^{3\times 3}$ and any $\beta>0$, there holds
\ba\label{monG}
& \eta_r: \R_{sym}^{3\times 3} \to [0, \infty)  \ \mbox{is continuous};\\
 &[\eta_r(\beta A)A - \eta_r(\beta B)B]: (A-B) \geq C^{-1} |A-B|^2 + C^{-1} \beta^{r-2}|A-B|^{r} \vc{1}_{r > 2}.
\ea
Here $\R_{sym}^{3\times 3}$ denotes the set of all real $3\times 3$ symmetric matrices, and $\vc{1}_{r > 2}$ is the characteristic function of the set $\{r>2\}$.

We further assume that $\eta_r(D\uu_\e )$ takes the form
\begin{align} \label{def:eta_r}
\eta_r(D \uu_\e) = \eta_0 + g_r(|D \uu_\e|),  \ \eta_0 >0,
\end{align}
where $g_r : [0, \infty) \to \R$ is a continuous function satisfying the \emph{growth condition}
\ba \label{growthG}
|g_r(s)| \leq C s \vc{1}_{s \leq 1} + C s^{\max\{r-2, 0\}} \vc{1}_{s \geq 1}, \qquad \mbox{for some $C > 0$ and all $s\geq 0$}.
\ea

\medskip

A typical example of stress tensor under consideration is the {\tc Carreau-Yasuda} law:
$$\eta_r(D\uu_\e )=(\eta_0-\eta_\infty)(1+\kappa_0 | D\uu_\e |^a)^{\frac{r-2}{a}}+\eta_\infty,\qquad \eta_0>\eta_\infty>0, \ \kappa_0>0, \  r>1, \  a\geq 1,$$
where {\tc$\eta_0/2$} is the zero-shear-rate viscosity, $\kappa_0>0$ is a time constant, and {\tc $(r-2)$} is a dimensionless constant describing the slope in the {\em power law region} of $\log \eta_r$
versus $\log (|D(\uu_{\e})|)$. {\tc The Carreau-Yasuda law is a simple phenomenological law for non-Newtonian flows. It is shear thinning for $r < 2$ and shear thickening for $r>2$.}

Clearly such Carreau-Yasuda law satisfies the properties stated in \eqref{monG} (see for example \cite[Chapter~5]{MNRR}).  Moreover, we may write
\begin{align*}
(\eta_0-\eta_\infty)(1+\kappa_0 | D\uu_\e |^a)^{\frac{r-2}{a}}+\eta_\infty = \eta_0 + (\eta_0 - \eta_\infty) [ (1+\kappa_0 | D\uu_\e |^a)^{\frac{r-2}{a}} - 1 ],
\end{align*}
and the function  $g_r(s) := (1+s^a)^\frac{r-2}{a}-1$ satisfies \eqref{growthG}. 

}



\medskip

The perforated domain $\O_{\e}$ under consideration is described as follows. Let $\O \subset \R^{3}$ be {\tc either the torus $\mathbb{T}^3$,} or a bounded domain of class {\tc $C^{3,\mu}, 0<\mu<1$}. The holes in $\O$ are denoted by $\mathcal{T}_{\e,k}$ satisfying
\begin{equation*}
 \mathcal{T}_{\e,k}=\e x_k+\e^{\a} \overline{\mathcal{T}} \Subset \e Q_k,
\end{equation*}
where $0<\e\ll 1$ is the small perforation parameter used to describe the mutual distance between the holes,  $Q_k=(-\frac{1}{2},\frac{1}{2})^3+k$ is the cube with center $x_k=x_0+k$, where $x_0\in(-\frac{1}{2},\frac{1}{2})^3,\; k\in {\mathbb{Z}}^3$, and $\alpha > 1$. Moreover, $\mathcal{T} \subset \R^{3}$ is a model hole which is assumed to be a simply connected $C^{2,\mu}$  domain contained in $Q_0$. Without loss of generality, we may assume $ 0 \in \mathcal{T} \subset B(0, \frac 18)$ such that, for $\e>0$ small enough,
\begin{align}\label{distHole}
{\rm dist} \ (\mathcal{T}_{\e, k}, \d (\e Q_k)) \geq \frac{\e}{8} > \e^\a.
\end{align}
We will always assume that $\e$ is small enough to satisfy \eqref{distHole}. The perforation parameter $\e>0$ is used to measure the mutual distance , $\e x_k=\e x_0+\e k$ are the locations of the holes, and $\e^{\a}$ is used to measure the size of the holes. In this paper, we are focusing on the case {\tc $1<\a<3$}.  

The perforated domain $\O_\e$ is then defined as
\begin{equation}\label{Omega-e}
  \O_\e=\O\backslash\bigcup_{k\in K_\e} \mathcal{T}_{\e,k},\qquad {\rm where} \ K_\e=\{k\in {\mathbb{Z}}^3:\e\overline{Q}_{k}\subset\O\}.
\end{equation}

\subsection{Previous results}

The study of homogenization problems in fluid mechanics has gained a lot of interest. A full overview of the literature, even related to viscous fluids with Dirichlet boundary conditions in perforated domains, is beyond our scope. 
Tartar \cite{Tar80}  considered the homogenization of steady Stokes equations in porous media {\tc ($\alpha=1$)} and derived Darcy's law. Allaire \cite{All90-1,All90-2} systematically studied the homogenization of steady Stokes and Navier-Stokes equations {\tc in any spatial dimension $d \geq 2$} and showed that the limit systems are determined by the ratio $\sigma_\e$ between the size and the mutual distance of the holes:
$$\sigma_\e=\bigg(\frac{\e^d}{a_\e^{d-2}}\bigg)^\frac{1}{2},\qquad d\geq3;\qquad\sigma_\e=\e\bigg|{\rm log}\frac{a_\e}{\e}\bigg|^\frac{1}{2},\qquad d=2,$$
where $\e$ and $a_\e$ are used to measure the mutual distance of the holes and the size of the holes, respectively. Particularly, if $\lim_{\e\rightarrow 0}\sigma_\e=0$ corresponding to the case of large holes, the homogenized system is the Darcy's law; if $\lim_{\e\rightarrow 0}\sigma_\e=\infty$ corresponding to the case of small holes, there arise the same Stokes equations in the homogeneous domain; if $\lim_{\e\rightarrow 0}\sigma_\e=\sigma_\ast\in (0,+\infty)$ corresponding to the case of critically sized holes, the homogenized equations are governed by the Brinkman law -- a combination of the Darcy's law and the original Stokes equations. The same results were shown in \cite{Lu20} by employing a generalized cell problem inspired by Tartar \cite{Tar80}.

{\tc Quantitative results for the homogenization in the Darcy regime have been shown in \cite{All90-1} for $\Omega = \mathbb T^d$. The case with boundaries has been studied more recently in \cite{BalaziAllaireOmnes24, MPM96, She22} for $\alpha = 1$ and in \cite{JLP25} for $\alpha > 1$.}

Later, the homogenization study is extended to more complicated models describing fluid flows: Mikeli{\'c} \cite{Mik91} for the nonstationary incompressible Navier-Stokes equations, Masmoudi \cite{Mas02} for the compressible Navier-Stokes equations, Feireisl, Novotn{\'y}, and Takahashi \cite{FNT10} for the complete Navier-Stokes-Fourier equations. In all these studies, only the case where the size of the holes is proportional to the mutual distance of the holes {\tc (i.e., $\alpha = 1$)} is considered and  Darcy's law is recovered in the limit.

Recently, cases with different sizes of holes are studied. Feireisl, Namlyeyeva, and Ne{\v c}asov{\'a} \cite{FNN16} studied the case with critical size of holes for the incompressible Navier-Stokes equations and they derived  the Brinkman equation;  Yang and the second author \cite{YL23} studied the homogenization of evolutionary incompressible Navier-Stokes system with large and small sizes of holes. {\tc The first author showed quantitative convergence result for the corresponding problem with vanishing viscosity in \cite{Hoefer23}.} In \cite{DFL17,FL15,LS18}, with collaborators the second author considered the case of small holes for the compressible Navier-Stokes equations and the homogenized equations remain unchanged. With collaborators, the third author also considered the case of small holes for the unsteady compressible Navier-Stokes equations in \cite{OP23} for three dimensional domains, in \cite{NO23} for two dimensional domains, and  in \cite{BO23}  for the case of randomly perforated domains.  The first author, Kowalczyk, and Schwarzacher \cite{HKS21} studied the case of large holes for the compressible Navier-Stokes equations at low Mach number and derived the Darcy's law;  the study in \cite{HKS21}  was extended to the case with critical size of holes  in \cite{BO22} and they derived the incompressible Navier-Stokes equations with Brinkman term. A similar approach for tiny holes and the full Navier-Stokes-Fourier system was given in \cite{BasaricChaud}. A more general setting has been treated in \cite{BFO23} where the authors considered the case of unsteady compressible Navier-Stokes equations at low Mach number under the assumption $\O_{\e}\to \O$ in the sense of Mosco's convergence and they derived the incompressible Navier-Stokes equations. 

There are not many mathematical studies concerning the homogenization of non-Newtonian flows. Mikeli\'{c} and Bourgeat \cite{BM96} considered the stationary case of Carreau type flows under the assumption $a_{\e}\sim \e$ and derived the Darcy's law. Mikeli\'{c} summarized some studies of homogenization of  stationary non-Newtonian flows in Chapter 4 of \cite{Hor97}.  Under the assumption $a_{\e}\sim \e$, the convergence from the evolutionary version of \eqref{NonNew} to the Darcy's law is shown in \cite{LQ23}. {\tc Recently, the Carreau model in a thin domain of thickness $\e$ was studied in \cite{ABSG22} for $\a=1$ and $1<r<2$.} To the best of the authors' knowledge, {\tc the only works of rigorous analytical results for the homogenization of a pure power law fluid in a thin domain provided $\a>1$ are the very recent results of \cite{ASG21, ASG25}. In the present} paper, we  show that the Darcy's law can be recovered from {\tc an even larger class of viscosities than} the Carreau model by homogenization provided $\a \in (1, \frac 32)$, {\tc for all $r>1$, and for all $1 < \a < 3$ if $r=2$}.\\

\subsection{Rescaling of the primitive system and summary of the main results}

Anticipating that the fluid's velocity is (in some sense) small of order $\e^{3-\a}$, we may rescale the system \eqref{NonNew-0} as\footnote{Strictly speaking, we define new functions $\hat \uu_\e = \e^{\a-3} \uu_\e$, $\hat p_\e = p_\e$, $\hat{\vc{f}} = \vc{f}$, and then drop the hats.}
\begin{equation}\label{NonNew}
\begin{cases}
-\e^{3-\a} \dive \big(\eta_r(\e^{3-\a} D\uu_\e )D\uu_\e  \big) + \e^\l\dive (\uu_\e \otimes \uu_\e) +\nabla p_\e=\vc{f} & {\rm in}\ \O_\e, \\
{\rm div} \,\uu_\e =0 & {\rm in}\  \O_\e,\\
\uu_\e =0 & {\rm on}\ \d\O_\e,
\end{cases}
\end{equation}
with $\l = {\tc \lambda' +} 2(3-\a) > 0$.
{\tc This scaling then coincides with setting the Reynolds, Euler, and Froude number equal to ${\rm Re} = \e^{\l + \a - 3}$, ${\rm Eu} = \e^{-\l} $, ${\rm Fr} = \e^\frac{\l}{2}$, respectively. Moreover, the Euler number represents physically the ratio of pressure forces and inertial forces. The scaling ${\rm Eu} = \e^{-\l}$ hence corresponds to the fact that inertial forces can be neglected in the limit, giving heuristically rise to a linear limiting system.\\

 Roughly speaking, our main results can be stated as follows: If  $1 < \a < 3/2$ and $\l > \a$, then the limiting system as $\eps \to 0$ is Darcy's law
\begin{align} \label{Darcy.intro}
\begin{cases}
\frac{1}{2}\eta_0 M_0 \uu =\vc{f}-\nabla p & \text{in }  \O,\\
\dive \uu = 0 & \text{in } \O,\\
\uu \cdot \vc{n} = 0 & \text{on }  \d \O,
\end{cases}
\end{align}
where $M_{0}$ is the permeability tensor which is a positive definite matrix.
We show convergence results both for the steady and the evolutionary version of \eqref{NonNew}. In both cases, our results are threefold. We show
\begin{itemize}
    \item qualitative convergence results for the fluid velocity field assuming a weak solution of the limit equation;
    \item corresponding quantitative convergence results provided the solution to the limit equation is sufficiently regular; the rates are different for $\Omega=\mathbb T^3$ and bounded domains due to the loss of the no slip boundary condition in the limit;
    \item corresponding qualitative and quantitative estimates for the fluid pressure.
\end{itemize}
The convergence rates appear to be new even in the Newtonian case $g_r = 0$ and coincide with those that have been recently obtained for the steady Stokes system in \cite{JLP25}.

We might compare system \eqref{NonNew} to the results of \cite{BM96}, where they considered a Carreau law with $\a=1$. Although claiming that the Reynolds number ${\rm Re} = \e^{-\gamma}$, after a complete rescaling of the equations such that the velocity $v_\e$ in \cite{BM96} is of order one, one indeed finds that their fully rescaled momentum equation reads
\begin{align*}
-\e^2 {\rm div}\,\big(\eta_r(\e^{2-\gamma}D v_\e ) D v_\e  \big) + \e^{2(2-\gamma)} \dive (v_\e \otimes v_\e) + \nabla p_\e=f,
\end{align*}
i.e., ${\rm Re} = \e^{2(1-\gamma)}$.
According to a two-scale ansatz $v_\e(x) \approx v_0(x,x/\e)$ one expects that the term $Dv_\e$ is of order $\e^{-1}$. Consequently, Theorem 1 in \cite{BM96} shows that the homogenized system is indeed Darcy's law (analogously to \eqref{Darcy.intro} above) as long as $\gamma < 1$. In order to compare the heuristics fully to our setting of small holes, we have to take into account the following facts for a hole size $\e^\a$ with $\a>1$: first, the Reynolds number is of order $\e^{\l + \a - 3}$; second, one expects the term $Dv_\e$ to be of order $\e^{-\a}$ around the holes\footnote{See \eqref{nabW} and \eqref{def-vi-pi-2}.}. These considerations yield precisely the scaling in system \eqref{NonNew}, and also show that the term $\e^{3-\a} D\uu_\e$ is expected to vanish only for $\a<\frac32$. 
}

\subsection{\tc Elements of the proof and structure of the remainder of the paper}
The proof of our results combines several methods. At the core is Tartar's homogenization technique via  oscillating test functions (see e.g. \cite{Tar80} for their use in the derivation of Darcy's law starting from steady Stokes in the case $\alpha = 1$). Suitable matrix valued test functions $W_\e$ in the case $1 < \alpha < 3$ have been developed by Allaire in \cite{All90-1, All90-2}. We state properties of those functions in Section \ref{sec:tools} and postpone the proof to the appendix. Additionally, Section  \ref{sec:tools} contains a Poincar\'e and a Korn inequality as well as a Bogovski\u{\i} operator in perforated domains, all of which are standard in the community.

Sections~\ref{sec:StatNSE} and \ref{sec:EvolNSE} are devoted to convergence of the fluid velocity for the stationary and evolutionary Navier-Stokes systems, respectively. Each of these sections is again split into three subsections: first, a proof of the qualitative homogenization result, then a proof of the quantitative convergence in the torus case, and finally adaptions for bounded domains. 
In all of this, we avoid dealing with the pressure by using only divergence free test functions. More precisely, in the qualitative convergence proof, we test the equation with $\phi_\e = W_\e \phi - B_\e(\phi)$, where $B_\e$ is a Bogovski\u{\i}-type operator (see Proposition \ref{lem:Bogovski.W_eps}) such that $\dive \phi_\e = 0$. This strategy is in the spirit of \cite{Hoefer23} as an alternative to the more classical method of extending the pressure by duality via restriction operators, which is employed for instance in  \cite{All90-1, All90-2, Tar80}. It is interesting to remark that this method does not seem to allow us to capture the full range of our results. We refer to the discussion in Section \ref{sec:pressure}.
The quantitative bounds in Sections~\ref{sec:StatNSE} and \ref{sec:EvolNSE} are obtained by inserting $ W_\e \uu - B_\e(\uu)$ into the relative energy inequality of the primitive Navier-Stokes equation, where $\uu$ is a solution to Darcy's law \eqref{Darcy.intro}. For bounded domains, due to the loss of the no slip boundary condition for $\uu$,  $ W_\e \uu - B_\e(\uu)$ is not admissible in the relative energy inequality. We overcome that by a suitable truncation in an $\eps^{(3 -\alpha)/2}$ neighborhood of $\partial \Omega$ similar to \cite{BalaziAllaireOmnes24} through a suitable vector potential of $\uu$. However, since the layer is of size $\eps^{(3 -\alpha)/2}$, which is optimized for the convergence rates, the support of this truncation intersects significantly with the support of $W_\e - \Id$. This makes the analysis slightly more delicate. 
The analysis in the evolutionary case is mostly parallel to the steady case. 

Finally, in Section~\ref{sec:pressure} we give convergence of the pressure for both stationary and evolutionary NSE, qualitatively and quantitatively. Here, as mentioned above, we do not rely on the standard pressure extension via duality. Instead, we simply extend the pressure by zero (which does not work in the case $\alpha = 1$). We then use a suitable Bogovski\u{\i} operator in $\Omega_\e$ taken from \cite{DFL17} to directly estimate the pressure by combination with the convergence of the fluid velocity. In contrast to the classical homogenization results \cite{All90-1, All90-2}, we do not obtain the convergence of the full pressure weakly in $L^2$, though, even in the steady case. Instead, the pressure is decomposed into two  parts, one that converges weakly in $L^2$ to the pressure of Darcy's law, the other vanishes as $\eps \to 0$ but only in $L^s$ for some $s < 2$. This is reminiscent to related convergence result for the pressure in the context of homogenization for compressible fluid (see e.g. \cite{HKS21, Mas02}). Notably the pressure parts related to the nonlinear terms, both the fluid inertia and the nonlinear non-Newtonian part, are only uniformly bounded in $L^s$ for some $s < 2$.



\section{Toolbox}\label{sec:tools}

\subsection{Notations and weak solutions}
We recall some notations of Sobolev spaces. Let $1\leq q\leq \infty$ and $\O$ be a bounded domain. We use the notation $L_0^q(\O)$ to denote the space of $L^q(\O)$ functions with zero mean value:
$$L_0^q(\O)=\Big\{f\in L^q(\O) \ : \ \int_\O f \dx = 0 \Big\}.$$
We use $W^{1,q}(\O)$ to denote classical Sobolev space, and $W^{1,q}_{0}(\O)$ denotes the completion of $C_c^\infty(\O)$ in $W^{1,q}(\O)$. Here $C_c^\infty(\O)$ is the space of smooth functions compactly supported in $\O$. 
For $1\leq  q <\infty$, $W^{1, q }(\mathbb{R}^3)=W^{1, q}_0(\mathbb{R}^3)$. For $1\leq q\leq \infty$, the functional space $W_{0,\rm div}^{1,q}(\O)$ is defined by
$$W_{0,\rm div}^{1,q}(\O)=\left\{\uu \in W_0^{1,q}(\O;\mathbb{R}^3): \ {\rm div}\,\uu=0 \ {\rm in} \, \O\right\}.$$
Throughout the paper, we use the notation $\tilde f$ to denote the zero extension of any $f\in L^{q}(\O_{\e}), 1\leq q \leq \infty$:
\ba 
\tilde f = \begin{cases} f &  \text{in } \O_{\e}, \\
0 &  \text{in } \RR^3 \setminus \O_{\e}.
\end{cases}
\nn
\ea
We use $C$ and $c$ to denote positive constants independent of $\e$ whose values may differ from line to line; $C$ is usually large while $c$ small.

\medskip

We introduce the definition of weak solutions to \eqref{NonNew}:
\begin{defi}\label{def-weak}
We say that {\tc $(\uu_{\e},p_\e)$} is a  weak solution of \eqref{NonNew} in $\O_{\e}$ provided:
\begin{itemize}
\item $\uu_{\e}\in  W_{0, \rm div}^{1,2}(\O_\e) \cap W_{0, \rm div}^{1,r}(\O_\e)$ {\tc and $p_\e \in L_0^{\min\{2, r'\}}(\O_\e)$};
{\tc 
\item There holds the integral identity for any test function $\varphi\in C_c^\infty( \O_\e;\mathbb{R}^3)$:
\ba\label{eq-weak}
& \int_{\O_{\e}} - \e^\l \uu_\e \otimes \uu_\e:\nabla\varphi + \e^{3-\a}\eta_r(\e^{3-\a}D\uu_\e )D\uu_\e : D\varphi + p_\e \dive \varphi \, {\rm d}x = \int_{\O_\e}\vc{f}\cdot\varphi \dx.
\ea
}


\end{itemize}

\end{defi}

Note that from the regularity of $\uu_\e$, we may use it as test function in \eqref{eq-weak} to get the following energy identity:
\ba\label{energy-ineq}
\e^{3-\a} \int_{\O_\e}\eta_r(\e^{3-\a} D\uu_\e )| D\uu_{\e}|^2\,\dx = \int_{\O_\e}\vc{f}\cdot \uu_{\e} \dx.
\ea

\begin{remark} \label{exis-weak}
The existence of weak solutions to non-Newtonian power law models (i.e., $\eta_r(D\uu_\e ) = \mu |D\uu_\e|^{r-2}$) is known due to the classical result of  Ladyzhenskaya \cite{Lad69}  with $r > 11/5$. The existence theory is then extended to more general $r$ and more general forms of the stress tensor: see Diening,  R\r{u}\v{z}i\v{c}ka, and Wolf \cite{DRW10}, Bul{\'i}{\v c}ek, Gwiazda, M\'alek, and \'Swierczewska-Gwiazda \cite{BGMS12}, for $r>6/5$. For the Carreau-Yasuda law, due to the presence of Newtonian part of the stress tensor (i.e., $\eta_0>0$), the existence of weak solutions can be shown for any $r>1$ following the lines of \cite{BGMS12} or \cite{DRW10}, {\tc see also \cite{MNRR}.}
\end{remark}

\subsection{Inverse of divergence and useful lemmas}
We introduce several useful conclusions which will be frequently used throughout this paper. The first one is the Poincar\'e inequality in perforated domains, the proof of which follows the same lines as \cite[Lemma 3.4.1]{All90-2}:  
\begin{lemma}\label{lem-Poincare}
Let $ \uu \in W_0^{1,q}(\O_\e;\mathbb{R}^3)$ with $1 \leq q < 3$ and $\O_{\e}$ be defined in \eqref{Omega-e} with $\a \geq 1$. Then there holds for some constant $C>0$ independent of $\e$ and $\uu$
\begin{equation}\label{Poincare}
 \| \uu \|_{L^q(\O_\e)}\leq C \min \Big \{ \e^{\frac{3-(3-q)\a}{q}}  ,  1 \Big \}\|\nabla  \uu \|_{L^q(\O_\e)}.
\end{equation}

\end{lemma}

\medskip

We then give the following standard Korn type inequality:
\begin{lemma} \label{lem-Korn}
Let $\O_{\e}$ be defined in \eqref{Omega-e} with $\a\geq 1$.  Let $1<q<\infty$. For arbitrary $ \uu \in W_0^{1,q}(\O_{\e};\mathbb{R}^3)$, there holds
\begin{equation}\label{Korn}
  \|\nabla  \uu \|_{L^q(\O_{\e})}  \leq C(q)\|D \uu \|_{L^q(\O_{\e})}.
\end{equation}

\end{lemma}

\medskip

The above two Lemmas are used to derive the uniform estimates for the velocity field. In order to get the estimates for the pressure $p_{\e}$, the idea is to use the  equations, which yield the corresponding  estimates for $\nabla p_{\e}$ in a negative norm, and then to employ the Bogovski\u{\i} operator to deduce the uniform estimates for $p_\e$. To this end, we shall recall the following result of Diening, Feireisl, and the second author (see \cite[Theorem 2.3]{DFL17}) which gives a construction of Bogovski\u{\i} type operator in perforated domains: 
\begin{proposition} \label{Bog-op}
Let $\O_\e$ be defined by \eqref{Omega-e}. Then there exists a linear operator
$$
\mathcal{B}_\e\colon L^q_0(\O_\e) \to W^{1,q}_0(\O_\e;\R^3), \qquad \mbox{for all } \ 1<q<\infty,
$$
such that for arbitrary $f \in L^q_0(\O_\e)$ there holds
$$
\begin{aligned}
&\dive \mathcal{B}_\e(f) = f \qquad \text{ a.e. in }\O_\e, \\
&\|\mathcal{B}_\e(f)\|_{W^{1,q}_0(\O_\e)} \leq C \big(1 + \e^{\frac{(3-q)\a-3}{q}}\big) \|f\|_{L^q(\O_\e)},
\end{aligned}
$$
where the constant $C>0$ is independent of $\e$.
\end{proposition}

\begin{remark} In bounded Lipschitz domain the existence of Bogovski\u{\i} operator is well-known (see \cite{BOG}, \cite{Gal94}).  In this case, the operator norm depends on the Lipschitz character of the domain, which for the perforated domain $\O_\e$ is unbounded as $\e \to 0$ due to the presence of small holes. The above result gives a Bogovski\u{\i} type operator on perforated domain $\O_{\e}$ with a precise dependency of the operator norm on $\e$. In particular, for $q$ in certain range such that $(3-q)\a-3 \geq 0$, such a Bogovski\u{\i} type operator is uniformly bounded. The construction of such a Bogovski\u{\i} type operator was done by Masmoudi in \cite{Mas02} for the case $\a = 1$ (see a sketch proof of such a construction in \cite{Mas04}), where the estimate constant on the right-hand side is $\frac{1}{\e}$ for any $1<q<\infty$. For the case $q = 2$ and any $\a \geq 1$, the construction of such a Bogovski\u{\i} type operator was shown in \cite{FL15} by employing the restriction operator in \cite{All90-1}, and later such a construction was generalized to $\frac 32 < q < 3$ in \cite{Lu21}.
  \end{remark}

\subsection{Local problem}

To show the homogenization process, we need special oscillating test functions {\tc that are identical to those used by Allaire \cite{All90-1, All90-2}} and some estimates for them. These are mostly taken from \cite{All90-1, All90-2} and later refinements. For completeness, we will {\tc prove them in the appendix}.
{\tc
\begin{proposition}\label{lem-local-1}
Let $1<\a<3$. Then there exists a matrix-valued function $W_{\e} \in W^{1,2}(\O; \R^{3 \times 3})$, a vector-valued function $Q_\e \in L_0^{2}(\O; \R^3)$, and a positive definite symmetric matrix $M_0 \in \R_{sym}^{3 \times 3}$ (called the permeability tensor) such that:
\begin{itemize}
\item $\|W_{\e} \|_{L^{\infty}(\O)} + \e^{\frac{3-\a}{2}} \big(\| \nabla W_{\e} \|_{L^{2}(\O)}  + \| Q_\e \|_{L^{2}(\O)}  \big) \leq C;$
\item $\dive (W_{\e} \vc e_i)  = 0 $ in $\O,$ \  $W_{\e}   = 0$ on the holes $\mathcal{T}_{\e, k}$ for all $k$, and $W_{\e} \to \Id$ strongly in $L^{2}(\O)$;
\item for any $\phi \in C_{c}^{\infty}(\O; \R^3)$, and for any family $\{\gamma_{\e}\}_{\e>0}\subset W^{1,2}(\O;\R^{3})$ satisfying $ \gamma_{\e} = 0$ on the holes $\mathcal{T}_{\e, k}$ for all $k$ and 
$$
\gamma_{\e} \weak \gamma \ \mbox{weakly in} \ L^{2}(\O), \qquad \e^{\frac{3-\a}{2}} \| \nabla \gamma_{\e} \|_{L^{2}(\O)}  \leq C, 
$$
there holds
$$
\e^{3-\a} \langle -\Delta W_{\e} + \nabla Q_\e , \phi \otimes \gamma_{\e} \rangle_{\O} \to \int_{\O} M_{0}\phi \cdot \gamma \dx.
$$

\item For any $q \in [1,\infty]$, we have
\ba \label{estW-I}
\|W_\e - \Id\|_{L^q(\O)} \leq C \begin{cases}
    \e^{\min\{1, \frac{3}{q}\} (\a - 1)} & \text{if } q \neq 3,\\
    \e^{\a - 1} |\log \e|^\frac13 & \text{if } q=3.
\end{cases}
\ea

\item For any $q  \in ( \frac32, \infty]$, we have
\ba\label{nabW}
\| \nabla W_{\e} \|_{L^{q}(\O)}  + \| Q_\e \|_{L^{q}(\O)}  \leq C  \e^{-\a+\frac{3(\a-1)}{q}}.
\ea


\item There holds
\ba \label{We-Qe-eq}
\|\e^{3-\a}(-\Delta W_\e + \nabla Q_\e) - M_0\|_{W^{-1,2}(\O_\e)} \leq C \e.
\ea

\item Let $\varpi \geq \e$ and define $\Omega^\varpi = \{x \in \Omega : {\rm dist}(x, \partial \Omega) < \varpi \}$. Then for any $q \in [1, \infty)$,
\ba \label{W_O_varpi}
    \|\nabla W_\e \|_{L^q(\Omega^\varpi)} &\leq C \varpi^\frac{1}{q} \|\nabla W_\e\|_{L^q(\Omega)}, \\
    \|W_\e - \Id\|_{L^q(\O^\varpi)} &\leq C \varpi^\frac{1}{q} \|W_\e - \Id\|_{L^q(\O)}.
\ea
 \end{itemize}
\end{proposition}

Furthermore, a straightforward generalization of \cite[Lemma~2.3]{Hoefer23} shows the following result:
\begin{lemma} \label{lem:Bogovski.W_eps}
There exists a bounded linear operator $B_\e : W^{1,q}(\O; \R^3) \to W_0^{1,q}(\O_\e; \R^3)$ for all $q \in (1,\infty)$ such that for any solenoidal $\phi \in W^{1,q}(\O; \R^3)$, we have
\ba
\dive B_\e(\phi) = W_\e : \nabla \phi, \nn
\ea
and
\ba 
\e^{-1} \|B_\e(\phi)\|_{L^q(\O_\e)} + \|\nabla B_\e(\phi)\|_{L^q(\O_\e)} \leq C \|(W_\e - \Id) : \nabla \phi\|_{L^q(\O)}. \nn
\ea
\end{lemma}

Thus, for each $\phi\in C_{c}^{\infty}(\O; \R^3)$, the modified function $\phi_\e := W_\e \phi - B_\e(\phi)$ becomes a good test function in the weak formulation of the original non-Newtonian equations in $\O_{\e}$. Note especially that $\phi_\e |_{\d \O_\e} = 0$, $\dive \phi_\e = 0$, and
\begin{align} \label{phi_eps}
\|\phi_\e - \phi\|_{L^q(\O_\e)} \leq C \Big( \|(W_\e - \Id) \phi\|_{L^q(\O_\e)} + \e \|(W_\e - \Id):\nabla \phi\|_{L^q(\O_\e)} \Big) \leq C \e^{\min\{1, \frac3q \}(\a-1)}
\end{align}
for any $q \neq 3$. Note that to deal with $q= \infty$, one can use the interpolation estimate
\begin{align*}
    \| B_\e(\phi) \|_{L^\infty(\Omega_\e)} \leq C\| B_\e(\phi)\|^{\frac 1 2}_{L^6(\Omega_\e)}\| \nabla B_\e(\phi)\|^{\frac 1 2}_{L^6(\Omega_\e)}.
\end{align*}
Moreover, for any $q \in (3/2,\infty)$,
\begin{align} \label{nabla.phi_eps}
\|\nabla \phi_\e \|_{L^q(\O_\e)} \leq C \Big( \|\nabla W_\e\|_{L^q(\O_\e)} + \|W_\e\|_{L^q(\O_\e)} + \|W_\e - \Id\|_{L^q(\O_\e)}\Big) \leq C \e^{-\a + \frac{3(\a-1)}{q}}, 
\end{align}
where we additionally used $W_\e = 0$ in $\O \setminus \O_\e$ and Poincar\'e inequality in $\O$ to see that $\|W_\e\|_{L^q(\O_\e)} \leq C \|\nabla W_\e\|_{L^q(\O_\e)}$. In particular, $\phi_\e \in W_{0, \dive}^{1,q}(\O_\e)$ for any $q \in (3/2,\infty)$.
}
By careful analysis, passing $\e \to 0$ gives the desired homogenized systems. This will be done in the next sections.

\section{Stationary NSE}\label{sec:StatNSE}
\subsection{Qualitative homogenization}

Our first result concerning the homogenization of the steady Navier-Stokes system reads as follows.

\begin{theorem}\label{thm-1}
Let
\begin{align*}
r>1 && 1<\a<\frac 32, && \l > \a.
\end{align*}
Let {\tc $\eta_r : \R_{sym}^{3\times 3} \to [0, \infty)$ be a continuous function satisfying \eqref{monG}--\eqref{growthG},} let $(\uu_\e)$ be a  weak solution of equation \eqref{NonNew}, and recall that $\tilde\uu_\e$ is the zero extension of $\uu_\e$. 
Then $\tilde \uu_\e \weak \uu$ weakly in $L^{2}(\O)$,
where $\uu$ satisfies the Darcy's law
\begin{align}\label{Darcy}
\begin{cases}
\frac{1}{2}\eta_0\uu = M_{0}^{-1}(\mathbf{f}-\nabla p) & \text{in }  \O,\\
\dive \uu = 0 & \text{in } \O,\\
\uu \cdot \mathbf{n} = 0 & \text{on }  \d \O,
\end{cases}
  \end{align}
with {\tc the  permeability tensor  $M_{0}$ as in Proposition~\ref{lem-local-1}. Additionally, if $g_r \equiv 0$, then the above conclusion holds true for any $1 < \a < 3$.}
\end{theorem}

%
{\tc
\begin{remark}
Thanks to $\uu \in L^2(\O)$ and $\ff \in L^2(\O)$, we also have $\nabla p \in L^2(\O)$. In turn, the system \eqref{Darcy} has to be understood in weak sense rather than a distributional one. The boundary condition on $\uu$ is well-defined due to the existence of a normal trace for divergence-free function, see \cite[Chapter~10.3]{F-N-book}, and is of course redundant if $\Omega = \mathbb{T}^3$.
\end{remark}
}

 The next two subsections are devoted to the proof of Theorem~\ref{thm-1}. 

\subsubsection{Uniform bounds}\label{sec:statBds}
It follows from the energy identity \eqref{energy-ineq},  the Poincar\'e inequality, and the Korn inequality (see Lemmas \ref{lem-Poincare} and \ref{lem-Korn}) that {\tc
\ba
\e^{3-\a} \int_{\O_\e} \eta_r(|\e^{3-\a} D \uu_\e|) | D\uu_\e |^2 \dx \leq \int_{\O_\e}\mathbf{f}\cdot \uu_\e  \dx   \leq  C \e^{\frac{3-\a}{2}} \|D\uu_\e \|_{L^2(\O_\e)} \| \ff \|_{L^2(\O_\e)}. \nn
\ea
Hence, by the coercivity in \eqref{monG},
\ba
\e^{3-\a} \|D\uu_\e \|_{L^2(\O_\e)}^{2} + \e^{(3-\a)(r-1)} \|D\uu_\e \|_{L^{r}(\O_\e)}^{r} \mathbf{1}_{r>2}  \leq  C. \nn
\ea

Consequently, the zero extension $\tilde \uu_{\e}$ of $\uu_{\e}$ satisfies
\ba\label{est-u-f}
\e^{\frac{3-\a}{2}} \|\nabla \tilde \uu_{\e}\|_{L^2(\O)}  \leq C, \qquad \| \tilde \uu_{\e}\|_{L^2(\O )} \leq C , \qquad \e^{\frac{(3-\a)(r-1)}{r}} \|\nabla \tilde \uu_{\e}\|_{L^{r}(\O)} \mathbf{1}_{r>2} \leq  C ,
\ea
}
and there exists $\uu \in L^2(\O)$ such that, up to a subsequence,
\ba\label{est-u-f2}
  \tilde  \uu_{\e} \weak \uu \ \mbox{weakly in} \ L^{2}(\O), \qquad \dive \uu = 0  \ \mbox{in} \ \O, \qquad \uu\cdot \vc n =0 \ \mbox{on} \ \d \O.
\ea

\subsubsection{\tc Proof of Theorem \ref{thm-1}}\label{sec:statQual}
To begin, we remark that it is enough to concentrate on divergence-free test functions. This is because the pressure $p_\e$ is uniquely determined on the complement of this space; see \cite[Lemma II.2.2.2]{Sohr01} for details.
\\

Given any solenoidal vector-valued function $\phi\in C_c^\infty(\O; \R^3)$, {\tc as mentioned after Lemma \ref{lem:Bogovski.W_eps}, we may take $\phi_\e = W_\e \phi - B_\e(\phi) \in W_{0, \dive}^{1,q}(\O_\e), \ q>\frac32,$ as a test function in the weak formulation of the momentum equation to infer}
\ba\label{limit-0}
 \int_{\O_{\e}} - \e^\l \uu_{\e}\otimes \uu_{\e}:\nabla \phi_\e + \e^{3-\a} \eta_r(\e^{3-\a} D\uu_\e )D\uu_{\e} : D \phi_\e \dx = \int_{\O_\e}\ff \cdot \phi_{\e} \dx.
\ea
Since $\phi_{\e}$ vanishes on  the holes and the extension $\tilde \uu_{\e}$ coincides with $\uu_{\e}$ in $\O_{\e}$,  we can write \eqref{limit-0} as
\ba
 \int_{\O} - \e^\l \tilde  \uu_{\e}\otimes \tilde \uu_{\e}:\nabla \phi_{\e} + \e^{3-\a} \eta_r(\e^{3-\a} D\tilde \uu_\e)D \tilde \uu _{\e} : D \phi_{\e} \dx = \int_{\O}\mathbf{f} \cdot \phi_{\e} \dx. \nn
\ea

{\tc We will focus on each of the terms separately.}

\medskip

{\tc First, using the uniform estimates in \eqref{est-u-f}, we have by interpolation and Sobolev embedding that
\ba
\|\uu_\e \otimes \uu_{\e}\|_{L^{q_{1}}(\O_\e)} \leq  \|\uu_\e \|_{L^{2q_1}(\O_\e)}^2  \leq \| \uu_\e\|_{L^2(\O_\e)}^{2(1-\th)}  \| \uu_\e\|_{L^6(\O_\e)}^{2\th} \leq C \| \uu_\e\|_{L^2(\O_\e)}^{2(1-\th)}  \| \nabla \uu_\e\|_{L^2(\O_\e)}^{2\th} \leq C \e^{-\th(3-\a)}, \nn
\ea
where 
\ba
1 \leq   q_{1} \leq   3, \qquad 0 \leq  \th \leq  1 , \qquad \frac{1}{q_{1}} = (1-\th) + \frac{\th}{3} = 1 - \frac{2\th}{3}. 
\nn\ea
 Therefore, it follows from \eqref{est-u-f} and \eqref{nabla.phi_eps}, together with the assumption $\l > \a$, that
\ba\label{limit-2}
 \int_{\O} - \e^\l \tilde \uu_{\e}\otimes \tilde \uu_{\e}:\nabla \phi_{\e} \dx &\leq C \e^\l \| \uu_{\e}\|_{L^{\frac{6}{3-2\theta}}(\O_\e)}^{2} \|\nabla \phi_{\e}\|_{L^{\frac{3}{2\th}}(\O_\e)} \leq C \e^{\l - \a - \th(5 - 3\a)} \to 0,
\ea
}
where we choose $\th>0$ suitably small. Furthermore,
\ba
\int_{\O}\mathbf{f}\cdot \phi_{\e} \dx \to \int_{\O}\mathbf{f}\cdot \phi \dx, \nn
\ea
which follows from \eqref{phi_eps} and $\alpha > 1$.

\medskip

We finally consider 
\ba\label{limit-6}
&\e^{3-\a} \int_{\O}  \eta_r(\e^{3-\a} D\tilde \uu_\e) D  \tilde \uu _{\e} : D \phi_{\e} \dx = \e^{3-\a} \int_{\O} \eta_0 D \tilde \uu _{\e} : D \phi_{\e} \dx +  \e^{3-\a}  \int_{\O} g_r(|\e^{3-\a} D \tilde \uu_\e|) D  \tilde \uu _{\e} : D \phi_{\e} \dx \\
&= \e^{3-\a} \int_{\O} \eta_0 D \tilde \uu _{\e} : D (W_\e \phi) \dx - \e^{3-\a} \int_{\O} \eta_0 D \tilde \uu _{\e} : D B_\e(\phi) \dx +  \e^{3-\a}  \int_{\O} g_r(|\e^{3-\a} D \tilde \uu_\e|) D  \tilde \uu _{\e} : D \phi_{\e} \dx.
\ea
For the first term on the right side of \eqref{limit-6} we have by $\dive \tilde \uu_\e = 0$ (see also \cite{All90-1})
\ba\label{limit-7}
\e^{3-\a} \int_{\O} \eta_0 D \tilde \uu _{\e} : D (W_{\e} \phi)\dx  & = \frac{\eta_0}{2} \e^{3-\a} \int_{\O}  \nabla  \tilde \uu _{\e} : \nabla (W_{\e} \phi)\dx  =  \frac{\eta_0}{2} \e^{3-\a} \langle  -\Delta (W_{\e} \phi) + \nabla (Q_\e \cdot \phi),  \tilde \uu _{\e} \rangle_\Omega \\
& = \frac{\eta_0}{2} \e^{3-\a} \langle  (-\Delta W_{\e} + \nabla Q_\e) \phi + \vc z_\e ,  \tilde \uu _{\e} \rangle_\Omega  \\
& = \frac{\eta_0}{2} \e^{3-\a} \langle  (-\Delta W_{\e} + \nabla Q_\e) \phi, \tilde \uu_\e \rangle_\Omega + \frac{\eta_0}{2} \e^{3-\a} \int_{\O}   \tilde \uu _{\e} \cdot \vc z_\e \dx,
\ea
where we set
\ba
\vc z_\e &= (\Delta W_\e) \phi - \Delta(W_\e \phi) + (Q_\e \cdot \nabla) \phi, \\
\|\vc z_\e\|_{L^2(\O)} &\leq ( \| W_\e\|_{W^{1,2}(\O)} + \|Q_\e\|_{L^2(\O)} ) \|\phi\|_{W^{2,\infty}(\O)} \leq C \e^{\frac{\a - 3}{2}}. \nn
\ea
It follows from \eqref{est-u-f}--\eqref{est-u-f2} and {\tc Proposition}~\ref{lem-local-1} that 
\ba\label{limit-8}
&  \e^{3-\a} \int_{\O}   \tilde \uu _{\e} \cdot \vc z_\e \dx    \leq C \e^{\frac{3-\a}{2}} \to 0, \\
 & - \frac{\eta_0}{2} \e^{3-\a} \langle  (\Delta W_{\e} -\nabla Q_\e) \phi, \tilde \uu_\e \rangle_\Omega  = \frac{\eta_0}{2} \e^{3-\a} \langle -\Delta W_{\e} +  \nabla Q_\e , \phi \otimes \tilde \uu_{\e} \rangle_{\O}   \to \frac{\eta_0}{2} \int_{\O} M_{0} \phi \cdot \uu \dx .
\ea
As long as there holds
{\tc
\ba\label{limit-9}
 \e^{3-\a} \int_{\O} g_r(|\e^{3-\a} \tilde D \uu_\e|) D  \tilde \uu _{\e} : D (\phi_\e)\dx &\to 0, \\
\e^{3-\a} \int_\O \eta_0 D \tilde \uu_\e : D B_\e(\phi) \dx &\to 0,
\ea
}
we deduce 
\ba
 \frac{\eta_0}{2} \int_{\O} M_{0} \phi \cdot \uu \dx =\int_{\O} \mathbf{f} \cdot \phi \dx. \nn
\ea
{\tc Since $\phi$ is solenoidal and $M_0$ is symmetric positive definite, we can deduce from \cite[Lemma II.2.2.2]{Sohr01} that there exists a pressure\footnote{As mentioned before, the fact that indeed $p = \lim_{\e \to 0} \tilde p_\e$ will be shown in Section~\ref{sec:pressure}.} $p$ such that the above equation is exactly the Darcy's law \eqref{Darcy}.} It is left to show \eqref{limit-9}. 

\medskip

{\tc We start with the second integral in \eqref{limit-9}. Given the bounds on $\tilde \uu_\e$ and $\nabla B_\e(\phi)$ (see Lemma~\ref{lem:Bogovski.W_eps}), we find
\ba 
\e^{3-\a} \int_\O \eta_0 D \tilde \uu_\e : D B_\e(\phi) \dx &\leq C \e^{3-\a} \|\nabla \tilde \uu_\e\|_{L^2(\O)} \|\nabla B_\e(\phi)\|_{L^2(\O)} \\
&\leq C \e^\frac{3-\a}{2} \|(W_\e - \Id) : \nabla \phi\|_{L^2(\O)} \leq C \e^\frac{3-\a}{2} \|\nabla \phi\|_{L^2(\O)} \to 0. \nn
\ea 
}

{\tc Let us come to the first integral in \eqref{limit-9}. First, if $g_r \equiv 0$ being the case for Newtonian fluids, we finish the proof already at this stage and Theorem~\ref{thm-1} is proven for any $1 < \a < 3$. We will therefore focus on the case $g_r \not \equiv 0$ in the sequel, where we additionally assume $1 < \a < \frac32$.}\\

By the growth condition \eqref{growthG}, we {\tc see in particular that for $1<r\leq 2$, we have
\begin{align}\label{G1}
    |g_r(|\e^{3-\a} D \tilde \uu_\e|) | \leq C |\e^{3-\a} D \tilde \uu_\e|^\kappa, \qquad \forall \kappa \in (0,1).
\end{align}
}
In turn, using \eqref{est-u-f} and {\tc \eqref{nabla.phi_eps} for $1<r \leq 2$ gives
\ba\label{limit-11}
& \e^{3-\a} \Big|  \int_{\O} g_r(|\e^{3-\a} D  \tilde \uu _{\e}|) D \tilde \uu_\e : D \phi_{\e} \dx \Big|\\
  &\leq C \e^{(3-\a)(1+\kappa)}  \int_{\O}|  D  \tilde \uu _{\e} |^{1+\kappa} |D \phi_{\e} | \dx  \\
  &\leq  C \e^{(3-\a)(1+\kappa)} \| | \nabla  \tilde \uu _{\e} |^{1+\kappa} \|_{L^{\frac{2}{1+\kappa}}(\O)}  \| \phi_{\e} \|_{W^{1, \frac{2}{1-\kappa}}(\O)}  \\
  & \leq C \e^{(3-\a)(1+\kappa)} \e^{-\frac{(3-\a)}{2} (1+\kappa)} \e^{ -\a + 3(\a - 1) \frac{1-\kappa}{2}}   =  C \e^{(3 - 2\a)\kappa} \to 0,
\ea
}
under the assumption $1<\a<\frac{3}{2}$.

For {\tc $r > 2$}, using  \eqref{est-u-f} and {\tc Proposition~\ref{lem-local-1}, together with
\begin{align*}
    |g_r(|\e^{3-\a} D \tilde \uu_\e |)| \leq C \Big( |\e^{3-\a} D \tilde \uu_\e|^\kappa + |\e^{3-\a} D\tilde \uu_\e|^{r-2} \Big) \qquad \forall \kappa \in (0,1),
\end{align*}
yields
\ba\label{limit-12}
& \e^{3-\a} \Big|  \int_{\O} g_r(|\e^{3-\a} D\tilde \uu_\e|) D \tilde \uu _{\e} : D \phi_{\e} \dx \Big|\\
  &\leq C \int_{\O} \e^{(3-\a)(1+\kappa)} |  D  \tilde \uu _{\e} |^{1+\kappa} : |D \phi_{\e} | \dx + C \int_{\O} \e^{(3-\a)(r-1)} |  D  \tilde \uu _{\e} |^{r-1} : |D \phi_{\e} | \dx  \\
  &\leq  C \e^{(3-2\a)\kappa} + C \e^{(3-\a)(r-1)} \| | \nabla  \tilde \uu _{\e} |^{r-1} \|_{L^{\frac{r}{r-1}}(\O)}  \| \phi_{\e} \|_{W^{1, r}(\O)}  \\
  & \leq  C \e^{(3-2\a)\kappa} + C \e^{(3-\a)(r-1)} \e^{-(3-\a) \frac{(r-1)^2}{r}} \e^{ -\a + \frac{3(\a - 1)}{r}}   =   C \e^{(3-2\a)\kappa} + C \e^{\frac{(3 - 2\a)(r-2)}{r}} \to 0,
\ea
under the assumption $1<\a<\frac{3}{2}$. All in all, \eqref{limit-9} is shown, which finishes the proof of Theorem~\ref{thm-1}.
}

\subsection{Quantitative homogenization on the torus}\label{sec:statQuan}
In this section, we will give explicit convergence rates {\tc for the torus case}. The torus has the advantage that we do not need to deal with the mismatch in the outer boundary conditions for a bounded domain, namely $\uu_\e|_{\d \O} = 0$ versus $\uu \cdot \vc n|_{\d \O} = 0$. We will come back to this issue in Section~\ref{sec:statBdDom}.\\
\begin{theorem}\label{thm-4}
Let {\tc $\O = \mathbb{T}^3$ and }
\begin{align*}
r>1, && 1<\a<\frac 32, && \l > \a.
\end{align*}
Let {\tc  $\eta_r$ comply with \eqref{monG}--\eqref{growthG},} let $(\uu_\e,p_\e)$ be a weak solution to \eqref{NonNew}, and let $(\uu, p) \in   [W^{1,\infty}(\O) \cap W^{2,2}(\O) ] \times W^{1,\infty}(\O)$ with $\dive \uu = 0$ be a strong solution to Darcy's law \eqref{Darcy}. Then, there exists an $\e_0 > 0$ such that for all $\e \in (0, \e_0)$, we have
{\tc
\ba\label{ConvSpeedS}
\|\tilde{\uu}_\e - \uu \|_{L^2(\O)}^2 & \leq C ( \e^{\a-1} + \e^{3-\a} + \e^{2(3-2\a)} ),
\ea
}
where the constant $C>0$ is independent of $\e$. {\tc If $g_r \equiv 0$, the we may choose any $1<\a<3$, and the last term in \eqref{ConvSpeedS} can be dropped.}
\end{theorem}

\begin{remark}
\begin{itemize}
\item To the best of our knowledge, this is the first result concerning the convergence rates in homogenization of non-Newtonian fluids.
\item Compared to \cite[Theorem~3.4.14]{All90-2}, we have the same convergence rates since in there, {\tc $g_r \equiv 0$} (hence the last term in \eqref{ConvSpeedS} vanishes). {\tc Note that \cite[Theorem~3.4.14]{All90-2} only applies to the torus case because boundary conditions are not taken into account; we will come back to this in Section~\ref{sec:statBdDom}.}
\end{itemize}
\end{remark}

\subsubsection{Relative energy identity}
From \eqref{eq-weak}, we have for any {\tc solenoidal} $\vU\in C^{\infty}(\O_{\e};\R^{3})$ satisfying $\vU|_{\d \O_\e} = 0$, 
\ba
\int_{\O_{\e}} - \e^\l \uu_{\e}\otimes \uu_{\e}:\nabla\vU  + \e^{3-\a} \eta_r(\e^{3-\a} D\uu_\e )D\uu_{\e} : D\vU \dx =  \int_{\O_\e}\mathbf{f}\cdot\vU  \dx. \nn
\ea

By the energy identity \eqref{energy-ineq}, we get
\ba
&\int_{\O_{\e}} \e^{3-\a} [ \eta_r(\e^{3-\a} D\uu_\e )D\uu_{\e} -  \eta_r(\e^{3-\a} D\vU  )D\vU  ] : (D\uu_\e - D\vU ) \dx \\
&= \int_{\O_\e} \e^{3-\a} \eta_r(\e^{3-\a} D \vU ) D \vU  : (D \vU  - D \uu_\e) \dx \\
&\qquad - \int_{\O_\e} \e^\l \uu_{\e}\otimes \uu_{\e}:\nabla\vU  \dx + \int_{\O_\e}\mathbf{f}\cdot (\uu_\e - \vU ) \dx. \nn
\ea
Using moreover $\dive \uu_\e = 0$ and $\uu_\e |_{\d \O_\e} = 0$, we may write
\ba
-\int_{\O_\e} \e^\l \uu_\e \otimes \uu_\e : \nabla \vU  \dx = -\int_{\O_\e} \e^\l ((\uu_\e \cdot \nabla) \vU ) \cdot (\uu_\e - \vU ) \dx \nn
\ea
to get the final {\tc relative energy identity (REI)}
\ba \label{relEn}
&\int_{\O_{\e}} \e^{3-\a} [ \eta_r(\e^{3-\a} D\uu_\e )D\uu_{\e} -  \eta_r(\e^{3-\a} D\vU  )D\vU  ] : D(\uu_\e - \vU ) \dx \\
&= \int_{\O_\e} \e^{3-\a} \eta_r(\e^{3-\a} D \vU ) D \vU  : (D \vU  - D \uu_\e) \dx - \e^\l \int_{\O_\e} (\uu_\e - \vU ) \cdot ( (\uu_\e \cdot \nabla) \vU  ) \dx \\
&\qquad + \int_{\O_\e}\mathbf{f}\cdot (\uu_\e - \vU ) \dx .
\ea

Note that inserting $\vU =0$ in the above yields the standard energy identity \eqref{energy-ineq}. Moreover, by density, the REI \eqref{relEn} holds for any {\tc $\vU \in W^{1,\max\{2,r\}}_{0, \rm{div}}  (\O_\e) $}.

\subsubsection{Proof of Theorem \ref{thm-4}}\label{sec:statQuan1}
Let $\uu$ be a strong solution to Darcy's law \eqref{Darcy} as in Theorem \ref{thm-4}. Using the operator $B_\e$ from Lemma~\ref{lem:Bogovski.W_eps}, we define {\tc $\vc w_\e = W_\e \uu - B_\e(\uu)$}.  Then {\tc for any $q>\frac32$, we have $\vc w_\e \in W_{0, \dive}^{1,q}(\O_\e)$ and}, by \eqref{phi_eps} and \eqref{nabla.phi_eps}, there holds  
\ba \label{est.w_e-stat}
&\|\vc w_\e\|_{L^\infty(\O)}  + {\tc \e^{\a - \frac{3(\a-1)}{q}} \|\nabla \vc w_\e\|_{L^q( \O)}} \leq C, \ {\tc q \in  (3/2, \infty)},\\
& \|\vc w_\e - \uu\|_{L^q(\O_\e)} \leq C \e^{\min\{1, \frac{3}{q} \} (\a - 1)},  \ \forall \, q \geq  1, \ {\tc q \neq 3}, \qquad \vc w_\e |_{\d \O_\e} = 0.
\ea

We choose $\vU = \vc w_\e$ in \eqref{relEn} to get
\begin{align*}
\begin{split}
&\e^{3-\a} \int_{\O_\e} [ \eta_r(\e^{3-\a} D \uu_\e) D\uu_\e - \eta_r(\e^{3-\a} D \vc w_\e) D\vc w_\e] : \nabla (\uu_\e - \vc w_\e) \dx + \e^\l \int_{\O_\e} (\uu_\e \cdot \nabla) \vc w_\e \cdot (\uu_\e - \vc w_\e) \dx \\
&= \e^{3-\a} \int_{\O_\e} \eta_r(\e^{3-\a} D \vc w_\e) D\vc w_\e : \nabla (\vc w_\e - \uu_\e) \dx + \int_{\O_\e} \vc f \cdot (\uu_\e - \vc w_\e) \dx.
\end{split}
\end{align*}

Next, we replace $\vc f$ by the Darcy's law \eqref{Darcy} to get
\begin{align*}
\int_{\O_\e} \vc f \cdot (\uu_\e - \vc w_\e) \dx = \int_{\O_\e} \Big( \frac{\eta_0}{2} M_0 \uu + \nabla p \Big) \cdot (\uu_\e - \vc w_\e) \dx.
\end{align*}
{\tc Using solenoidality of $\uu_\e$ and $\vc w_\e$, we see
\ba
\int_{\O_\e} \nabla p \cdot (\uu_\e - \vc w_\e) \dx = 0.
\nn
\ea

Hence, the REI takes the form
\ba\label{rel-en-1}
&\int_{\O_{\e}} \e^{3-\a} [ \eta_r(\e^{3-\a} D\uu_\e )D\uu_{\e} -  \eta_r(\e^{3-\a} D\vc w_\e )D\vc w_\e ] : D(\uu_\e - \vc w_\e ) \dx \\
&=\int_{\O_\e} \e^{3-\a} \eta_r(\e^{3-\a} D \vc w_\e) D\vc w_\e : (D \vc w_\e - D \uu_\e ) \dx \\
& \qquad - \e^\l \int_{\O_\e} (\uu_\e - \vc w_\e) \cdot ( (\uu_\e \cdot \nabla) \vc w_\e) \dx  + \int_{\O_\e} \frac{\eta_0}{2} M_0 \uu \cdot (\uu_\e - \vc w_\e) \dx.
\ea
}
To conclude, we will estimate all terms separately.\\

{\tc The dissipation, i.e., the first left-hand side term, is bounded below due to \eqref{monG} and Korn's inequality \eqref{Korn} by
\begin{align} \label{dissipation}
    &\int_{\O_{\e}} \e^{3-\a} [ \eta_r(\e^{3-\a} D\uu_\e )D\uu_{\e} -  \eta_r(\e^{3-\a} D\vc w_\e )D\vc w_\e ] : D(\uu_\e - \vc w_\e ) \dx \\
    &\geq c \e^{3-\a}  \|\nabla (\uu_\e - \vc w_\e )\|_{L^2(\O_\e)}^2 + c \mathbf{1}_{r>2} \e^{(3-\a)(r-1)}  \|\nabla (\uu_\e - \vc w_\e )\|_{L^r(\O_\e)}^r. 
\end{align}
}

The second term on the right-hand side of \eqref{rel-en-1} may be split as
\ba
&\e^\l \int_{\O_\e} (\uu_\e - \vc w_\e) \cdot ( (\uu_\e \cdot \nabla) \vc w_\e) \dx \\
&= \e^\l \int_{\O_\e} (\uu_\e - \vc w_\e) \cdot ( (\vc w_\e \cdot \nabla) \vc w_\e) \dx + \e^\l \int_{\O_\e} (\uu_\e - \vc w_\e) \cdot ((\uu_\e - \vc w_\e) \cdot \nabla) \vc w_\e \dx\\
&\leq C \e^\l \|\nabla \vc w_\e\|_{L^2(\O_\e)} \|\uu_\e - \vc w_\e\|_{L^2(\O_\e)} {\tc \|\vc w_\e\|_{L^\infty(\O_\e)}} + C \e^\l {\tc \|\nabla (W_\eps \uu)\|_{L^\infty(\O_\e)} } \|\uu_\e - \vc w_\e\|_{L^2(\O_\e)}^2 \\
& \qquad {\tc + \|\nabla B_\eps(\uu)\|_{L^3(\O_\e)} \|\uu_\e - \vc w_\e\|_{L^2(\O_\e)} \|\nabla(\uu_\e - \vc w_\e)\|_{L^2(\O_\e)}} \\
& \leq C \e^\l \|\nabla (\uu_\e - \vc w_\e)\|_{L^2(\O_\e)} + C {\tc (\e^{\l+3 - 2\a} + \e^{\l + \a - 1 + \frac{3-\a}2} |\log \e|^{\frac 1 3})} \|\nabla (\uu_\e - \vc w_\e)\|_{L^2(\O_\e)}^2 \\
& \leq C_\delta \e^{2\l - (3-\a)} + (\delta \e^{3-\a} + C  \e^{\l+3 - 2\a}) \|\nabla (\uu_\e - \vc w_\e)\|_{L^2(\O_\e)}^2,
\nn \ea
{\tc due to the estimates on $W_\e$ and $B_\e$ from Proposition \ref{lem-local-1} and Lemma \ref{lem:Bogovski.W_eps}, the Poincar\'e inequality in Lemma \ref{lem-Poincare}, and where we used $\alpha > 1$ in the last estimate.}
Note that $\l + 3 - 2 \a = (\l - \a) + (3 - \a)$, so by $\l >  \a $ and \eqref{dissipation}, we can absorb the last term by the left-hand side of \eqref{rel-en-1} for $\e$ and $\delta$ small enough. Further, again due to $\l>\a$, we have {\tc $\a-1 \leq 2\l - (3 - \a)$}. Thus, for $\e>0$ {\tc and $\delta > 0$} small enough,
\ba \label{relE1-stat}
&c \e^{3-\a}  \|\nabla (\uu_\e - \vc w_\e )\|_{L^2(\O_\e)}^2 + c \mathbf{1}_{r>2} \e^{(3-\a)(r-1)}  \|\nabla (\uu_\e - \vc w_\e )\|_{L^r(\O_\e)}^r  \\
&\leq \int_{\O_\e} \e^{3-\a} \eta_r(\e^{3-\a} D \vc w_\e) D\vc w_\e : (D \vc w_\e - D \uu_\e ) \dx + \int_{\O_\e} \frac{\eta_0}{2} M_0 \uu \cdot (\uu_\e - \uu) \dx + C \e^{\a-1}. \ea

Using the definition of $\eta_r$, we split
\ba\label{eta-split-1}
&\int_{\O_\e} \e^{3-\a} \eta_r(\e^{3-\a} D \vc w_\e) D\vc w_\e : D( \vc w_\e - \uu_\e) \dx \\
&= \e^{3-\a} \int_{\O_\e} \eta_0 D \vc w_\e : D(\vc w_\e - \uu_\e) \dx + \e^{3-\a} \int_{\O_\e} g_r(|\e^{3-\a} D \vc w_\e|) D \vc w_\e : D(\vc w_\e - \uu_\e) \dx \\
&= \e^{3-\a} \int_{\O_\e} \frac{\eta_0}{2} (- \Delta (W_\e \uu) + \nabla (Q_\e \cdot \uu)) \cdot (\vc w_\e - \uu_\e) \dx - \e^{3-\a} \int_{\O_\e} {\tc \frac{\eta_0}{2} } \nabla (Q_\e \cdot \uu) \cdot (\vc w_\e - \uu_\e) \dx \\
&\qquad + \e^{3-\a} \int_{\O_\e} g_r(|\e^{3-\a} D \vc w_\e|) D \vc w_\e : D(\vc w_\e - \uu_\e) \dx {\tc - \e^{3-\a} \int_{\O_\e} \eta_0 D B_\e(\uu) : D(\vc w_\e - \uu_\e) \dx}. 
\ea

As for the first integral on the right-hand side, we see
\ba
&- \Delta (W_\e \uu) + \nabla (Q_\e \cdot \uu) = (-\Delta W_\e + \nabla Q_\e) \uu + \vc z_\e,\\
&\vc z_\e = (\Delta W_\e) \uu  - \Delta (W_\e \uu) + (Q_\e \cdot \nabla) \uu , \qquad \|\vc z_\e\|_{L^2(\O_\e)} \leq C \e^\frac{\a - 3}{2}. \nn
\ea

From {\tc Proposition~\ref{lem-local-1}}, we know
\ba
\|\e^{3-\a}( -\Delta W_\e + \nabla Q_\e) - M_0\|_{W^{-1,2}(\O_\e)} \leq C \e.
\nn\ea
Hence, we may write
\ba
&\e^{3-\a} \int_{\O_\e} \frac{\eta_0}{2} (- \Delta (W_\e \uu) + \nabla (Q_\e \cdot \uu)) \cdot (\vc w_\e - \uu_\e) \dx \\
&= \int_{\O_\e} \frac{\eta_0}{2} [\e^{3-\a}(-\Delta W_\e + \nabla Q_\e) - M_0] \uu \cdot (\vc w_\e - \uu_\e) \dx + \int_{\O_\e} \frac{\eta_0}{2} M_0 \uu \cdot (\vc w_\e - \uu_\e) \dx \\
&\qquad + {\tc \e^{3-\a} } \int_{\O_\e} \vc z_\e \cdot (\vc w_\e - \uu_\e) \dx
\nn\ea
with
{\tc \ba
&\e^{3-\a} \int_{\O_\e} \vc z_\e \cdot (\vc w_\e - \uu_\e) \dx \leq \e^{3-\a} \|\vc z_\e\|_{L^2(\O_\e)} \|\vc w_\e - \uu_\e\|_{L^2(\O_\e)} \\
&\leq C \e^{3-\a} \|\nabla(\vc w_\e - \uu_\e)\|_{L^2(\O_\e)} \leq C_\delta \e^{3-\a} + \delta \e^{3-\a} \|\nabla(\vc w_\e - \uu_\e)\|_{L^2(\O_\e)}^2, \nn
\ea
and
\ba
&\int_{\O_\e} \frac{\eta_0}{2} [\e^{3-\a}(-\Delta W_\e + \nabla Q_\e) - M_0] \uu \cdot (\vc w_\e - \uu_\e) \dx \leq C \e \|\vc w_\e - \uu_\e\|_{W^{1,2}(\O_\e)} \\
&\leq C \e \|\nabla(\vc w_\e - \uu_\e)\|_{L^2(\O_\e)} \leq C_\delta \e^{\a-1} + \delta \e^{3-\a} \|\nabla(\vc w_\e - \uu_\e)\|_{L^2(\O_\e)}^2,
\nn \ea
where the last term in either of the above two estimates can be absorbed by the  dissipation for $\delta>0$ small enough.} {\tc Additionally,}
\ba
    \e^{3-\a} \int_{\O_\e} D B_\e(\uu) : D(\vc w_\e - \uu_\e) \dx &\leq C \e^{3-\a} \|(W_\e - \Id):\nabla \uu\|_{L^2(\O_\e)} \|\nabla (\vc w_\e - \uu_\e)\|_{L^2(\O_\e)} \\
    &\leq \delta \e^{3-\a} \|\nabla(\vc w_\e - \uu_\e)\|_{L^2(\O_\e)}^2 + C_\delta \e^{3-\a + 2(\a-1)}. \nn
\ea

{\tc Moreover, the solenoidality of $\vc w_\e$ and $\vu_\e$ leads to }
\ba
- \e^{3-\a} \int_{\O_\e} {\tc \frac{\eta_0}{2} } \nabla (Q_\e \cdot \uu) \cdot (\vc w_\e - \uu_\e) \dx = 0.
\nn
\ea

Hence, we deduce {\tc from \eqref{relE1-stat} that}
\ba
&c \e^{3-\a}  \|\nabla (\uu_\e - \vc w_\e )\|_{L^2(\O_\e)}^2 + c \mathbf{1}_{r>2} \e^{(3-\a)(r-1)}  \|\nabla (\uu_\e - \vc w_\e )\|_{L^r(\O_\e)}^r \\
&\leq C \e^{3-\a} \int_{\O_\e} g_r(|\e^{3-\a}D \vc w_\e|) D \vc w_\e : D(\vc w_\e - \uu_\e) \dx + C \Big( \e^{\a-1} + \e^{3-\a} \Big). 
 \nn \ea

{\tc If $g_r \equiv 0$, we finish the proof of Theorem~\ref{thm-4} already at this stage and we can treat all $1<\a<3$. Therefore, we focus in the sequel on $g_r \not \equiv 0$ and $1<\a<\frac32$.\\

Using \eqref{growthG}, we first see that
\ba
|g_r(s)| &\leq C s \mathbf{1}_{s \leq 1} + C s^{r-2} \mathbf{1}_{r>3} \mathbf{1}_{s \geq 1} + C s^{\max\{r-2, 0\}} \mathbf{1}_{r \leq 3} \mathbf{1}_{s \geq 1} \\
&\leq C s + C s^{r-2} \mathbf{1}_{r>3}.
\nn
\ea

Combining the above estimate with \eqref{est.w_e-stat}, we find
\begin{align*}
&\e^{3-\a} \int_{\O_\e} g_r(|\e^{3-\a} D \vc w_\e|) D \vc w_\e : D(\vc w_\e - \uu_\e) \dx \\
&\leq C \e^{2(3-\a)} \int_{\O_\e} |D\vc w_\e|^2 |D(\uu_\e - \vc w_\e)| \dx \\
&\qquad + C \e^{(3-\a)(r-1)} \mathbf{1}_{r > 3} \int_{\O_\e} |D \vc w_\e|^{r-1} \, |D(\uu_\e - \vc w_\e)| \dx \\
&\leq C \e^{2(3-\a)} \|D\vc w_\e\|_{L^4(\O_\e)}^2 \|D(\uu_\e - \vc w_\e)\|_{L^2(\O_\e)} \\
&\qquad + C \e^{(3-\a)(r-1)} \| D\vc w_\e \|_{L^{2(r-1)}(\O_\e)}^{r-1} \|D(\uu_\e - \vc w_\e)\|_{L^2(\O_\e)} \mathbf{1}_{r>3} \\
& \leq \delta \e^{3 - \a} \|D(\vc w_\e - \uu_\e)\|_{L^2(\O_\e)}^2 + C_\delta \e^{3(3 - \a)} \|D \vc w_\e\|^4_{L^4(\O_\e)} \\
& \qquad +  C_\delta \e^{(2r - 3)(3 - \a)} \|D \vc w_\e\|^{2(r-1)}_{L^{2(r-1)}(\O_\e)} \mathbf{1}_{r>3} \\
& \leq \delta \e^{3 - \a} \|D(\vc w_\e - \uu_\e)\|_{L^2(\O_\e)}^2 + C_\delta \e^{2(3 - 2\a)} + C_\delta \e^{2(3-2\a)(r-2)} \mathbf{1}_{r>3} \\
&\leq \delta \e^{3-\a} \|D(\vc w_\e - \uu_\e)\|_{L^2(\O_\e)}^2 + C_\delta \e^{2(3-2\a)},
\end{align*}
where we used that $r>3 \Leftrightarrow r-2 > 1$.
}

In turn, choosing $\delta>0$ small enough, we arrive at
\ba
&c \e^{3-\a}  \|\nabla (\uu_\e - \vc w_\e )\|_{L^2(\O_\e)}^2 + c \mathbf{1}_{r>2} \e^{(3-\a)(r-1)}  \|\nabla (\uu_\e - \vc w_\e )\|_{L^r(\O_\e)}^r \\
&\leq C \Big( \e^{\a-1} + \e^{3-\a} + {\tc \e^{2(3-2\a)} } \Big)  .
\nn \ea

{\tc Hence, to get the final inequality \eqref{ConvSpeedS}, it is enough to see that
\begin{align*}
&\|\tilde{\uu}_\e - \uu \|_{L^2(\O)}^2 \leq C \big( \|\uu_\e - \vc w_\e\|_{L^2(\O_\e)}^2 + \|\vc w_\e - \uu \|_{L^2(\O)}^2 \big) \\
&\leq C \e^{3-\a} \|\nabla( \uu_\e - \vc w_\e)\|_{L^2(\O_\e)}^2 + C \e^{\a-1} \\
&\leq C \Big(  \e^{\a-1} + \e^{3-\a} + \e^{2(3-2\a)} \Big).
\end{align*}
}
This ends the proof of Theorem~\ref{thm-4}.

{\tc
\begin{remark}
    As can be seen from the above proof, the absorption of dissipative terms only needs $L^2$-estimates. In particular, we may treat viscosities that fulfill \eqref{monG} in the weaker form
    \begin{align*}
        [\eta_r(\beta A)A - \eta_r(\beta B)B] : (A-B) \geq C^{-1} |A-B|^2.
    \end{align*}
\end{remark}
}

\subsection{\tc Adaptions for bounded domains}\label{sec:statBdDom}
\providecommand{\twe}{\tilde{\vc w}_\e}
For a bounded domain $\Omega \subset \R^3$, using $\vc w_\e = W_\e \uu$ as test function in the REI \eqref{relEn} is not allowed anymore due to the mismatch of boundary conditions $\uu_\e |_{\d \O} = 0$ versus $\uu \cdot \vc n |_{\d \O} = 0$. To overcome this drawback, we will introduce a proper boundary layer corrector as in \cite{BalaziAllaireOmnes24} and show the following result:

\begin{theorem}\label{thm-3}
{\tc Let $\Omega$ be a bounded domain of class $C^{3,\mu}$, $\mu\in (0,1)$.} Let
\begin{align*}
r>1, && 1<\a<\frac 32, && \l > \a.
\end{align*}
Let $\eta_r$ comply with \eqref{monG}--\eqref{growthG}, and  $(\uu_\e,p_\e)$ be a weak solution to \eqref{NonNew}, and let $(\uu, p) \in   [W^{1,\infty}(\O) \cap W^{2,2}(\O) {\vc \cap C^{1,\mu}(\O)} ] \times W^{1,\infty}(\O)$ with $\dive \uu = 0$ be a strong solution to Darcy's law \eqref{Darcy}. Then, there exists an $\e_0 > 0$ such that for all $\e \in (0, \e_0)$, we have
\ba\label{ConvSpeedBd}
\|\tilde{\uu}_\e - \uu \|_{L^2(\O)}^2 & \leq C \Big( \e^{\a-1} + {\tc \e^{\frac{3-\a}2}}  
+ \e^{2(3-2\a)} \Big),
\ea
where the constant $C>0$ is independent of $\e$. The last term in \eqref{ConvSpeedBd} can be taken to be zero if {\tc $g_r \equiv 0$, and in this case, we can cover all $1 < \a < 3$}.
 \end{theorem}

\begin{remark}
    {\tc While reviewing the present work, the authors learned that for the stationary Stokes case ($g_r=0$, $\lambda = \infty$) the same convergence rates were recently obtained with a different technique than presented here in \cite{JLP25}. }
\end{remark}

To deal with bounded domains, we use a suitable cutoff of the limit function $\uu$ that preserves the solenoidal property. Therefore, it is useful to do the cutoff on the level of a vector potential of $\uu$.
To this end, we rely on the following Lemma, taken from \cite{BalaziAllaireOmnes24} (see also \cite[Lemma A.1]{Kato83} for the same result for smooth domains).
\begin{lemma}[{\cite[Lemma 4.10]{BalaziAllaireOmnes24}}]
    Let $\Omega \subset \R^3$ be a bounded domain of class $C^{3,\mu}$ for some $\mu \in (0,1)$. {\tc Let $\vc g \in C^{1,\mu}(\Omega;\R^3)$ satisfying $\vc g \cdot \vc n = 0$ on $\partial \Omega$. Then there exists $\vc A = (\mathcal{A}_{ij})_{1\leq i < j \leq 3} \in C^{2,\mu}(\O)$ such that 
    \begin{align*}
        \mathbf{curl}\, \vc A = \vc g \text{ on } \partial \Omega, && \vc A = 0 \text{ on } \partial \Omega.
    \end{align*}}
\end{lemma}
{\tc Here the operator  $\mathbf{curl}$ is defined as
$$\mathbf{curl} \, \vc A :=  \left( \sum_{j=1}^{i-1} \frac{\d \mathcal{A}_{ji}}{\d x_{j}} -  \sum_{j=i+1}^{3} \frac{\d \mathcal{A}_{ij}}{\d x_{j}}   \right)_{1\leq i \leq 3}.$$
In particular, $$
  \dive \mathbf{curl} = 0.
$$}

Note that $\uu$ satisfies the assumptions of the above Lemma. We therefore fix $\vc A$ such that
\begin{align*}
\mathbf{curl} \, \vc A = \uu \text{ on } \d \O, \qquad \vc A |_{\d \O} = 0.
\end{align*}

To lean the notation in the following, we will use the symbol $a \lesssim b$ to indicate that there is a constant $C>0$ independent of $a, b$, and $\e$ such that $a \leq C b$. Now, for $\varpi > 0$, let $\O^\varpi = \{ x \in \Omega : {\rm dist}(x, \d \O) < \varpi \}$ and $\xi^\varpi \in C^2(\O)$ be such that $\xi^\varpi = 1$ on $\O^{\varpi /2}$, $\xi^\varpi = 0$ in $\O \setminus \O^\varpi$, and
\begin{align*}
0\leq \xi^\varpi\leq 1, \qquad  \|\nabla \xi^\varpi\|_{L^\infty(\O)} \lesssim \varpi^{-1}, \qquad \|\nabla^2 \xi^\varpi\|_{L^\infty(\O)} \lesssim \varpi^{-2}.
\end{align*}

Set $\vv^\varpi = {\rm curl}(\xi^\varpi \vc A)$, then
\begin{align*}
\|\vv^\varpi\|_{L^p(\O)} = \|\vv^\varpi\|_{L^p(\O^\varpi)} \lesssim \|\nabla \xi^\varpi\|_{L^\infty(\O)} \|\vc A\|_{L^\infty(\O^\varpi)} |\O^\varpi|^\frac{1}{p} + \|\nabla \vc A\|_{L^\infty(\O)} |\O^\varpi|^\frac{1}{p} \lesssim \varpi^\frac{1}{p},
\end{align*}
where we used that due to $\vc A|_{\d \O} = 0$, we have $\|\vc A\|_{L^\infty(\O^\varpi)} \lesssim \varpi \|\nabla \vc A\|_{L^\infty(\O^\varpi)}$. Moreover,
\begin{align*}
\|\nabla \vv^\varpi\|_{L^p(\O)} &\lesssim \|\nabla^2 \xi^\varpi\|_{L^\infty(\O)} \|\vc A\|_{L^\infty(\O^\varpi)} |\O^\varpi|^\frac{1}{p} + \|\nabla \xi^\varpi\|_{L^\infty(\O)} \|\nabla \vc A\|_{L^\infty(\O^\varpi)} |\O^\varpi|^\frac{1}{p} \\
&\qquad + \|\nabla^2 \vc A\|_{L^\infty(\O)} |\O^\varpi|^\frac{1}{p} \\
&\lesssim \varpi^{-1 + \frac{1}{p}} = \varpi^{- \frac{1}{p'}}.
\end{align*}

We define the test function $\twe$ as
\begin{align*}
\twe = W_\e(\uu - \vv^\varpi) - B_\e(\uu - \vv^\varpi),
\end{align*}
where $\varpi$ will be chosen later dependent on $\e$ such that $\e \leq \varpi \ll 1$.
Note that this yields $\twe  \in W_{0, \dive}^{1,\max\{2,r\}}  (\O_\e)$. In particular, $\twe$ is admissible in the REI \eqref{rel-en-1}.


Using property \eqref{W_O_varpi} of $W_\e$, we deduce, since $\|W_\e\|_{L^\infty} \lesssim 1$, $\e \leq \varpi$, and using the estimate for $B_\e$ from Lemma \ref{lem:Bogovski.W_eps},
\begin{align*}
\|\twe - W_\e \uu\|_{L^p(\O)} &\lesssim \|\vv^\varpi\|_{L^p(\O^\varpi)} + \|B_\e(\vu - \vv^\varpi)\|_{L^p(\O)} \lesssim \varpi^\frac{1}{p} + \e\|(W_\e - \Id):\nabla(\uu - \vv^\varpi)\|_{L^p(\O)} \\
&\lesssim \varpi^\frac{1}{p} + \e\|W_\e - \Id\|_{L^p(\O)} + \e  \|W_\e - \Id\|_{L^{p}(\O^\varpi)} \|\nabla \vv^\varpi\|_{L^\infty(\O)} \\
&\lesssim \varpi^\frac{1}{p} + \e \varpi^{-\frac{1}{p'}} \|W_\e - \Id\|_{L^{p}(\O)} \lesssim \varpi^\frac{1}{p},
\end{align*}
and, for $p> \frac32$, 
\begin{align*}
\|\nabla(\twe - W_\e \uu)\|_{L^p(\O)} &\lesssim \|\nabla (W_\e \vv^\varpi)\|_{L^p(\O^\varpi)} + \|\nabla B_\e(\uu - \vv^\varpi)\|_{L^p(\O)} \\
&\lesssim \|\nabla W_\e\|_{L^p(\O^\varpi)} \|\vv^\varpi\|_{L^\infty(\O)} + \|\nabla \vv^\varpi\|_{L^p(\O)} \\
&\qquad + \|(W_\e - \Id) : \nabla (\uu - \vv^\varpi)\|_{L^p(\O)} \\
&\lesssim \varpi^\frac{1}{p} \|\nabla W_\e\|_{L^p(\O)} + \varpi^{- \frac{1}{p'}} + \|(W_\e - \Id) : \nabla (\uu - \vv^\varpi)\|_{L^p(\O)} \\
&\lesssim \varpi^\frac{1}{p} \e^\frac{(3-p)\a - 3}{p} + \varpi^{- \frac{1}{p'}} + \varpi^{-\frac{1}{p'}} \|W_\e - \Id\|_{L^{p}(\O)} \\
&\lesssim \varpi^\frac{1}{p} \e^\frac{(3-p)\a - 3}{p} + \varpi^{- \frac{1}{p'}}.
\end{align*}\\

We choose now $\varpi = \eps^{\frac {3-\alpha}{2}} \gg \eps$ such that
\begin{align*}
    \|\twe - W_\e \uu\|_{L^2(\O)} & \leq C\eps^{\frac {3-\a}{4}},  \\
    \|\nabla(\twe - W_\e \uu)\|_{L^2(\O)} &\leq C\eps^{\frac {\alpha-3}{4}}, \\
    \|\nabla\twe\|_{L^q(\O)} & \leq C\e^{-\a + \frac{3(\a-1)}{q}} \qquad \text{for all } q \in(3/2,\infty),
\end{align*}
and we point out that the last estimate is the same as for  $\vc w_\e$ in \eqref{est.w_e-stat}.

Inserting $\twe$ into the REI \eqref{relEn} and replacing $\vc f$ with Darcy's law as before gives
\ba
&\int_{\O_{\e}} \e^{3-\a} [ \eta_r(\e^{3-\a} D\uu_\e )D\uu_{\e} -  \eta_r(\e^{3-\a} D \twe ) D \twe ] : D(\uu_\e - \twe ) \dx  \\
&\leq \int_{\O_\e} \e^{3-\a} \eta_r(\e^{3-\a} D \twe) D \twe : (D \twe - D \uu_\e ) \dx  - \e^\l  \int_{\O_\e} (\uu_\e - \twe) \cdot ( \d_t \twe + (\uu_\e \cdot \nabla) \twe) \dx\\
&\qquad + \int_{\O_\e} \frac{\eta_0}{2} M_0 \uu \cdot (\uu_\e - \twe) \dx  . \nn
\ea

First, we see that
\begin{align*}
    \|\uu_\e - \twe\|_{L^{2+\kappa}(\O_\e)}^2 \leq \|\uu_\e - \twe\|_{L^2(\O_\e)}^{2\th} \|\uu_\e - \twe\|_{L^6(\O_\e)}^{2(1-\th)} \lesssim \e^{\th(3-\a)} \|\nabla(\uu_\e - \twe)\|_{L^2(\O_\e)}^2,
\end{align*}
where
\begin{align*}
    0 < \kappa < 4, \qquad 0 < \th < 1, \qquad \frac{1}{2+\kappa} = \frac{\th}{2} + \frac{1-\th}{6} \Rightarrow \th = \frac{4-\kappa}{4 + 2\kappa}.
\end{align*}

Hence,
\ba
&\e^\l \int_{\O_\e} (\uu_\e - \twe) \cdot ( (\uu_\e \cdot \nabla) \twe) \dx  \\
&= \e^\l  \int_{\O_\e} (\uu_\e - \twe) \cdot  ((\twe \cdot \nabla) \twe) \dx  + \e^\l \int_{\O_\e} (\uu_\e - \twe) \cdot ((\uu_\e - \twe) \cdot \nabla) \twe \dx \\
&\leq C \e^\l \|\nabla \twe\|_{L^2(\O_\e)} \|\uu_\e - \twe\|_{L^2(\O_\e)} + C \e^\l \|\nabla \twe\|_{L^\frac{2+\kappa}{\kappa}(  \O_\e)} \|\uu_\e - \twe\|_{L^{2+\kappa}(\O_\e)}^2 \\
& \leq C \e^{\l + \frac{3-\a}{2}} \|\nabla(\uu_\e - \twe)\|_{L^2(\O_\e)} \|\nabla \twe\|_{L^2(\O_\e)} \\
&\qquad + C \e^{\l + \th(3-\a)} \|\nabla(\uu_\e - \twe)\|_{L^2(\O_\e)}^2 (\|\nabla(W_\e \uu)\|_{L^\frac{2+\kappa}{\kappa}(\O_\e)} + \|\nabla(\twe - W_\e \uu)\|_{L^\frac{2+\kappa}{\kappa}(\O_\e)}) \\
& \leq C \e^{\l + \frac{3-\a}{2}} \|\nabla(\uu_\e - \twe)\|_{L^2(\O_\e)} (1 + \e^\frac{\a-3}{2}) \\
&\qquad + C \e^{\l + \th(3-\a)} \|\nabla(\uu_\e - \twe)\|_{L^2(\O_\e)}^2 (\e^{-\a} \e^\frac{3\kappa (\a - 1)}{2+\kappa} + \varpi^\frac{\kappa}{2+\kappa} \e^{-\a} \e^\frac{3\kappa (\a - 1)}{2+\kappa} + \varpi^{-\frac{2}{2+\kappa}}) \\
& \leq C \e^{2\l +\a-3} + \delta \e^{3-\a} \|\nabla(\uu_\e - \twe)\|_{L^2(\O_\e)}^2 \\
&\qquad + C \e^{\l - (1 - \th) (3-\a)} (\e^{-\a} \e^\frac{3\kappa (\a - 1)}{2+\kappa} + \varpi^{-\frac{2}{2+\kappa}}) \e^{3-\a} \|\nabla(\uu_\e - \twe)\|_{L^2(\O_\e)}^2.
\nn \ea

Note that the constant depends on $\kappa$.
For $\kappa \to 0$ (and thus $\theta \to 1$) and the choice $\varpi = \e^\frac{3-\a}{2}$ the last right-hand side term is bounded by
\begin{align*}
    C \e^{\l - \a - o(1)} \e^{3-\a} \|\nabla(\uu_\e - \twe)\|_{L^2(\O_\e)}^2.
\end{align*}
Since $\l > \a$, we can choose $\kappa$ sufficiently small such that (for $\e$ small enough)
this term is bounded by 
\begin{align*}
    \delta \e^{3-\a} \|\nabla(\uu_\e - \twe)\|_{L^2(\O_\e)}^2
\end{align*}
and may therefore be absorbed into the dissipation.

Using the definition of $\eta_r$ in \eqref{def:eta_r} similarly to before, we get, using also that $\dive(\twe - \uu_\e)=0$ and writing $\twe = W_\e \uu + (\twe - W_\e \uu)$,
\ba
& \int_{\O_\e} \e^{3-\a} \eta_r(\e^{3-\a} D \twe) D\twe : D( \twe - \uu_\e) \dx  \\
&= \e^{3-\a} \int_{\O_\e} \eta_0 D (W_\e \uu) : D(\twe - \uu_\e) \dx + \e^{3-\a}  \int_{\O_\e} \eta_0 D (\twe - W_\e \uu) : D(\twe - \uu_\e) \dx  \\
&\qquad + \e^{3-\a}  \int_{\O_\e} g_r(|\e^{3-\a} D\twe|) D \twe : D(\twe - \uu_\e) \dx.
\nn \ea

The first and third right-hand side terms are treated as for the torus case. The second term is estimated by
\begin{align*}
\e^{3-\a}  \int_{\O_\e} \eta_0 D (\twe - W_\e \uu) : D(\twe - \uu_\e) \dx & \leq C\e^{3-\a} \|\nabla(\twe - W_\e \uu)\|_{L^2(\O_\e)} \|\nabla(\twe - \uu_\e)\|_{L^2(\O_\e)} \\
&\leq C\e^{\frac{3-\a}2}  + \delta \e^{3-\a} \|\nabla(\twe - \uu_\e)\|_{L^2(\O_\e)}^2.
\end{align*}
After absorbing into the dissipation and using the Poincar\'e inequality \eqref{Poincare}, we end up with 
\ba
\|\uu_\e - \twe\|_{L^2(\O_\e)}^2 \leq C \e^{3-\a} \|\nabla (\uu_\e - \twe)\|_{L^2(\O_\e)}^2
\leq  C(\e^{\a -1} +  \e^{\frac{3-\a}2} + \e^{2(3-2\a)}). 
\nn \ea

To conclude the estimate for $\uu_\e - \uu$ we observe
\begin{align*}
\|\uu -  \twe\|_{L^2(\O)}^2 \leq C \|(W_\e - \Id) \uu\|_{L^2(\O)}^2 + \|W_\e \uu - \twe\|_{L^2(\O)}^2 
\leq C(\e^{2(\a-1)} + \e^{\frac{3-\a}2}),
\end{align*}
thus
\begin{align*}
    \|\uu - \tilde \uu_\e\|_{L^2(\O)}^2 \leq C \|\uu - \twe \|_{L^2(\O)}^2 + \|\twe - \tilde \uu_\e\|_{L^2(\O)}^2 \leq C(\e^{\a -1} +  \e^{\frac{3-\a}2} + \e^{2(3-2\a)}),
\end{align*}
ending the proof of Theorem~\ref{thm-3}.

\section{Evolutionary NSE}\label{sec:EvolNSE}
In this section, for $T>0$, we consider in $(0,T) \times \O_\e$ the evolutionary ``sister'' of \eqref{NonNew}, namely
\begin{equation}\label{NonNewTime}
\begin{cases}
\e^\l( \d_t \uu_\e + \dive (\uu_\e \otimes \uu_\e) ) -\e^{3-\a} {\rm div}\,\big(\eta_r(\e^{3-\a} D\uu_\e )D\uu_\e  \big) +\nabla p_\e=\mathbf{f} & {\rm in}\ (0,T) \times \O_\e, \\
{\rm div} \,\uu_\e =0 & {\rm in}\  (0,T) \times \O_\e,\\
\uu_\e =0 & {\rm on}\ (0,T) \times \d\O_\e,\\
\uu_\e(0, \cdot) = \uu_{\e 0} & {\rm in} \ \O_\e,
\end{cases}
\end{equation}
where this time $\mathbf{f} \in L^{2}((0,T)\times \O;\R^{3})$, and the initial data are given by
\begin{align}\label{init}
\uu_{\e0} \in L^2(\O_\e), && \dive \uu_{\e0} = 0, && \e^{\frac{\l}{2}}\|\uu_{\e0}\|_{ L^2(\O_\e)} \leq C.
\end{align}
As before, this scaling corresponds to ${\rm Re} = \e^{\l + \a - 3}$, {\tc ${\rm Eu} = \e^{-\l}$,} ${\rm Fr} = \e^\frac{\l}{2}$, with the additional Strouhal number being equal to ${\rm Sr} = 1$. Note that this scaling corresponds also to a rescaling in time as $\hat t = \e^{-\l} t$; thus, in the limit, we shall expect \emph{time-independent} equations by means of a long-time behavior.\\

The notion of weak solutions is similar to the one of Definition \ref{def-weak}:
\begin{defi}\label{def-weak-time}
We say that  $\uu_{\e}$ is a finite energy weak solution of \eqref{NonNewTime} in $(0,T) \times \O_{\e}$ with initial datum $\uu_{\e 0} \in L^2(\O_\e)$, $\dive \uu_{\e 0} = 0$, provided
\begin{itemize}
\item $\uu_{\e}\in  L^2(0,T; W_{0, \rm div}^{1,2}(\O_\e)) \cap L^r(0,T; W_{0, \rm div}^{1,r}(\O_\e)) \cap C_{\rm weak}([0,T];L^2(\O_\e))$;



\item There holds the integral identity for any solenoidal test function $\varphi\in C_c^\infty( [0,T) \times \O_\e;\mathbb{R}^3)$:
\ba\label{eq-weak-time}
& \int_0^T \int_{\O_{\e}} -\e^\l \uu_\e \cdot \d_t \varphi - \e^\l \uu_{\e}\otimes \uu_{\e}:\nabla\varphi + \e^{3-\a} \eta_r(\e^{3-\a} D\uu_\e )D\uu_{\e} : D\varphi 
\dx \dt \\
&\qquad = \int_{\O_\e} \e^\l \uu_{\e 0} \cdot \varphi(0, \cdot) \dx + \int_0^T \int_{\O_\e}\mathbf{f}\cdot\varphi \dx \dt ;
\ea

\item For a.a. $\tau \in (0,T)$, there holds the energy inequality
\ba\label{energy-ineq-time}
&\e^\l \int_{\O_\e} \frac12 |\uu_\e|^2(\tau, \cdot) \dx + \e^{3-\a} \int_0^\tau \int_{\O_\e}\eta_r(\e^{3-\a} D\uu_\e )| D\uu_{\e}|^2 \dx \dt \\
&\leq \e^\l \int_{\O_\e} \frac12 |\uu_{\e 0}|^2 \dx + \int_0^\tau \int_{\O_\e}\mathbf{f}\cdot \uu_{\e}\dx \dt.
\ea

\end{itemize}
\end{defi}
The existence of such weak solutions is know thanks to the pioneer results introduced in Remark~\ref{exis-weak}. 

\subsection{Qualitative homogenization}
Our main theorem in this section reads as follows.
\begin{theorem}\label{thm-2}
Let
\begin{align*}
r > 1, && 1<\a<\frac32, && \l>\a. 
\end{align*}
{\tc Let $\eta_r$ comply with \eqref{monG}--\eqref{growthG}}, let the initial datum $\uu_{\e0} \in L^2(\O_\e)$  satisfy \eqref{init}, and let $(\uu_\e)$ be a finite energy weak solution of equations \eqref{NonNewTime} with initial datum $\uu_{\e0}$. 
Then, $ \tilde \uu_\e \weak \uu$  weakly in $L^2((0,T) \times \O)$,
where the limit $\uu$ satisfies Darcy's law
\begin{align}\label{Darcy-t}
\begin{cases}
\frac{1}{2}\eta_0\uu = M_{0}^{-1}(\mathbf{f}-\nabla p) & \text{in } (0,T) \times \O,\\
\dive \uu = 0 & \text{in } (0,T) \times \O,\\
\uu \cdot \mathbf{n} = 0 & \text{on } (0,T) \times \d \O.
\end{cases}
  \end{align}
  Here, $M_{0}$ is the same permeability tensor as before. {\tc Moreover, if $g_r \equiv 0$, then the above conclusion holds true for any $1 < \a < 3$.}
\end{theorem}
 
 The next two subsections are devoted to the proof of Theorem~\ref{thm-2}. 

\subsubsection{Uniform bounds}\label{sec:evolBds}
From the energy inequality and the boundedness of the initial datum $\e^\frac{\l}{2} \uu_{\e0}$ in $L^2(\O_\e)$, we infer
\begin{align*}
&\e^\l \int_\O \frac12 |\tilde \uu_\e|^2(\tau, \cdot) \dx + \e^{3-\a} \int_0^\tau \int_\O \eta_r(\e^{3-\a} D \tilde \uu_\e) |D\tilde \uu_\e|^2 \dx \dt \\
& \leq C  + \int_0^\tau \int_\O \mathbf{f} \cdot \tilde \uu_\e \dx \dt \leq   C   + C \|\ff\|_{L^2((0,T) \times \O_\e)} \| \uu_\e\|_{L^2((0,T) \times \O_\e)} \\
&\leq C  + C \e^\frac{3-\a}{2} \|D\uu_\e\|_{L^2((0,T) \times \O_\e)} \leq C + {\tc \frac{\eta_0}{2} \e^{3-\a} \int_0^\tau \int_{\O_\e} |D \uu_{\e}|^2 \dx \dt. }
\end{align*}
This implies
\begin{align*}
&\e^\l \int_\O \frac12 |\tilde \uu_\e|^2(\tau, \cdot) \dx + {\tc \e^{3-\a} \int_0^\tau \int_\O \frac{\eta_0}{2} |D\tilde \uu_\e|^2 + g_r(|\e^{3-\a} D \tilde \uu_\e|) |D \tilde \uu_\e|^2 \dx \dt \leq C, }
\end{align*}
hence, {\tc by the monotonicity condition \eqref{monG}, we find}
\ba\label{est-u-f-t0}
\e^\frac{\l}{2} \|\uu_\e\|_{L^\infty(0,T;L^2(\O_\e))} \leq C, && \e^{3-\a} \|D\uu_\e\|_{L^2((0,T) \times \O_\e)}^2 + \e^{(3-\a)(r-1)} \|D \uu_\e\|_{L^r((0,T) \times \O_\e)}^r \mathbf{1}_{r>2} \leq C.
\ea
Consequently, as before,
\ba\label{est-u-f-t}
\e^\frac{3-\a}{2}\|\nabla \tilde \uu_\e\|_{L^2((0,T) \times \O)} \leq C, && \|\tilde \uu_\e\|_{L^2((0,T) \times \O)} \leq C, && \e^\frac{(3-\a)(r-1)}{r} \|\nabla \tilde \uu_\e\|_{L^r((0,T) \times \O)} \mathbf{1}_{r>2} \leq C,
\ea
and, up to subsequences,
\ba\label{est-u-f-t-2}
\tilde \uu_\e \weak \uu \text{ weakly in } L^2((0,T) \times \O) , \qquad \dive\uu = 0 \ \mbox{in} \  (0,T) \times \O, \qquad \uu\cdot \vc n =0 \ \mbox{on} \ (0,T) \times \d \O.
\ea

\subsubsection{Proof of Theorem \ref{thm-2}}\label{sec:evolQual}
Recall the functions $(W_\e, Q_\e)$ {\tc from Proposition}~\ref{lem-local-1}. Let $\phi\in C_c^\infty([0,T) \times \O; \R^3)$ {\tc with $\dive \phi = 0$}. Taking as before $\phi_\e = W_\e \phi - B_\e(\phi)$ as a test function in the weak formulation of \eqref{eq-weak-time} implies that for each $\tau \in [0,T]$,
\ba\label{Pe-eq}
  0 & = \int_0^\tau \int_{\O_{\e}}  \ff \cdot  \phi_{\e} \dx \dt + \e^\l \int_{\O_\e} \uu_{\e 0} \cdot \phi_\e(0, \cdot) \dx + \e^{\l} \int_0^\tau \int_{\O_{\e}} \uu_\e \cdot \d_t \phi_\e + \uu_\e \otimes \uu_\e : \nabla \phi_{\e} \dx \dt  \\
  & \qquad - \e^{3-\a}\frac{\eta_0}{2} \int_0^\tau \int_{\O_{\e}} \nabla \uu_\e : \nabla  \phi_{\e} \dx \dt - \int_0^\tau \int_{\O_{\e}}\e^{3-\a} g_r(|\e^{3-\a} D\uu_\e|) D\uu_\e : \nabla  \phi_{\e} \dx \dt.
  \ea

Again by {\tc Proposition}~\ref{lem-local-1}, there holds
\ba
\int_0^\tau \int_{\O_{\e}} \ff \cdot  \phi_{\e}  \dx \dt = \int_0^\tau \int_{\O} \ff \cdot \phi_{\e}  \dx \dt \to \int_0^\tau \int_{\O} \ff \cdot  \phi \dx \dt. \nn
\ea

From  \eqref{init} and \eqref{est-u-f-t0}, it is straightforward to deduce 
\ba
\int_{\O_{\e}} \e^\l \uu_{\e0}  \cdot  \phi_{\e}(0, \cdot) \dx + \int_0^\tau \int_{\O_\e} \e^\l \uu_\e \cdot \d_t \phi_\e \dx \dt &\leq C \e^{\l} (\|\uu_{\e}\|_{L^{\infty}(0,T;L^{2}(\O_{\e}))} + \|\uu_{\e0}\|_{L^{2}(\O_{\e})} ) \\
&\leq C \e^{\frac \l 2} \to 0. \nn
\ea

By Proposition~\ref{lem-local-1} and \eqref{est-u-f-t}, similar arguments as in \eqref{limit-2} imply
\ba
\e^{\l} \int_0^\tau \int_{\O_{\e}} \uu_\e \otimes \uu_\e : \nabla  \phi_{\e} \dx \dt & \leq C \e^{\l} \|\uu_\e\otimes \uu_\e\|_{L^{1}(0,T; L^{\frac{3}{3-2\th}}(\O_{\e}))} \|\nabla \phi_{\e}\|_{L^{\frac{3}{2\th}}(\O)}  \\
& \leq C \e^{\l - \a - \th (5 - 3\a)}  \to 0, \nn
\ea
where we have chosen  $\th>0$ suitably small and used the following interpolation estimate
\ba
\|\uu_\e\otimes \uu_\e\|_{L^{1}(0,T; L^{\frac{3}{3-2\th}}(\O_{\e}))} &\leq \|\uu_\e \otimes  \uu_{\e} \|_{L^{1}(0,T; L^1(\O_\e))}^{1-\th} \|\uu_\e \otimes \uu_{\e}\|_{L^{1}(0,T; L^3(\O_\e))}^\th \\
&  \leq C \| \uu_\e\|_{L^{2}(0,T; L^2(\O_\e))}^{2(1-\th)}  \| \uu_\e\|_{L^{2}(0,T; L^6(\O_\e))}^{2\th} \leq C \e^{-\th(3-\a)}. \nn
\ea

By \eqref{est-u-f-t-2}, along the lines in \eqref{limit-7}--\eqref{limit-8}, we can deduce
\ba
 & -  \e^{3-\a}\frac{\eta_0}{2} \int_0^\tau \int_{\O_{\e}} \nabla \uu_\e : \nabla  \phi_{\e} \dx \dt \to \frac{\eta_0}{2} \int_0^\tau \int_{\O} M_{0}\phi \cdot \uu \dx \dt. \nn
\ea

{\tc It is left to show that the last term in \eqref{Pe-eq} vanishes as $\e \to 0$. Again, if $g_r \equiv 0$, we finish the proof here and cover all $1 < \a < 3$. Hence, in the sequel we will focus on the case $g_r \not \equiv 0$.

For the last term on the right-hand side of \eqref{Pe-eq}, 
by similar arguments as given in \eqref{limit-11}--\eqref{limit-12}, and using \eqref{est-u-f-t}, {\tc Proposition}~\ref{lem-local-1}, and \eqref{G1}, it holds for $1<r \leq 2$ and any $\kappa \in (0,1)$
\ba
& \e^{3-\a} \Big|  \int_{0}^{\tau} \int_{\O_{\e}} g_r(|\e^{3-\a} D \tilde \uu_\e|) D  \tilde \uu _{\e} : D \phi_{\e} \dx \dt \Big|\\
  &\leq C \e^{(3-\a)(1+\kappa)}  \int_{0}^{\tau}  \int_{\O}|  D  \tilde \uu _{\e} |^{1+\kappa}   |D \phi_{\e} | \dx \dt \\
  &\leq  C \e^{(3-\a)(1+\kappa)} \int_{0}^{\tau}\| | \nabla  \tilde \uu _{\e} |^{1+\kappa} \|_{L^{\frac{2}{1+\kappa}}(\O)}  \| \phi_{\e} \|_{W^{1, \frac{2}{1-\kappa}}(\O)}  \dt \\
 &\leq  C \e^{(3-\a)(1+\kappa)} \|  \nabla  \tilde \uu _{\e}  \|_{L^{1+\kappa} ( 0,T; L^{2}(\O))}^{1+\kappa}   \| \phi_{\e} \|_{W^{1, \frac{2}{1-\kappa}}(\O)}  \\
  & \leq C \e^{(3-\a)(1+\kappa)} \e^{-\frac{(3-\a)}{2} (1+\kappa)} \e^{ -\a + 3(\a - 1) \frac{1-\kappa}{2}}   =  C \e^{(3 - 2\a)\kappa} \to 0, \nn
\ea
and for $r>2$, 
\ba
& \e^{3-\a} \Big| \int_0^\tau \int_{\O} g_r(|\e^{3-\a} D \tilde \uu_\e|) D \tilde \uu _{\e} : D \phi_{\e}) \Big| \dx \dt \\
  &\leq C \e^{(3-2\a)\kappa} + C \e^{(3-\a)(r-1)}  \int_{0}^{\tau}  \int_{\O} |  D  \tilde \uu _{\e} |^{r-1} : |D \phi_{\e} | \dx \dt\\
  &\leq  C \e^{(3-2\a)\kappa} + C \e^{(3-\a)(r-1)}  \int_{0}^{\tau} \| | \nabla  \tilde \uu _{\e} |^{r-1} \|_{L^{\frac{r}{r-1}}(\O)}  \| \phi_{\e} \|_{W^{1, r}(\O)}  \dt \\
  &\leq  C \e^{(3-2\a)\kappa} + C \e^{(3-\a)(r-1)}   \|  \nabla  \tilde \uu _{\e}  \|_{L^r(0,T; L^{r} (\O))}^{r-1}  \| \phi_{\e} \|_{W^{1, r}(\O)}  \\
  & \leq C \e^{(3-2\a)\kappa} + C \e^{(3-\a)(r-1)} \e^{-(3-\a) \frac{(r-1)^2}{r}} \e^{ -\a + \frac{3(\a - 1)}{r}}   =  C \e^{(3-2\a)\kappa} + C \e^{\frac{(3 - 2\a)(r-2)}{r}} \to 0, \nn
\ea
under the assumption $1<\a<\frac{3}{2}$.
}

\medskip

Summing up the above convergences, we finally get for each $\tau \in [0,T]$ that
\ba
 \frac{\eta_0}{2} \int_0^\tau \int_{\O} M_{0} \phi \cdot \uu \dx \dt = \int_0^\tau \int_{\O} \ff \cdot \phi \dx \dt. \nn
\ea

\subsection{Quantitative homogenization}\label{sec:evolQuan}
In this section, we derive a relative energy inequality for the time-dependent setting, and use it to prove the following theorem regarding speed of convergence {\tc in the torus case}:
\begin{theorem}\label{thm-3-evol}
Let {\tc $\O = \mathbb{T}^3$ and}
\begin{align*}
r>1, && 1<\a<\frac 32, && \l > \a.
\end{align*}
{\tc Let $\eta_r$ comply with \eqref{monG}--\eqref{growthG}}. Let $(\uu_\e, p_\e)$ be a weak solution to \eqref{NonNewTime} emanating from the initial datum $\uu_{\e 0} \in L^2(\O_\e)$ satisfying \eqref{init}, and let $(\uu, p) \in   [ W^{1,\infty}(0,T;W^{1,\infty}(\O)) \cap L^\infty(0,T;W^{2,2}(\O)) ] \times L^\infty(0,T;W^{1,\infty}(\O))$ with $\dive \uu = 0$ be a strong solution to Darcy's law \eqref{Darcy-t} with initial datum $ \|\uu(0, \cdot) \|_{L^{2}(\O)} \leq C$. Then, there exists an $\e_0 > 0$ such that for all $\e \in (0, \e_0)$ , we have
\ba\label{ConvSpeed-t}
\|\tilde{\uu}_\e - \uu \|_{L^2((0,T) \times \O)}^2 & \leq C \Big( \e^\l \|\tilde{\uu}_{\e 0} - \uu(0, \cdot)\|_{L^2(\O)}^2 + \e^{\a-1} + \e^{3-\a}  + {\tc \e^{2(3-2\a)} } \Big),
\ea
where the constant $C>0$ is independent of $\e$. The last term in \eqref{ConvSpeed-t} can be taken to be zero if {\tc $g_r \equiv 0$, and in this case, we can cover all $1 < \a < 3$}.
 \end{theorem}

\begin{remark}
Due to $\uu(0,\cdot) \in L^2(\O)$, we may replace the first term on the right-hand side of \eqref{ConvSpeed-t} simply by $\e^\l \|\tilde{\uu}_{\e0}\|_{L^2(\O)}^2$ since
\begin{align*}
\e^\l \|\tilde{\uu}_{\e0} - \uu(0,\cdot)\|_{L^2(\O)}^2 \leq C \e^\l \big( \|\tilde{\uu}_{\e0}\|_{L^2(\O)}^2 + 1 \big) \leq C \e^\l \big( \|\tilde{\uu}_{\e0} - \uu(0,\cdot)\|_{L^2(\O)}^2 + 1 \big),
\end{align*}
and $\e^\lambda \leq \e^\alpha \leq \e^{\alpha -1}$ can be absorbed in the remainder of the right-hand side of \eqref{ConvSpeed-t}.
\end{remark}

The rest of this section is devoted to the proof of Theorem~\ref{thm-3-evol}.

\subsubsection{Relative energy inequality}

From \eqref{eq-weak-time}--\eqref{energy-ineq-time}, we can derive that for any {\tc solenoidal} $\vU\in C^{\infty}([0,T]\times \O_{\e};\R^{3})$ satisfying $\vU|_{\d \O_\e} = 0$, 
\ba
& \int_0^\tau \int_{\O_{\e}} -\e^\l \uu_\e \cdot \d_t \vU  - \e^\l \uu_{\e}\otimes \uu_{\e}:\nabla\vU  + \e^{3-\a} \eta_r(\e^{3-\a} D\uu_\e )D\uu_{\e} : D\vU \dx \dt   \\
&\qquad =  - \int_{\O_\e} \e^\l \big(  \uu_\e(\tau, \cdot) \cdot \vU (\tau, \cdot) - \uu_{\e 0} \cdot \vU (0, \cdot) \big)\dx   +  \int_0^\tau \int_{\O_\e}\mathbf{f}\cdot\vU  \dx \dt, \nn
\ea
which yields
\ba
& \int_0^\tau \int_{\O_{\e}} \e^\l \uu_\e \cdot \d_t \vU  \dx \dt - \left[ \int_{\O_\e} \e^\l \uu_{\e} \cdot \vU  \dx \right]_{t=0}^{t=\tau} \\
&\qquad =  \int_0^\tau \int_{\O_{\e}} \e^{3-\a} \eta_r(\e^{3-\a} D\uu_\e )D\uu_{\e} : D\vU  - \e^\l \uu_{\e}\otimes \uu_{\e}:\nabla\vU \dx \dt - \int_0^\tau \int_{\O_\e}\mathbf{f}\cdot\vU  \dx \dt.
\nn
\nn
\ea

Define the relative energy by
\begin{align*}
E_\e(\uu_\e | \vU ) = \frac12 \e^\l |\uu_\e - \vU |^2, \qquad \forall  \, \vU  \in C^{\infty}([0,T]\times \O;\R^{3}), \qquad {\tc \dive \vU = 0, \qquad \vU|_{\d \O_\e} = 0}.
\end{align*}

It follows that
\ba
\left[ \frac12 \e^\l \int_{\O_\e} |\uu_\e - \vU |^2 \dx \right]_{t=0}^{t=\tau} &= \left[ \frac12 \e^\l \int_{\O_\e} |\uu_\e|^2 \dx \right]_{t=0}^{t=\tau} + \e^\l \int_0^\tau \int_{\O_\e} \vU  \cdot \d_t \vU  \dx \dt - \left[ \e^\l \int_{\O_\e} \uu_\e \cdot \vU  \right]_{t=0}^{t=\tau} \\
&= \left[ \frac12 \e^\l \int_{\O_\e} |\uu_\e|^2 \dx \right]_{t=0}^{t=\tau} - \e^\l \int_0^\tau \int_{\O_\e} (\uu_\e - \vU ) \cdot \d_t \vU  \dx \dt \\
&\qquad + \int_0^\tau \int_{\O_{\e}} \e^{3-\a} \eta_r(\e^{3-\a} D\uu_\e )D\uu_{\e} : D\vU  - \e^\l \uu_{\e}\otimes \uu_{\e}:\nabla\vU  \dx \dt \\
&\qquad - \int_0^\tau \int_{\O_\e}\mathbf{f} \cdot \vU  \dx \dt .
\nn
\ea

By the energy inequality \eqref{energy-ineq-time}, we see
\ba
\e^\l \left[ \int_{\O_\e} \frac12 |\uu_\e|^2 \dx \right]_{t=0}^{t=\tau} \leq - \e^{3-\a} \int_0^\tau \int_{\O_\e} \eta_r(\e^{3-\a} D\uu_\e) |D\uu_\e|^2 \dx \dt + \int_0^\tau \int_{\O_\e} \vc f \cdot \uu_\e \dx \dt,
\nn \ea
hence
\ba
&\left[ \int_{\O_\e} E_\e(\uu_\e | \vU ) \dx \right]_{t=0}^{t=\tau} + \int_0^\tau \int_{\O_{\e}} \e^{3-\a} [ \eta_r(\e^{3-\a} D\uu_\e )D\uu_{\e} -  \eta_r(\e^{3-\a} D\vU  )D\vU  ] : (D\uu_\e - D\vU ) \dx \dt \\
&\leq \int_0^\tau \int_{\O_\e} \e^{3-\a} \eta_r(\e^{3-\a} D \vU ) D \vU  : (D \vU  - D \uu_\e) \dx \dt - \e^\l \int_0^\tau \int_{\O_\e} (\uu_\e - \vU ) \cdot \d_t \vU  \dx \dt \\
&\qquad - \int_0^\tau \int_{\O_\e} \e^\l \uu_{\e}\otimes \uu_{\e}:\nabla\vU  \dx \dt + \int_0^\tau \int_{\O_\e}\mathbf{f}\cdot (\uu_\e - \vU ) \dx \dt. \nn
\ea
Using moreover $\dive \uu_\e = 0$ and $\uu_\e |_{\d \O_\e} = 0$, we may write
\ba
-\int_0^\tau \int_{\O_\e} \e^\l \uu_\e \otimes \uu_\e : \nabla \vU  \dx \dt = -\int_0^\tau \int_{\O_\e} \e^\l ((\uu_\e \cdot \nabla) \vU ) \cdot (\uu_\e - \vU ) \dx \dt \nn
\ea
to get the final {\tc relative energy inequality (REI)}
\ba \label{relEnIn}
&\left[ \int_{\O_\e} E_\e(\uu_\e | \vU ) \dx \right]_{t=0}^{t=\tau} + \int_0^\tau \int_{\O_{\e}} \e^{3-\a} [ \eta_r(\e^{3-\a} D\uu_\e )D\uu_{\e} -  \eta_r(\e^{3-\a} D\vU  )D\vU  ] : D(\uu_\e - \vU ) \dx \dt \\
&\leq \int_0^\tau \int_{\O_\e} \e^{3-\a} \eta_r(\e^{3-\a} D \vU ) D \vU  : (D \vU  - D \uu_\e) - \e^\l \int_0^\tau \int_{\O_\e} (\uu_\e - \vU ) \cdot ( \d_t \vU  + (\uu_\e \cdot \nabla) \vU  ) \dx \dt \\
&\qquad + \int_0^\tau \int_{\O_\e}\mathbf{f}\cdot (\uu_\e - \vU ) \dx \dt .
\ea

Note that inserting $\vU =0$ in the above yields the standard energy inequality \eqref{energy-ineq-time}. Moreover, by density, the REI \eqref{relEnIn} holds for any $\vU $ satisfying $\vU  \in W^{1,2}([0,T],W_{0, \dive}^{1,\max\{2,r\}}  (\O_\e)).$

\subsubsection{Proof of Theorem \ref{thm-3-evol}}

Let {\tc $(W_\e, Q_\e)$ be the functions from Proposition~\ref{lem-local-1}} and $\uu$ be a regular strong solution of Darcy's law \eqref{Darcy-t} as required in Theorem~\ref{thm-3-evol}. Define further {\tc $\vc w_\e = W_\e \uu - B_\e(\uu)$.} Then, by Proposition~\ref{lem-local-1} and Lemma \ref{lem:Bogovski.W_eps}, there holds {\tc similarly as in the stationary case that}
\ba
&\|\vc w_\e\|_{W^{1,\infty}(0,T; L^\infty(\O))}  + {\tc \e^{\a - \frac{3(\a-1)}{q}} \|\nabla \vc w_\e\|_{L^q((0,T) \times \O)}} \leq C, \ {\tc q \in  (3/2, \infty),} \\
& \|\vc w_\e - \uu\|_{L^\infty(0,T;L^q(\O_\e))} \leq C \e^{\min\{1, \frac{3}{q} \} (\a - 1)} \ \forall \, q \geq  1, \ {\tc q \neq 3}, \qquad \vc w_\e |_{\d \O_\e} = 0. \nn
\ea

We use {\tc $\vU = W_\e \uu - B_\e(\uu) = \vc w_\e$} as test function in the REI \eqref{relEnIn} to obtain
\ba
&\left[ \int_{\O_\e} E_\e(\uu_\e | \vc w_\e) \dx \right]_{t=0}^{t=\tau} + \int_0^\tau \int_{\O_{\e}} \e^{3-\a} [ \eta_r(\e^{3-\a} D\uu_\e )D\uu_{\e} -  \eta_r(\e^{3-\a} D\vc w_\e )D\vc w_\e ] : D(\uu_\e - \vc w_\e ) \dx \dt \\
&\leq \int_0^\tau \int_{\O_\e} \e^{3-\a} \eta_r(\e^{3-\a} D \vc w_\e) D\vc w_\e : (D \vc w_\e - D \uu_\e ) \dx \dt \\
&\qquad - \e^\l \int_0^\tau \int_{\O_\e} (\uu_\e - \vc w_\e) \cdot ( \d_t \vc w_\e + (\uu_\e \cdot \nabla) \vc w_\e) \dx \dt + \int_0^\tau \int_{\O_\e}\mathbf{f}\cdot (\uu_\e - \vc w_\e) \dx \dt . \nn
\ea

Note especially that, for the torus case, this choice of $\vU$ is a valid one due to {\tc $\dive \vU = 0$ and} $\vU = \vc w_\e = 0$ on $\d \O_\e$. {\tc Moreover, as $\uu$ is a regular solution to Darcy's law, we have $\vc w_\e \in W^{1,2}(0,T; W_{0, \dive}^{\max\{2,r\}}(\O_\e))$.}

\medskip

Using Darcy's law \eqref{Darcy-t} {\tc and solenoidality of $\uu_\e$ and $\vc w_\e$ similar to before, we rewrite the force term and insert the result into the REI to obtain}
\ba\label{rela-ener-1}
&\left[ \int_{\O_\e} E_\e(\uu_\e | \vc w_\e) \dx \right]_{t=0}^{t=\tau} + \int_0^\tau \int_{\O_{\e}} \e^{3-\a} [ \eta_r(\e^{3-\a} D\uu_\e )D\uu_{\e} -  \eta_r(\e^{3-\a} D\vc w_\e )D\vc w_\e ] : D(\uu_\e - \vc w_\e ) \dx \dt \\
&\leq \int_0^\tau \int_{\O_\e} \e^{3-\a} \eta_r(\e^{3-\a} D \vc w_\e) D\vc w_\e : (D \vc w_\e - D \uu_\e ) \dx \dt \\
& \qquad - \e^\l \int_0^\tau \int_{\O_\e} (\uu_\e - \vc w_\e) \cdot ( \d_t \vc w_\e + (\uu_\e \cdot \nabla) \vc w_\e) \dx \dt  + \int_0^\tau \int_{\O_\e} \frac{\eta_0}{2} M_0 \uu \cdot (\uu_\e - \vc w_\e) \dx \dt.
\ea

{\tc Again, the dissipation is bounded from below due to \eqref{monG} and Korn's inequality by
\begin{align*}
    &\int_0^\tau \int_{\O_{\e}} \e^{3-\a} [ \eta_r(\e^{3-\a} D\uu_\e )D\uu_{\e} -  \eta_r(\e^{3-\a} D\vc w_\e )D\vc w_\e ] : D(\uu_\e - \vc w_\e ) \dx \dt \\
    &\geq c \e^{3-\a}  \|\nabla (\uu_\e - \vc w_\e )\|_{L^2((0,T) \times \O_\e)}^2 + c \mathbf{1}_{r>2} \e^{(3-\a)(r-1)}  \|\nabla (\uu_\e - \vc w_\e )\|_{L^r((0,T) \times \O_\e)}^r.
\end{align*}
}

The second term on the right-hand side of \eqref{rela-ener-1} may be split as
\ba
&\e^\l \int_0^\tau \int_{\O_\e} (\uu_\e - \vc w_\e) \cdot ( \d_t \vc w_\e + (\uu_\e \cdot \nabla) \vc w_\e) \dx \dt \\
&= \e^\l \int_0^\tau \int_{\O_\e} (\uu_\e - \vc w_\e) \cdot ( \d_t \vc w_\e + (\vc w_\e \cdot \nabla) \vc w_\e) \dx \dt + \e^\l \int_0^\tau \int_{\O_\e} (\uu_\e - \vc w_\e) \cdot ((\uu_\e - \vc w_\e) \cdot \nabla) \vc w_\e \dx \dt\\
&\leq C \e^\l {\tc(1 + \|\nabla \vc w_\e\|_{L^2((0,T) \times \O_\e)})} \|\uu_\e - \vc w_\e\|_{L^2((0,T) \times \O_\e)} \\
&\qquad + C \e^\l {\tc \|\nabla (W_\eps \uu)\|_{L^\infty((0,T) \times \O_\e)} } \|\uu_\e - \vc w_\e\|_{L^2((0,T) \times \O_\e)}^2 \\
& \qquad {\tc + \|\nabla B_\eps(\uu)\|_{L^\infty(0,T;L^3(\O_\e))} \|\uu_\e - \vc w_\e\|_{L^2((0,T) \times \O_\e)} \|\nabla(\uu_\e - \vc w_\e)\|_{L^2((0,T) \times \O_\e)}} \\
& \leq C_\delta \e^{\a-1} + \delta \e^{3-\a} \|\nabla (\uu_\e - \vc w_\e)\|_{L^2((0,T) \times \O_\e)}^2,
\nn \ea
{\tc where we used the same reasoning as in the stationary case (see before \eqref{relE1-stat}), and the fact that $\|\nabla \vc w_\e\|_{L^2(\O_\e)} \leq C \e^{(\a-3)/2} \gg 1$}. Thus, for $\e>0$ {\tc and $\delta > 0$} small enough,
\ba \label{relE1}
&\left[ \int_{\O_\e} E_\e(\uu_\e | \vc w_\e) \dx \right]_{t=0}^{t=\tau} + c \e^{3-\a}  \|\nabla (\uu_\e - \vc w_\e )\|_{L^2((0,T) \times \O_\e)}^2 \\
&\qquad + c \mathbf{1}_{r>2} \e^{(3-\a)(r-1)}  \|\nabla (\uu_\e - \vc w_\e )\|_{L^r((0,T) \times \O_\e)}^r  \\
&\leq \int_0^\tau \int_{\O_\e} \e^{3-\a} \eta_r(\e^{3-\a} D \vc w_\e) D\vc w_\e : (D \vc w_\e - D \uu_\e ) \dx \dt \\
&\qquad + \int_0^\tau \int_{\O_\e} \frac{\eta_0}{2} M_0 \uu \cdot (\uu_\e - \uu) \dx \dt + C \e^{\a-1}. \ea

{\tc The remaining steps are very similar to the stationary case. In turn, we just sketch them.} Using the definition of $\eta_r$, we split as before
\ba
&\int_0^\tau \int_{\O_\e} \e^{3-\a} \eta_r(\e^{3-\a} D \vc w_\e) D\vc w_\e : D( \vc w_\e - \uu_\e) \dx \dt \\
&= \e^{3-\a} \int_0^\tau \int_{\O_\e} \frac{\eta_0}{2} (- \Delta (W_\e \uu) + \nabla (Q_\e \cdot \uu)) \cdot (\vc w_\e - \uu_\e) \dx \dt \\
&\qquad - \e^{3-\a} \int_0^\tau \int_{\O_\e} {\tc \frac{\eta_0}{2} } \nabla (Q_\e \cdot \uu) \cdot (\vc w_\e - \uu_\e) \dx \dt \\
&\qquad + \e^{3-\a} \int_0^\tau \int_{\O_\e} g_r(|\e^{3-\a} D \vc w_\e|) D \vc w_\e : D(\vc w_\e - \uu_\e) \dx \dt {\tc - \int_0^\tau \int_{\O_\e} \eta_0 D B_\e(\uu) : D(\vc w_\e - \uu_\e) \dx \dt}.
\nn \ea

{\tc As for the first integral, we see
\ba
&- \Delta (W_\e \uu) + \nabla (Q_\e \cdot \uu) = (-\Delta W_\e + \nabla Q_\e) \uu + \vc z_\e,\\
&\vc z_\e = (\Delta W_\e) \uu  - \Delta (W_\e \uu) + (Q_\e \cdot \nabla) \uu , \qquad \|\vc z_\e\|_{L^2((0,T) \times \O_\e)} \leq C \e^\frac{\a - 3}{2}.
\nn \ea

For the second integral, the divergence free condition for $\vc w_\e$ and $\vu_\e$ gives 
\ba
- \e^{3-\a} \int_0^\tau \int_{\O_\e} {\tc \frac{\eta_0}{2} } \nabla (Q_\e \cdot \uu) \cdot (\vc w_\e - \uu_\e) \dx \dt = 0.
\nn
\ea

From {\tc Proposition~\ref{lem-local-1}} we have
\ba
\|\e^{3-\a}( -\Delta W_\e + \nabla Q_\e) - M_0\|_{W^{-1,2}(\O_\e)} \leq C \e.
\nn\ea
Hence, we may write
\ba
&\e^{3-\a} \int_0^\tau \int_{\O_\e} \frac{\eta_0}{2} (- \Delta (W_\e \uu) + \nabla (Q_\e \cdot \uu)) \cdot (\vc w_\e - \uu_\e) \dx \dt \\
&= \int_0^\tau \int_{\O_\e} \frac{\eta_0}{2} [\e^{3-\a}(-\Delta W_\e + \nabla Q_\e) - M_0] \uu \cdot (\vc w_\e - \uu_\e) \dx \dt + \int_0^\tau \int_{\O_\e} \frac{\eta_0}{2} M_0 \uu \cdot (\vc w_\e - \uu_\e) \dx \dt \\
&\qquad + {\tc \e^{3-\a} } \int_0^\tau \int_{\O_\e} \vc z_\e \cdot (\vc w_\e - \uu_\e) \dx \dt
\nn\ea
with
\ba
{\tc \e^{3-\a} } \int_0^\tau \int_{\O_\e} \vc z_\e \cdot (\vc w_\e - \uu_\e) \dx \dt & \leq C \e^{3-\a} \|\vc z_\e\|_{L^2((0,T) \times \O_\e)} \|\vc w_\e - \uu_\e\|_{L^2((0,T) \times \O_\e)}  \\
& \leq C \e^{3-\a} \|\nabla(\vc w_\e - \uu_\e)\|_{L^2((0,T) \times \O_\e)}\\
& \leq C_\de {\tc  \e^{3-\a}  + \delta \e^{3-\a} \|\nabla(\vc w_\e - \uu_\e)\|_{L^2((0,\tau) \times \O_\e)}^2},
\nn\ea
and
\ba
&\int_0^\tau \int_{\O_\e} \frac{\eta_0}{2} [\e^{3-\a}(-\Delta W_\e + \nabla Q_\e) - M_0] \uu \cdot (\vc w_\e - \uu_\e) \dx \dt \\
& \leq C \e \|\nabla(\vc w_\e - \uu_\e)\|_{L^2((0,T) \times \O_\e)} \\
&\leq C_\delta \e^{\a-1} + \delta \e^{3-\a} \|\nabla(\vc w_\e - \uu_\e)\|_{L^2((0,T) \times \O_\e)}^2,
\nn \ea
where the last term in either of the above two estimates can be absorbed by the dissipation.} Hence, we deduce {\tc from \eqref{relE1} that}
\ba
&\left[ \int_{\O_\e} E_\e(\uu_\e | \vc w_\e) \dx \right]_{t=0}^{t=\tau} +  c \e^{3-\a}  \|\nabla (\uu_\e - \vc w_\e )\|_{L^2((0,T) \times \O_\e)}^2 \\
&\qquad + c \mathbf{1}_{r>2} \e^{(3-\a)(r-1)}  \|\nabla (\uu_\e - \vc w_\e )\|_{L^r((0,T) \times \O_\e)}^r \\
&\leq C \e^{3-\a} \int_0^\tau \int_{\O_\e} g_r(|\e^{3-\a}D \vc w_\e|) D \vc w_\e : D(\vc w_\e - \uu_\e) \dx \dt + C \big( \e^{\a-1} + \e^{3-\a} \big). 
 \nn \ea

{\tc Using \eqref{growthG} and \eqref{est.w_e-stat}, we deduce the same way as in Section~\ref{sec:statQuan1} that
\begin{align*}
\e^{3-\a} \int_0^\tau \int_{\O_\e} g_r(|\e^{3-\a} D \vc w_\e|) D \vc w_\e : D(\vc w_\e - \uu_\e) \leq \delta \e^{3 - \a} \|D(\vc w_\e - \uu_\e)\|_{L^2((0,T) \times \O_\e)}^2 + C_\delta \e^{2(3-2\a)}.
\end{align*}
}

In turn, choosing $\delta>0$ small enough,
we arrive at{\tc
\ba
&\left[ \int_{\O_\e} E_\e(\uu_\e | \vc w_\e) \dx \right]_{t=0}^{t=\tau}  \\
& \qquad +  c \e^{3-\a}  \|\nabla (\uu_\e - \vc w_\e )\|_{L^2((0,T) \times \O_\e)}^2 + c \mathbf{1}_{r>2} \e^{(3-\a)(r-1)}  \|\nabla (\uu_\e - \vc w_\e )\|_{L^r((0,T) \times \O_\e)}^r \\
&\leq C \Big( \e^{\a-1} + \e^{3-\a} + \e^{2(3-2\a)} \Big)  .
\nn \ea
}

To get the final inequality \eqref{ConvSpeed-t}, it is enough to see that
\begin{align*}
&\|\tilde{\uu}_\e - \uu \|_{L^2((0,T) \times \O)}^2 \leq C \big( \|\uu_\e - \vc w_\e\|_{L^2((0,T) \times \O_\e)}^2 + \|\vc w_\e - \uu \|_{L^2((0,T) \times \O)}^2 \big) \\
&\leq C \e^{3-\a} \|\nabla( \uu_\e - \vc w_\e)\|_{L^2((0,T) \times \O_\e)}^2 + C \e^{\a-1} \\
&\leq C \int_0^\tau \int_{\O_\e} \e^{3-\a} [ \eta_r(\e^{3-\a} D\uu_\e )D\uu_{\e} -  \eta_r(\e^{3-\a} D\vc w_\e )D\vc w_\e ] : D(\uu_\e - \vc w_\e ) \dx \dt + C \e^{\a-1} \\
&\leq \int_{\O_\e} E_\e(\uu_\e | \vc w_\e)(0) \dx + C \Big(  \e^{\a-1} + \e^{3-\a} + {\tc \e^{2(3-2\a)} } \Big),
\end{align*}
as well as
\begin{align*}
\int_{\O_\e} E_\e(\uu_\e | \vc w_\e)(0) \dx &= \int_{\O} \frac12 \e^\l |\tilde{\uu}_{\e 0} - \vc w_\e(0,\cdot)|^2 \dx \\
&\leq C \e^\l \|\tilde{\uu}_{\e 0} - \uu (0,\cdot)\|_{L^2(\O)}^2 + C \e^\l \|(W_\e - \mathbb{I}) \uu (0,\cdot)\|_{L^2(\O)}^2 \\
&\leq C \e^\l \|\tilde{\uu}_{\e 0} - \uu (0,\cdot)\|_{L^2(\O)}^2 + C  \e^{\l} \| \uu (0,\cdot)\|_{L^2(\O)}^2 \|(W_\e - \mathbb{I}) \|_{L^\infty(\O)}^2\\
&\leq C \e^\l \|\tilde{\uu}_{\e 0} - \uu (0,\cdot)\|_{L^2(\O)}^2 + C \e^{\l}.
\end{align*}
{\tc Seeing that due to $\l > \a$, we have $\e^\l < \e^{\a-1}$,} this ends the proof of Theorem~\ref{thm-3-evol}.

\subsection{Adaptions for bounded domains}\label{sec:evolBdDom}
The analogue for the stationary case Theorem~\ref{thm-3} reads for the evolutionary system as follows:
\begin{theorem}\label{thm-3-ev-bd}
{\tc Let $\Omega$ be a bounded domain of class $C^{3,\mu}$, $\mu\in (0,1)$.} Let
\begin{align*}
r>1, && 1<\a<\frac 32, && \l > \a.
\end{align*}
{\tc Let $\eta_r$ comply with \eqref{monG}--\eqref{growthG}}.  Let $(\uu_\e, p_\e)$ be a weak solution to \eqref{NonNewTime} emanating from the initial datum $\uu_{\e 0} \in L^2(\O_\e)$ satisfying \eqref{init}, and let $(\uu, p) \in   [ W^{1,\infty}(0,T;W^{1,\infty}(\O)) \cap L^\infty(0,T;W^{2,2}(\O)) {\vc \cap L^\infty(0,T;C^{1,\mu}(\O))} ] \times L^\infty(0,T;W^{1,\infty}(\O))$ with $\dive \uu = 0$ be a strong solution to Darcy's law \eqref{Darcy-t} with initial value $ \|\uu(0, \cdot) \|_{L^{2}(\O)} \leq C$. Then, there exists an $\e_0 > 0$ such that for all $\e \in (0, \e_0)$ , we have
\ba
\|\tilde{\uu}_\e - \uu \|_{L^2((0,T) \times \O)}^2 & \leq C \Big( \e^\l \|\tilde{\uu}_{\e 0} - \uu(0, \cdot)\|_{L^2(\O)}^2 + \e^{\a-1} + {\tc \e^\frac{3-\a}{2} + \e^{2(3-2\a)} } \Big), \nn
\ea
where the constant $C>0$ is independent of $\e$. The last term in \eqref{ConvSpeed-t} can be taken to be zero if {\tc $g_r \equiv 0$, and in this case, we can cover all $1 < \a < 3$}.
 \end{theorem}

{\tc The proof of Theorem~\ref{thm-3-ev-bd} follows the very same lines as in the stationary case. Since no new ideas appear here, we leave the details to the reader.}

\section{\tc Pressure convergence}\label{sec:pressure}
Having shown the convergence of velocities, we want to have a look what happens to the pressures. 
Recall that \eqref{NonNew} is given by
\ba\label{pe-1}
 \nabla p_\e = \e^{3-\a} \dive \big(\eta_r(\e^{3-\a} D\uu_\e )D\uu_\e  \big) - \e^\l\dive (\uu_\e \otimes \uu_\e)   +  \mathbf{f}, \qquad  {\rm in}\ \mathcal{D}'(\O_\e).
\ea
Similar to the convergence of velocities, we will show the following results:
\begin{theorem}[Stationary case]
    Under the assumptions of Theorem~\ref{thm-1}, there exists an associated pressure $p_\e \in L^2(\Omega_\e)$, such that $(\uu_\e,p_\e)$ is a distributional solution to \eqref{NonNew}. The extension of the pressure can be decomposed as $\tilde p_\e = \tilde p_\e^{(1)} + \tilde p_\e^{\rm res}$, where $\tilde p_\e^{(1)} \weak p$ weakly in $L^2(\O)$, and $\tilde p_\e^{\rm res} \to 0$ strongly in $L^q(\O)$ for some $q>1$. The limit $p$ is the pressure in Darcy's law \eqref{Darcy}. Furthermore, under the assumptions of Theorem~\ref{thm-4}, we have
    \begin{align}\label{tlak}
        \|\tilde{p}_\e - p\|_{L^{1}(\O)} \leq C (\e^\frac{\a-1}{2} + \e^\frac{3-\a}{2} + \e^{(\l-\a)^-} + \e^{(3-2\a)^-} + \e^\frac{(3-2\a)(r-2)}{r} \mathbf{1}_{r>2}).
    \end{align}
    Here, we denote by $s^-$ any number smaller than but arbitrarily close to $s$. Lastly, the last to terms in \eqref{tlak} are taken to be zero if $g_r \equiv 0$.
\end{theorem}

\begin{theorem}[Evolutionary case]
    Under the assumptions of Theorem~\ref{thm-2}, there exists a function $P_\e \in C(0,T; L_0^2(\O_\e))$ such that $(\uu_\e,\partial_t P_\e)$ is a distributional solution to \eqref{NonNewTime}. Moreover, its extension can be decomposed as $\tilde P_\e = \tilde P_\e^{(1)} + \tilde P_\e^{\rm res}$, where $\tilde P_\e^{(1)} \weak^\ast P$ weakly-$\ast$ in $L^\infty(0,T; L^2(\O))$, and $\tilde P_\e^{\rm res} \to 0$ strongly in $L^\infty(0,T; L^q(\O))$ for some $q>1$. Moreover, the limit $p = \partial_t P$ is the pressure in Darcy's law \eqref{Darcy-t}. Furthermore, under the assumptions of Theorem~\ref{thm-3-evol}, we have
    \begin{align}\label{tlak.evol}
        \|\partial_t \tilde{P}_\e - p\|_{W^{-1,2}(0,T; L^{1}(\O))} \leq C (\e^\frac{\a-1}{2} + \e^\frac{3-\a}{2} + \e^{(\l-\a)^-} + \e^{(3-2\a)^-} + \e^\frac{(3-2\a)(r-2)}{r} \mathbf{1}_{r>2}).
    \end{align}
    Here, we denote by $s^-$ any number smaller than but arbitrarily close to $s$. Lastly, the last to terms in \eqref{tlak} are taken to be zero if $g_r \equiv 0$.
\end{theorem}

\begin{remark}[Quantitative estimates in bounded domains]
    Under the assumptions of Theorem~\ref{thm-3} and \ref{thm-3-ev-bd}, respectively, the above statements remain valid up to replacing  $\e^\frac{3-\a}{2}$ by $\e^\frac{3-\a}{4}$ in \eqref{tlak} and \eqref{tlak.evol}, respectively.
    The proof of these adaptations to bounded domains follows the lines of Section \ref{sec:statBdDom} and will be skipped.
\end{remark}

\begin{remark}
    In \eqref{tlak} and \eqref{tlak.evol}, the space $L^1(\Omega)$ can be replaced by $L^s(\Omega)$ for some $s>1$ in the following sense: For any $\delta > 0$, there exists $s > 1$ such that the estimate holds upon setting $(\lambda - \alpha)- = \lambda - \alpha - \delta$ and $(3 - 2 \alpha)- =   3 - 2 \alpha - \delta$. The exact value of $s$ can be taken from the proof in Section~\ref{sec:pQuan}.
\end{remark}

The following sections are therefore devoted to uniform bounds, qualitative and quantitative convergences, respectively, for both stationary and evolutionary NSE.

\subsection{Uniform bounds}\label{sec:pressBds}
\subsubsection{Stationary case}\label{sec:pBdsStat}
{\tc A suitable  way to extend the pressure} to the whole of $\O$ is by duality as it was given by Allaire in \cite{All90-1} for the case $q=2$, and then generalized by the results of \cite{Lu21} to the range $\frac{3}{2} < q < 3$. Such a duality argument applies also in our case, however, due to the restriction $q>\frac{3}{2}$, we would get a worse range for $\l$. To overcome this drawback, we employ the Bogovski\u{\i} operator given in Proposition~\ref{Bog-op} to show directly the estimates of $p_{\e}$ for $q \approx 1$ (but still $q>1$).

Given any $\phi \in C_{c}^{\infty}(\O_{\e})$, we apply the Bogovski\u{\i} operator in Proposition \ref{Bog-op} to define 
\ba\label{Bog-phi}
\Phi_\eps = \mathcal{B}_\e (\phi - \langle \phi\rangle_{\O_{\e}} ), 
\ea
where the notation $\langle \phi\rangle_{\O_{\e}} $ stands for the average of $\phi$ on $\O_{\e}$:
 $$
 \langle \phi\rangle_{\O_{\e}}  = \frac{1}{|\O_{\e}|}\int_{\O_{\e}} \phi \, \dx.
 $$
Clearly, by Proposition \ref{Bog-op}, $\Phi_\eps \in W_{0}^{1,q}(\O_{\e};\R^{3})$ for any $1<q<\infty$ with estimates
\ba\label{Bog-phi-est}
\|\Phi_\eps\|_{W^{1,q}_0(\O_\e)} \leq C \big(1 + \e^{\frac{(3-q)\a-3}{q}}\big) \|\phi\|_{L^q(\O_\e)},  \mbox{ for all } \ 1<q<\infty.
\ea
The idea is to use $\Phi_\eps$ as a test function in \eqref{pe-1} to derive the estimates of $p_{\e}$. Notice that $p_{\e}$ is of zero average, so there holds
\ba\label{pe-2}
\langle  \nabla p_\e, \Phi_\eps\rangle_{\O_{\e}} = - \langle p_\e, \dive \Phi_\eps\rangle_{\O_{\e}}  = - \langle p_\e,  \phi - \langle \phi\rangle_{\O_{\e}} \rangle_{\O_{\e}}  =  - \langle p_\e,  \phi   \rangle_{\O_{\e}} .
\ea
As a result, 
\ba\label{pe-3}
 \langle p_\e,  \phi   \rangle_{\O_{\e}}  & =  - \langle  \nabla p_\e, \Phi_\eps\rangle_{\O_{\e}} \\
 & =  - \langle  \e^{3-\a} \dive \big(\eta_r(\e^{3-\a} D\uu_\e )D\uu_\e  \big) - \e^\l\dive (\uu_\e \otimes \uu_\e)   +  \mathbf{f}, \Phi_\eps\rangle_{\O_{\e}} \\
 & =  \e^{3-\a} \langle \eta_r(\e^{3-\a} D\uu_\e )D\uu_\e , \nabla \Phi_\eps\rangle_{\O_{\e}} - \e^\l \langle \uu_\e \otimes \uu_\e  , \nabla \Phi_\eps\rangle_{\O_{\e}}  - \langle   \mathbf{f}, \Phi_\eps\rangle_{\O_{\e}} \\
 &= \e^{3-\a} \langle \eta_0 D\uu_\e , \nabla \Phi_\eps\rangle_{\O_{\e}} + \e^{3-\a} \langle g_r(|\e^{3-\a} D\uu_\e|) D\uu_\e , \nabla \Phi_\eps\rangle_{\O_{\e}} \\
 &\qquad - \e^\l \langle \uu_\e \otimes \uu_\e  , \nabla \Phi_\eps\rangle_{\O_{\e}}  - \langle   \mathbf{f}, \Phi_\eps\rangle_{\O_{\e}}.
\ea

Clearly, if $g_r \equiv 0$, the second term vanishes, which will give us the full range $1 < \a < 3$ instead of $1 < \a < \frac32$ similarly as before. We now show the estimates of the right-hand side of \eqref{pe-3} term by term. Recalling the uniform estimates of $\uu_{\e}$ in \eqref{est-u-f} as
\ba
\e^{\frac{3-\a}{2}} \|\nabla \tilde \uu_{\e}\|_{L^2(\O)}  \leq C, \qquad \| \tilde \uu_{\e}\|_{L^2(\O )} \leq C , \qquad \e^{\frac{(3-\a)(r-1)}{r}} \|\nabla \tilde \uu_{\e}\|_{L^{r}(\O)} \mathbf{1}_{r>2} \leq  C , \nn
\ea
the first term on the right-hand side of \eqref{pe-3} satisfies
\ba\label{pe-5}
  \e^{3-\a} | \langle \eta_0 D\uu_\e , \nabla \Phi_\eps\rangle_{\O_{\e}} | & \leq  C   \e^{3-\a} \| D\uu_\e \|_{L^{2}(\O_{\e})} \| \nabla \Phi_\eps \|_{L^{2}(\O_{\e})}  \\\
  & \leq C  \e^{\frac{3-\a}{2}} \| \nabla \Phi_\eps \|_{L^{2}(\O_{\e})}.
\ea
Moreover, from \eqref{Bog-phi-est}, we have
\ba\label{pe-6}
 \| \nabla \Phi_\eps \|_{L^{2}(\O_{\e})}  \leq C (1 + \e^{\frac{\a - 3}{2} } )\| \phi \|_{L^{2}(\O_{\e})} .
\ea
From \eqref{pe-5} and \eqref{pe-6} we deduce
\ba\label{pe-f1}
 \e^{3-\a} | \langle \eta_0 D\uu_\e , \nabla \Phi_\eps\rangle_{\O_{\e}} |  \leq C   \|  \phi \|_{L^{2}(\O_{\e})}.
\ea

\medskip

{\tc
If $1<r<2$, similarly to \eqref{limit-11} the second term in \eqref{pe-3} is estimated as
\begin{align*}
    \e^{3-\a} \langle g_r(|\e^{3-\a} D\uu_\e|) D\uu_\e , \nabla \Phi_\eps\rangle_{\O_{\e}} &\leq C \e^{(3-\a)(1+\kappa)} \e^{-\frac{3-\a}{2}(1+\kappa)} \|\nabla \Phi_\eps \|_{L^\frac{2}{1-\kappa}(\O_\e)} \\
    &\leq C \e^{\frac12 (3-\a)(1+\kappa)} \Big( 1 + \e^\frac{(3(1-\kappa) - 2)\a - 3(1-\kappa)}{2} \Big) \|\phi\|_{L^\frac{2}{1-\kappa}(\O_\e)} \\
    &\leq C \e^{\frac12 (3-\a)(1+\kappa)} \e^{-\frac{3-\a + 3\kappa (\a-1)}{2}} \|\phi\|_{L^\frac{2}{1-\kappa}(\O_\e)} \\
    &\leq C \e^{(3-2\a) \kappa} \|\phi\|_{L^\frac{2}{1-\kappa}(\O_\e)},
\end{align*}
for any $\kappa \in (0,1)$.

\medskip

If $r>2$, then similarly to \eqref{limit-12}
\begin{align*}
    \e^{3-\a} \langle g_r(|\e^{3-\a} D\uu_\e|) D\uu_\e , \nabla \Phi_\eps\rangle_{\O_{\e}} &\leq C \e^{(3-2\a)\kappa} \|\phi\|_{L^\frac{2}{1-\kappa}(\O_\e)} + C \e^\frac{(3-\a)(r-1)}{r} \|\nabla \Phi_\eps\|_{L^r(\O_\e)} \\
    &\leq C \e^{(3-2\a)\kappa} \|\phi\|_{L^\frac{2}{1-\kappa}(\O_\e)} + C \e^\frac{(3-\a)(r-1)}{r} \e^\frac{(3-r)\a-3}{r} \|\phi\|_{L^r(\O_\e)} \\
    &\leq C \e^{(3-2\a)\kappa} \|\phi\|_{L^\frac{2}{1-\kappa}(\O_\e)} + C \e^\frac{(3-2\a)(r-2)}{r} \|\phi\|_{L^r(\O_\e)},
\end{align*}
for any $\kappa \in (0,1)$, where we also used that for any $r>2$ and any $1<\a<3$, we have $(3-r)\a-3 < 0$.
}

\medskip

For the third term on the right-hand side of \eqref{pe-3}, by interpolation for any $\th \in (0,1)$,
\ba
\|\uu_\e \otimes \uu_{\e}\|_{L^\frac{3}{3-2\th}(\O_\e)} \leq \|\uu_\e \|_{L^\frac{6}{3-2\th}(\O_\e)}^2  \leq C \| \uu_\e\|_{L^2(\O_\e)}^{2(1-\th)}  \| \uu_\e\|_{L^6(\O_\e)}^{2\th} \leq C \e^{-\th(3-\a)}. \nn
\ea
 Therefore, choosing without loss of generality $\th < \frac34$, 
\ba\label{pe-f2}
\e^\l \left|\langle \uu_{\e}\otimes \uu_{\e}  ,  \nabla  \Phi_\eps \rangle_{\O_\e} \right| &  \leq \e^\l \|\uu_{\e}\|_{L^\frac{6}{3-2\th}(\O_\e)}^{2} \|     \nabla  \Phi_\eps \|_{L^\frac{3}{2\th}(\O_\e)}  \\
&   \leq C \e^{\l -\th(3-\a) } \|  \nabla \Phi_\eps\|_{L^{\frac{3}{2\th}}(\O_\e)} \\
&   \leq C \e^{\l -\th(3-\a) } (1 + \e^\frac{(6\th - 3)\a - 6\th}{3} )\|  \phi\|_{L^{\frac{3}{2\th}}(\O_\e)} \\
& \leq C \e^{\l - \a - \th (5 - 3\a)} \|  \phi\|_{L^{\frac{3}{2\th}}(\O_\e)} .
\ea
Note that always $\frac{3}{2\theta} \geq \frac{3}{2}$, so this is precisely the place and reason why we do not use the pressure extension by duality as done in \cite{All90-1, Lu21}\footnote{\tc Moreover, the nonlinear term in \cite{All90-1} is treated by abstract compactness arguments, which will not yield convergence rates as we show later on.}. Since $\l>\a$, we can always choose $\th>0$ small enough, for example
\ba\label{theta}
\th  =  \min \left\{ \frac{\l - \a}{2 (5-3\a)}, \frac12 \right\} {\tc \text{ if } \a < \frac53, \qquad \th \in (0, \frac34) \text{ arbitrary if } \a \geq \frac53,  }
\ea
 such that  
 \ba\label{theta1}
 \l - \a - \th (5 - 3\a)  \geq   \frac{\l - \a }{2} >0. 
 \ea

\medskip

For the last term in \eqref{pe-3}, we find by applying the Poincar\'e inequality in Lemma~\ref{lem-Poincare} that
\ba\label{pe-f3}
\left|\langle  \mathbf{f} , \Phi_\eps  \rangle_{\O_{\e}} \right|  &  \leq \| \ff \|_{L^{2}(\O)}  \|   \Phi_\eps \|_{L^{2}(\O_\e)}  \\
&   \leq C \| \ff \|_{L^{2}(\O)} \e^{\frac{3-\a}{2} }   \|  \nabla  \Phi_\eps  \|_{L^{2}(\O_\e)}  \\
& \leq C \| \phi\|_{L^2(\O)}. 
\ea

\medskip

Plugging the estimates \eqref{pe-f1}, \eqref{pe-f2}, and \eqref{pe-f3} into \eqref{pe-3} implies for any $\kappa \in (0,1)$
\ba\label{pe-f}
 | \langle p_\e,  \phi   \rangle_{\O_{\e}} |  &  \leq  C \|  \phi \|_{L^{2}(\O_{\e})} + C \e^{\l - \a - \th (5 - 3\a)} \|  \phi\|_{L^{\frac{3}{2\th}}(\O_\e)} \\
 &\qquad + C \e^{(3-2\a)\kappa} \|\phi\|_{L^\frac{2}{1-\kappa}(\O_\e)} + C \e^\frac{(3-2\a)(r-2)}{r} \|\phi\|_{L^{r}(\O_\e)} \mathbf{1}_{r>2} . \ea
Using \cite[Lemma~3.5]{HNO24}, this means we can decompose $p_{\e} $ as 
\ba\label{pe-dec}
& p_{\e} = p_\e^{(1)} +   p_\e^{\rm res}, \qquad  p_\e^{\rm res} =  \e^{\l - \a - \th (5 - 3\a)} p_\e^{(2)} + \e^{(3-2\a)\kappa} p_\e^{(3)} + \e^\frac{(3-2\a)(r-2)}{r} p_\e^{(4)} \mathbf{1}_{r>2} , \\
&  \|  p_\e^{(1)} \|_{L^{2}(\O_{\e})}  + \|  p_\e^{(2)} \|_{L^{\frac{3}{3-2\th}}(\O_{\e})} + \|p_\e^{(3)}\|_{L^\frac{2}{1+\kappa}(\O_\e)} + \|p_\e^{(4)}\|_{L^\frac{r}{r-1}(\O_\e)} \mathbf{1}_{r>2}  \leq C. 
\ea
Here we shall choose $\th>0$ small (see \eqref{theta}) such that \eqref{theta1} is satisfied\footnote{\tc For the Newtonian case $g_r \equiv 0$, it is rather straightforward to get the decomposition \eqref{pe-dec} with $p_\e^{(3)} = p_\e^{(4)}  = 0$, and even also $p_\e^{(2)} =0$ but at the cost of higher $\lambda$ if $\a$ is close to $1$.}.

Let $\tilde p_{\e}$ ($ \tilde  p_\e^{(1)}, \ \tilde p_\e^{\rm res}$) be the zero extension of  $  p_{\e}$ ($  p_\e^{(1)}, \ p_\e^{\rm res}$). Consequently, we may always split $\tilde p_\e = \tilde p_\e^{(1)} + \tilde p_\e^{\rm res}$ with
\ba\label{p-dec-1}
\tilde p_\e^{(1)} &\weak p \text{ weakly in }   L^2(\O),\\
\|\tilde p_\e^{\rm res}\|_{L^q(\O)} & \leq C \e^{\s}\   \mbox{for some $q = q(r, \l, \a, \kappa)>1$, $\s = \s(r, \l, \a, \kappa)>0$}.
\ea

\subsubsection{Evolutionary case}\label{sec:pBdsEvol}
The estimates of the pressure is more delicate in the evolutionary case.  To recover the pressure from the equation, the idea is to integrate the momentum equation in time. Let $\uu_{\e}$ be a finite energy weak solution of \eqref{NonNewTime} in the sense of Definition \ref{def-weak-time}. Introduce
\begin{equation*}
\begin{gathered}
  \mathbf{U}_\e = \int_0^t\uu_\e\, {\rm d}s, \qquad \mathbf{G}_\e=\int_0^t  (\uu_\e\otimes \uu_\e)\, {\rm d}s, \qquad {\tc \mathbf{H}_\e=\int_0^t g_r(|\e^{3-\a} D\uu_\e|) D\uu_\e \, {\rm d}s, } \\
  {\tc P_\e = \int_0^t p_\e \, {\rm d}s, } \qquad \mathbf{F}=\int_0^t\mathbf{f} \, {\rm d}s.
\end{gathered}
\end{equation*}
It follows from \eqref{est-u-f-t} that  $$\mathbf{U}_\e\in C([0,T]; W_{0, \rm div}^{1,2} (\O_\e)), \qquad \mathbf{G}_\e\in C([0,T]; L^{3}(\O_\e)), \qquad \mathbf{F}\in C([0,T]; L^2(\O_\e)), $$ and by \eqref{growthG}, we have
\begin{equation*}
\mathbf{H}_\e\in
\begin{cases}
C([0,T];L^{2}(\O_\e)) & \text{if } 1<r\leq2,\\
C([0,T]; (L^2 + L^{\frac{r}{r-1}})(\O_\e)) & \text{if } r > 2.
\end{cases}
\end{equation*}
Moreover, it follows from \eqref{est-u-f-t-2} that 
\ba
  \tilde{\mathbf{U}}_\e  \weak \mathbf{U} = \int_{0}^{t} \uu \,{\rm d}s  \text{ weakly in } L^2((0,T) \times \O). \nn
\ea

Additionally, we see that for {\tc all $t \in [0,T]$,}
\begin{equation*}
  \nabla P_\e=\mathbf{F}-\e^\l (\uu_\e-\uu_{\e0}) - \e^{\l} \dive \mathbf{G}_\e +\e^{3-\a}\frac{\eta_0}{2}\Delta \mathbf{U}_\e +\e^{3-\a} {\rm div}\,\mathbf{H}_\e \qquad \mbox{in}  \ \mathcal{D}'(\O_{\e}).
\end{equation*}
Exactly along the lines in the previous section, with the estimates \eqref{est-u-f-t0} and \eqref{est-u-f-t} at hand, we can derive the uniform bounds of $P_{\e}$ as
\ba\label{Pe-dec}
& P_{\e} = P_\e^{(1)} + P_\e^{\rm res}, \qquad  P_\e^{\rm res} =  \e^{\l - \a - \th (5 - 3\a)} P_\e^{(2)} + \e^{(3-2\a)\kappa} P_\e^{(3)} + \e^\frac{(3-2\a)(r-2)}{r} P_\e^{(4)} \mathbf{1}_{r>2} , \\
&  \| P_\e^{(1)} \|_{L^\infty(0,T; L^{2}(\O_{\e}))}  + \| P_\e^{(2)} \|_{L^\infty(0,T; L^{\frac{3}{3-2\th}}(\O_{\e}))} \\
&\qquad + \|P_\e^{(3)}\|_{L^\infty(0,T; L^\frac{2}{1+\kappa}(\O_\e))} + \|P_\e^{(4)}\|_{L^\infty(0,T; L^\frac{r}{r-1}(\O_\e))} \mathbf{1}_{r > 2} \leq C. 
\ea
Here in \eqref{Pe-dec}, $\th>0$ is a small number and will be chosen such that $\l - \a - \th (5 - 3\a)>0$, which is always possible due to $\l>\a.$ 

\medskip

 Hence, in any case, we may split the zero extension $\tilde P_\e = \tilde P_\e^{(1)} + \tilde P_\e^{\rm res}$ with
\ba
& \tilde P_\e^{(1)} \weak P \text{ weakly* in }  L^{\infty}(0,T; L^2(\O)),\\
&\|\tilde P_\e^{\rm res}\|_{L^{\infty}(0,T; L^q(\O))}  \leq C \e^{\s}   \   \mbox{for some $q>1, \ \s>0$}. \nn
\ea

\begin{remark}
    Similar decompositions as given in \eqref{pe-dec} and \eqref{Pe-dec} are obtained in \cite[Lemma~3.2]{HNO24}.
\end{remark}

\subsection{Convergence}\label{sec:pConv}
In this section, we will show that the pressure $p_\e$ converges to its counterpart $p$ in Darcy's law \eqref{Darcy}. We will just focus on the steady case since the time-dependent case follows by similar arguments.

\subsubsection{Qualitative proof}\label{sec:pQual}
From \eqref{p-dec-1}, we see that there exists a $p \in L^2(\O)$ such that
\begin{align*}
    \tilde p_\e^{(1)} &\weak p \text{ weakly in } L^2(\O),\\
    \tilde p_\e^{\rm res} &\to 0 \text{ strongly in } L^q(\O) \text{ for some } q>1.
\end{align*}

We will show in the sequel that this $p$ is actually the unique pressure in $L_0^2(\O)$ that solves the Darcy law \eqref{Darcy}.\\

Taking for $\phi \in C_c^\infty(\O; \R^3)$ the test function $W_\e \phi$ in \eqref{pe-1}, we infer
\begin{align}\label{jedna}
    \int_{\O_\e} p_\e \dive (W_\e \phi) \dx = \int_{\O_\e} \e^{3-\a} \eta_r(\e^{3-\a} D\uu_\e)D\uu_\e : \nabla (W_\e \phi) + \e^\l \uu_\e \otimes \uu_\e : \nabla (W_\e \phi) - \ff \cdot W_\e \phi \dx.
\end{align}
The same way as in Section~\ref{sec:statQual}, we deduce that the right hand side of the above equality converges to
\begin{align*}
    \int_\O \frac{\eta_0}{2} M_0 \uu \cdot \phi - \ff \cdot \phi \dx.
\end{align*}
As for the left hand side, we find
\begin{align*}
    \int_{\O_\e} p_\e \dive(W_\e \phi) \dx = \int_\O \tilde p_\e^{(1)} W_\e : \nabla \phi \dx + \int_\O \tilde p_\e^{\rm res} W_\e : \nabla \phi \dx,
\end{align*}
where we used that $\dive W_\e = 0$. Since $W_\e \to \Id$ strongly in $L^2(\O)$, $\|W_\e\|_{L^\infty(\O)} \leq C$, and $\tilde p_\e^{(1)} \weak p$ weakly in $L^2(\O)$ and $\tilde p_\e^{\rm res} \to 0$ strongly in some $L^q(\O)$, we find
\begin{align*}
    \int_\O \tilde p_\e^{(1)} W_\e : \nabla \phi \dx &\to \int_\O p \Id : \nabla \phi \dx = \int_\O p \dive \phi \dx,\\
 \int_\O \tilde p_\e^{\rm res} W_\e : \nabla \phi \dx &\leq C \|\tilde p_\e^{\rm res}\|_{L^q(\O)} \|\nabla \phi\|_{L^{q'}(\O)} \to 0.
\end{align*}

Hence, we get
\begin{align*}
    \int_\O p \dive \phi \dx = \int_\O \frac{\eta_0}{2} M_0 \uu \cdot \phi - \ff \cdot \phi \dx,
\end{align*}
which is the weak formulation of \eqref{Darcy}. Due to uniqueness of the solution to \eqref{Darcy}, we infer that $p = \lim_{\e \to 0} \tilde p_\e$ is the correct Darcy pressure, completing the proof of Theorem~\ref{thm-1}.

\subsubsection{Quantitative proof}\label{sec:pQuan}
Lastly, quantitative convergence rates for the torus case can be obtained as follows. Let $(\uu, p)$ be a strong solution to Darcy's law \eqref{Darcy}. Since $(\uu, p)$ is a strong solution, it satisfies \eqref{Darcy} pointwise in $\O$ and especially pointwise in $\O_\e$. Let $\phi \in C_c^\infty(\O; \R^3)$, use $W_\e \phi$ as test function in \eqref{Darcy}, and subtract the outcome from \eqref{jedna} to get
\ba\label{dvojka}
    \int_{\O_\e} (p_\e - p) \dive (W_\e \phi) \dx &= \int_{\O_\e} \e^{3-\a} \eta_r(\e^{3-\a} D\uu_\e)D\uu_\e : \nabla (W_\e \phi) \dx - \int_{\O_\e} \frac{\eta_0}{2} M_0 \uu \cdot  (W_\e \phi) \dx \\
    &\qquad + \int_{\O_\e} \e^\l \uu_\e \otimes \uu_\e : \nabla (W_\e \phi) \dx \\
    &= \eta_0 \int_{\O_\e} \e^{3-\a} D\uu_\e : \nabla (W_\e \phi) - \frac{\eta_0}{2} M_0 \uu \cdot  (W_\e \phi) \dx \\
    &\qquad + \int_{\O_\e} \e^{3-\a} g_r(|\e^{3-\a} D \uu_\e|) D\uu_\e : \nabla (W_\e \phi) \dx \\
    &\qquad + \int_{\O_\e} \e^\l \uu_\e \otimes \uu_\e : \nabla (W_\e \phi) \dx.
\ea
Arguing as in Section \ref{sec:pBdsStat} and {\tc using the estimate \eqref{nabW}}, the last two integrals in \eqref{dvojka}  give rise to
{\tc 
\ba
    &\int_{\O_\e} \e^{3-\a} g_r(|\e^{3-\a} D \uu_\e|) D\uu_\e : \nabla (W_\e \phi) \dx + \int_{\O_\e} \e^\l \uu_\e \otimes \uu_\e : \nabla (W_\e \phi) \dx \\
    &\lesssim \e^{\frac{3-\a}{2}(1+\kappa)} \|\nabla (W_\e \phi)\|_{L^\frac{2}{1-\kappa}(\O_\e)} + \e^\frac{(3-\a)(r-1)}{r} \|\nabla (W_\e \phi)\|_{L^r(\O_\e)} \mathbf{1}_{r>2} + \e^{\l - \th(3-\a)} \|\nabla (W_\e \phi) \|_{L^\frac{3}{2\th}(\O_\e)} \\
      &\lesssim \e^{\frac{3-\a}{2}(1+\kappa)} (\|\nabla W_\e \|_{L^{\frac{2}{1-\kappa}+}(\O_\e)} \| \phi\|_{L^{\infty-}(\O_\e)} +  \|W_\e \|_{L^{\infty}(\O_\e)} \| \nabla \phi\|_{L^{\frac{2}{1-\kappa}}(\O_\e)}) \\
      & \qquad + \e^\frac{(3-\a)(r-1)}{r} (\|\nabla W_\e \|_{L^{r+}(\O_\e)} \| \phi\|_{L^{\infty-}(\O_\e)} +  \|W_\e \|_{L^{\infty}(\O_\e)} \| \nabla \phi\|_{L^{r}(\O_\e)}) \mathbf{1}_{r>2} \\
      & \qquad + \e^{\l - \th(3-\a)} (\|\nabla W_\e \|_{L^{\frac{3}{2\th}+}(\O_\e)} \| \phi\|_{L^{\infty-}(\O_\e)} +  \|W_\e \|_{L^{\infty}(\O_\e)} \| \nabla \phi\|_{L^{\frac{3}{2\th}}(\O_\e)})  \\
    &\lesssim (\e^{(3-2\a)\kappa - }  + \e^{\frac{(3-2\a)(r-2)}{r}-} \mathbf{1}_{r>2} + \e^{\l - \a - \th(5-3\a)-} )\|\phi\|_{W^{1,\infty-}(\Omega)}\\
    & \lesssim (\e^{(3-2\a)-} + \e^{\frac{(3-2\a)(r-2)}{r}-} \mathbf{1}_{r>2} + \e^{(\l - \a)-}) \|\phi\|_{W^{1,\infty-}(\Omega)},
\nn \ea
}
{\tc  where we have chosen $\th$ sufficiently close to $0$ and $\kappa$ sufficiently close to $1$. Moreover, $s+$ and $s-$ denote numbers arbitrarily close to, but strictly greater than and less than $s$, respectively. (More precisely, there is an implicit dependence between different exponents $a+$ and $s-$, which we choose not to make explicit for notational convenience. However, for any $\delta > 0$, it will always be possible to choose any $a+ \in (a,a+\delta)$ and any $s- \in (s - \delta,s)$.)}

We rewrite the first integral on the right-hand side of \eqref{dvojka}  as
\begin{align*}
    &\eta_0 \int_{\O_\e} \e^{3-\a} D\uu_\e : \nabla (W_\e \phi) - \frac{\eta_0}{2} M_0 \uu \cdot  (W_\e \phi) \dx \\
    &= \eta_0 \int_{\O_\e} \e^{3-\a} D\uu_\e : \nabla (W_\e \phi) - \frac{\eta_0}{2} \uu_\e \cdot (M_0 \phi) \dx + \frac{\eta_0}{2} \int_{\O_\e} (\uu_\e - \uu) \cdot (M_0 \phi) \dx \\
    &\qquad + \frac{\eta_0}{2} \int_{\O_\e} M_0 \uu \cdot (\Id - W_\e) \phi \dx
\end{align*}
to see that, similarly to \eqref{limit-7} and using also \eqref{ConvSpeedS},
{\tc \begin{align*}
    &\eta_0 \int_{\O_\e} \e^{3-\a} D\uu_\e : \nabla (W_\e \phi) - \frac{\eta_0}{2} M_0 \uu \cdot  (W_\e \phi) \dx \\
    & = \frac{\eta_0}{2} \int_{\O_\e} \big[\e^{3-\a} \big( \nabla W_\e :\nabla (\uu_\e \phi) - Q_\e \dive (\uu_\e \phi) \big) -  M_0  (\uu_\e \phi)  \big]\dx   \\
    & \qquad + \frac{\eta_0}{2} \int_{\O_\e}  \e^{3-\a}  \big[\nabla \uu_\e :(W_\e \otimes \nabla \phi) -\nabla W_\e :(\uu_\e \otimes \nabla \phi)  +  Q_\e (\uu_\e \cdot \nabla \phi) \big] \dx \\
    & \qquad  + \frac{\eta_0}{2} \int_{\O_\e} (\uu_\e - \uu) \cdot (M_0 \phi) \dx + \frac{\eta_0}{2} \int_{\O_\e} M_0 \uu \cdot (\Id - W_\e) \phi \dx\\ 
    &\lesssim \|\e^{3-\a} (- \Delta W_\e + \nabla Q_\e) - M_0\|_{W^{-1,2}(\O_\e)} \|\uu_\e \phi\|_{W^{1,2}(\O_\e)} \\
    & \qquad +  \e^{3-\a} \big(\| \nabla \uu_\e\|_{L^2(\Omega_\e)} \|W_\e\|_{L^\infty(\Omega_\e)}\|\nabla \phi\|_{L^2(\Omega_\e)} \big) \\
    & \qquad + \e^{3-\a} \big((\| \nabla W_\e\|_{L^{2+}(\Omega_\e)} + \|Q_\e\|_{L^{2+}(\Omega_\e)}  ) \| \uu_\e\|_{L^{2}(\Omega_\e)}\| \nabla \phi\|_{L^{\infty-}(\Omega_\e)}    \big)\\
    &\qquad + \|\uu_\e - \uu\|_{L^2(\O_\e)}\|\phi\|_{L^2(\O_\e)} + \|\Id - W_\e\|_{L^2(\O_\e)} \|\phi\|_{L^2(\O_\e)}  \|\uu\|_{L^\infty(\O_\e)}  \\
    &\lesssim \e \big( \|\uu_\e\|_{L^6(\Omega_\e) }\|\nabla \phi\|_{L^3(\Omega_\e)} + \|\nabla \uu_\e\|_{L^2(\Omega_\e) }\|\phi\|_{L^\infty(\Omega_\e)} \big) \\
    &\qquad + \e^{\frac{3-\a}{2}-} \|\phi\|_{W^{1, \infty-}(\Omega)} + (\e^\frac{\a - 1}{2} + \e^\frac{3-\a}{2} + \e^{3-2\a})  \|\phi\|_{L^2(\O_\e)}\\
    &\lesssim \big(\e^\frac{\a - 1}{2} + \e^{\frac{3-\a}{2}-} + \e^{3-2\a}\big)\|\phi\|_{W^{1, \infty-}(\Omega)}.
\end{align*}
}

Finally, we rewrite
\begin{align*}
    \int_{\O_\e} (p_\e - p) \dive(W_\e \phi) \dx &= \int_\O (\tilde p_\e - p) W_\e : \nabla \phi \dx \\
    &=\int_\O (\tilde p_\e - p) \Id :\nabla \phi \dx + \int_\O \tilde p_\e (W_\e - \Id) : \nabla \phi \dx \\
    &\qquad + \int_\O p (\Id - W_\e) : \nabla \phi \dx.
\end{align*}

{\tc By \eqref{estW-I}, together with the estimate for $\tilde p_\e$ in \eqref{pe-dec} with $\th$ sufficiently close to $0$ and $\kappa$ sufficiently close to $1$,  we infer
\ba
&    \int_\O \tilde p_\e (W_\e - \Id) : \nabla \phi \dx \\
&\quad \lesssim \|\tilde p_\e^{(1)}\|_{L^2(\O_\e)} \|W_\e - \Id\|_{L^{2+}(\O_\e)} \|\nabla \phi\|_{L^{\infty-}(\Omega_\e)}  \\
    & \qquad + \left( \e^{(\l - \a)- } \|\tilde p_\e^{(2)}\|_{L^{1+}(\O_\e)} + \e^{(3-2\a)- }\|\tilde p_\e^{(3)}\|_{L^{1+}(\O_\e)}   \right) \|W_\e - \Id\|_{L^\infty(\O_\e)} \|\nabla \phi\|_{L^{\infty-}(\Omega_\e)}  \\
    & \qquad +   \e^\frac{(3-2\a)(r-2)}{r} \|p_\e^{(4)}\|_{L^\frac{r}{r-1}(\O_\e)}  \mathbf{1}_{r>2}  \|W_\e - \Id\|_{L^{r+}(\O_\e)} \|\nabla \phi\|_{L^{\infty-}(\Omega_\e)}  \\
    & \quad  \lesssim \big(\e^{\a - 1 } +  \e^{(\l - \a)-} + \e^{(3-2\a)-} + \e^{\frac{(3-2\a)(r-2)}{r} + \min\{1, \frac{3}{r}\} -} \mathbf{1}_{r >2}\big) \|\nabla \phi\|_{L^{\infty-}(\Omega_\e)},\\
&    \int_\O p (\Id - W_\e) : \nabla \phi \dx  \lesssim  \|p\|_{L^2(\O_\e)}  \|\Id - W_\e\|_{L^{2+}(\O_\e)}  \|\nabla \phi\|_{L^{\infty-}(\Omega_\e)} \lesssim \e^{\a-1}  \|\nabla \phi\|_{L^{\infty-}(\Omega_\e)} .
\ea

Observing for any $1<\a<3$ and any $r>2$ that
$$
\frac{(3-2\a)(r-2)}{r} + \min \left\{1, \frac{3}{r} \right\}> 3 - 2 \a,
$$
we collect the estimates above to conclude
\begin{align*}
    \int_\O (\tilde p_\e - p) \dive \phi \dx \lesssim \big(\e^\frac{\a - 1}{2} + \e^{\frac{3-\a}{2}-} + \e^{(\l - \a)- } + \e^{(3-2\a)- } + \e^{\frac{(3-2\a)(r-2)}{r}-} \mathbf{1}_{r>2}  \big) \|\phi\|_{W^{1, \infty-}(\Omega)}.
\end{align*}

Since $\Omega$ is the torus, for any $f \in L_0^q(\O)$ with $1<q<\infty$,  there exists $\phi \in W^{1,q}(\O; \R^3)$ such that
\begin{align*}
    \dive \phi = f, \qquad \|\phi\|_{W^{1,q}(\O)} \leq C \|f\|_{L^q(\O)}.
\end{align*}
Consequently, using also that $\tilde p_\e$ and $p$ have mean value zero, we find 
\begin{align*}
    \|\tilde p_\e - p\|_{L^{1}(\O)} \lesssim  \|\tilde p_\e - p\|_{L^{1+}(\O)} & = \sup \left\{ \int_\O (\tilde p_\e - p) f \dx, \ f\in L_0^{\infty-}(\Omega), \ \|f\|_{ L^{\infty-}(\O)}=1 \right\} \\
    &\lesssim  (\e^\frac{\a - 1}{2} + \e^{\frac{3-\a}{2}-} + \e^{(\l - \a)- } + \e^{(3-2\a)- } + \e^{\frac{(3-2\a)(r-2)}{r}-} \mathbf{1}_{r>2}  ).
\end{align*}
}

\begin{remark}
    Time-dependent estimates follow similarly to the convergence proof for the velocity. Integrating the equations in time only and using the convergences for the velocity already obtained, once we noticed that $\d_t : L^2(0,T;L^{1}(\O)) \to W^{-1,2}(0,T; L^{1}(\O))$ is a continuous operator, we conclude
    \begin{align*}
        \|\tilde p_\e - p\|_{W^{-1,2}(0,T; L^{1}(\O))} \leq C (\e^\frac{\a - 1}{2} + \e^{\frac{3-\a}{2}-} + \e^{(\l - \a)- } + \e^{(3-2\a)- }  + \e^{\frac{(3-2\a)(r-2)}{r}-} \mathbf{1}_{r>2} ).
    \end{align*}
\end{remark}

%

\appendix

\section{Proof of Proposition~\ref{lem-local-1}}
In \cite{All90-1, All90-2}, Allaire employed the following problem of Stokes equations in exterior domain $\R^3 \setminus \mathcal{T}$, called the {\em  local problem}:
\be\label{pb-local}
\begin{cases}
-\Delta \vv^i + \nabla q^i =  0 & \text{in } \R^3 \setminus  \mathcal{T},\\
\dive \vv^i =0 & \text{in } \R^3 \setminus  \mathcal{T},\\
\vv^i =0 & \text{on } \mathcal{T},\\
\vv^i = \vc e_i, & \text{at infinity},
\end{cases}
\ee
to construct a family of functions $(\vv_\e^i, q_\e^i) \in W^{1,2}(\O_{\e};\R^{3})\times L_0^{2}(\O_{\e})$ which vanish on the holes in order to modify the $C_{c}^{\infty}(\O)$ test functions and derive the limit equations as $\e\to 0$. Here, $\{\vc e_i\}_{i=1,2,3}$ is the standard Euclidean basis of $\R^3$. Allaire showed that the Dirichlet problem \eqref{pb-local} is well-posed in $D^{1,2}(\R^{3}\setminus \mathcal{T};\R^{3}) \times L_0^{2}(\R^{3}\setminus \mathcal{T})$ and showed some decay estimates of the solutions at infinity, where $D^{1,2}$ denotes the homogeneous Sobolev space. The corresponding functions $(\vv_\e^i, q_\e^i)$ are defined as follows: in cubes $\e Q_{k}$ that intersect with the boundary of $\O$,
\ba\label{def-vi-pi-1}
\vv_{\e}^{i} = \vc e_i, \ q_\e^i = 0, \ \mbox{in $\e Q_{k} \cap \O$}, \qquad \mbox{if} \ \e Q_{k} \cap \d \O \neq \emptyset;
\ea
and in cubes $\e Q_{k}$ whose closures are contained in $\O$,
\ba\label{def-vi-pi-2}
\vv_{\e}^{i} &= \vc e_i, \qquad  q_\e^i = 0, \ &&\mbox{in $\e Q_{k} \setminus  B(\e x_{k},  \frac{\e}{2})$}, \\
-\Delta \vv_{\e}^{i} +  \nabla q_\e^i &= 0, \qquad \dive \vv_{\e}^{i} = 0,\ &&\mbox{in $ B(\e x_{k},  \frac{\e}{2}) \setminus  B(\e x_{k}, \frac{\e}{4})$}, \\
\vv_{\e}^{i} (x) &=  \vv^{i} \bigg(\frac{x - \e x_k}{\e^{\a}}\bigg), \qquad q_\e^i (x) = \frac{1}{\e^{\a}} q^{i}\bigg (\frac{x - \e x_k}{\e^{\a}}\bigg),\ &&\mbox{in $ B(\e x_{k}, \frac{\e}{4}) \setminus  \mathcal{T}_{\e,k}$}, \\
\vv_{\e}^{i} &= 0, \qquad  q_\e^i = 0, \ &&\mbox{in $\mathcal{T}_{\e,k}$},
\ea
together with matching (Dirichlet) boundary conditions. {\tc Finally, the matrix $W_\e$ and the vector $Q_\e$ are defined as $W_\e = (\vv_\e^1, \vv_\e^2, \vv_\e^3)$ and $Q_\e = (q_\e^1, q_\e^2, q_\e^3)^{\rm T}$.
}

Given the functions $(\vv_\e^i, q_\e^i)$ constructed as above, the first three statements of Proposition~\ref{lem-local-1} are already proven in \cite[Proposition~3.4.12]{All90-2}, and \eqref{nabW} is proven in \cite[Lemma~3.2]{HKS21}. {\tc As a matter of fact, given the functions $\vv^i$ from problem \eqref{pb-local}, the matrix $M_0$ can be written as
\begin{align*}
    (M_0)_{ij} = \int_{\R^3 \setminus \mathcal{T}} \nabla \vv^i : \nabla \vv^j \dx,
\end{align*}
showing that $M_0$ is indeed a positive definite symmetric matrix.} Hence, we will just focus on the {\tc proofs of the remaining statements. Especially \eqref{estW-I}, meaning the estimate of $W_\e - \Id$, is of interest for the following reason:} for $q=2$, this estimate is content of \cite[Equation (3.4.35)]{All90-2} as the English translation from the original work \cite{All89PhD}, however, we want to note that the error estimates given in \cite{All89PhD, All90-2} differ from each other: indeed, equation (IV.2.4) in \cite{All89PhD} tells
\ba
\|\vv_\e^i - \vc e_i\|_{L^2(\O)}^2 \leq C (\e \sigma_\e^{-1})^2 = C \e^{\a - 1}, \nn
\ea
whereas the page before claims
\ba
\|\vv_\e^i - \vc e_i\|_{L^2(\O)}^2 \leq C [ (a_\e \e^{-1})^3 + (\e \sigma_\e^{-1})^4 ] = C [ \e^{3(\a - 1)} + \e^{2(\a - 1)} ]. \nn
\ea
On the other hand, \cite[Equation (3.4.35)]{All90-2} tells
\ba
\|\vv_\e^i - \vc e_i\|_{L^2(\O)} \leq C (\e \sigma_\e^{-1})^2 = C \e^{\a - 1}, \nn
\ea
from which the author deduces equation (3.4.49) as
\ba
\|W_\e - \Id \|_{L^2(\O)} \leq C \e \sigma_\e^{-1} = C \e^\frac{\a - 1}{2}. \nn 
\ea
Hence, we want to give here the corrected estimates, and also generalize them to all $q \in [1, \infty]$.\\

\begin{proof}[\tc Proof of \eqref{estW-I}]
First, we focus on the region $B(\e x_k, \frac{\e}{2}) \setminus B(\e x_k, \frac{\e}{4})$. There, we have
\begin{align*}
\begin{cases}
-\Delta \vv_\e^i + \nabla q_\e^i = 0, \ \dive \vv_\e^i = 0 & \text{in } B(\e x_k, \frac{\e}{2}) \setminus B(\e x_k, \frac{\e}{4}),\\
\vv_\e^i(x) = \vv^i((x-\e x_k)/\e^\a) & \text{on } \d B(\e x_k, \frac{\e}{4}),\\
\vv_\e^i = \vc e_i & \text{on } \d B(\e x_k, \frac{\e}{2}).\\
\end{cases}
\end{align*}

Setting $\vc w_\e^i(x) = \vv_\e^i(\e x + \e x_k) - \vc e_i$ and $p_\e^i(x) = \e q_\e^i(\e x + \e x_k)$, we see
\begin{align*}
\begin{cases}
-\Delta \vc w_\e^i + \nabla p_\e^i = 0, \ \dive \vc w_\e^i = 0 & \text{in } B(0, \frac{1}{2}) \setminus B(0, \frac{1}{4}),\\
\vc w_\e^i(x) = \vv^i(\e^{1-\a} x) - \vc e_i & \text{on } \d B(0, \frac{1}{4}),\\
\vc w_\e^i = 0 & \text{on } \d B(0, \frac{1}{2}).\\
\end{cases}
\end{align*}

By the pointwise expansion of $\vv^i$ given in \cite[Equation (2.3.25)]{All90-1} and in \cite[Equation (A.8)]{All90-1} , we have at infinity
\begin{align*}
\vv^i(x) = \vc e_i + \mathcal{O}(|x|^{-1}), \qquad
\nabla \vv^i(x) = \mathcal{O}(|x|^{-2}), \qquad \nabla^{2} \vv^i(x) = \mathcal{O}(|x|^{-3}).
\end{align*}
In turn, we have on $\d B(0, \frac14)$
\begin{align*}
\nabla^{l} \vc w_\e^i = \nabla^{l}(\vv^i(\e^{1-\a} x))= \mathcal{O}(\e^{l(1-\a)} \e^{(1+l)(\a-1)}) = \mathcal{O}(\e^{\a-1}), \ {\tc l \in \{0, 1, 2\}},
\end{align*}
{\tc in the sense of $L^\infty(\d B(0, \frac14))$ norm, which leads to}
\begin{align*}
 \|\vc w_\e^i\|_{W^{2,q}(B(0, \frac12) \setminus B(0, \frac14))} + \|p_\e^i\|_{W^{1,q}(B(0, \frac12) \setminus B(0, \frac14))} \leq C \e^{\a-1}, \ \forall\, q\in (1, \infty),
\end{align*}
and, by Sobolev embedding,
\begin{align*}
 \|\vc w_\e^i\|_{W^{1,\infty}(B(0, \frac12) \setminus B(0, \frac14))} + \|p_\e^i\|_{L^{\infty}(B(0, \frac12) \setminus B(0, \frac14))} \leq C \e^{\a-1}.
\end{align*}

Back to $\vv_\e^i$, this yields 
\begin{align}\label{vve1}
\begin{split}
\|\vv_\e^i - \vc e_i\|_{L^q(B(\e x_k, \frac{\e}{2}) \setminus B(\e x_k, \frac{\e}{4}))}^q &\leq C \e^3 \e^{q(\a-1)}, \  q\in (1, \infty ), \\
\|\vv_\e^i - \vc e_i\|_{L^\infty(B(\e x_k, \frac{\e}{2}) \setminus B(\e x_k, \frac{\e}{4}))} &\leq C \e^{\a-1}.
\end{split}
\end{align}

Next, we focus on the region $B(\e x_k, \frac{\e}{4}) \setminus \mathcal{T}_{\e, k}$. By definition of $\vv_\e^i$ and $q_\e^i$, it is easy to obtain that 
\ba\label{vve2}
\|\vv_\e^i - \vc e_i \|_{L^{q}(B(\e x_k, \frac{\e}{4}) \setminus \mathcal{T}_{\e, k})}^q &\leq C \e^{3\a} \| \vv^i - \vc e_i \|_{L^q(B(0, \frac{\e^{1-\a}}{4}) \setminus \mathcal{T})}^q \leq C \e^{3 \a}, \ \forall\, q\in (3, \infty),
\ea
and
\ba\label{vve3}
\|\vv_\e^i - \vc e_i \|_{L^\infty(B(\e x_k, \frac{\e}{4}) \setminus \mathcal{T}_{\e, k})} &\leq C \| \vv^i - \vc e_i \|_{L^\infty (B(0, \frac{\e^{1-\a}}{4}) \setminus \mathcal{T})} \leq C.
\ea
Here, the restriction $q>3$ is to ensure the integrability of $|\vv^{i} - \vc e_i|^{q}$ in $\R^{3} \setminus \mathcal{T}$.\\

Consequently, putting together \eqref{vve1}--\eqref{vve3} and taking into account that the number of holes inside $\O$ is of order $\e^{-3}$, we find for any $q \in (3, \infty)$ that 
\ba\label{est-vi-all}
\|\vv_\e^i - \vc e_i\|_{L^q(\O)}^q &\leq C \e^{-3} (\e^3 \e^{q(\a-1)} + \e^{3\a}) \leq C \e^{3(\a-1)}.
\ea
The estimates of $\|\vv_\e^i - \vc e_i\|_{L^q(\O)}$ with $q\in [1,3]$ {\tc follow similar lines}. The estimates for the pressure $q_{\e}^{i}$ can be derived similarly.
\end{proof}

{\tc

\begin{proof}[Proof of \eqref{W_O_varpi}] {\tc Recall the definition of $\O_\e$ from \eqref{Omega-e}, and define
$$
K_\e': = \{ k \in \Z^3:  \e Q_k \cap \d \Omega \neq \emptyset\}, \quad K_\e^\varpi :=   \{ k \in K_\e : \e Q_k \cap \O^\varpi \neq \emptyset \}.
$$
Then, for $\e>0$ small enough, and that for $\varpi \geq \e$, we have
$$
\# K_\e'  \lesssim \e^{-2}, \quad  \# K_\e^\varpi  \lesssim  \varpi \e^{-3},
$$
where $\# K$ denotes the cardinality of set $K$. Using the fact that $W_\e$ is $\e Q_0$ periodic in $\R^3$ gives
\ba
\|\nabla W_\e\|_{L^q(\O^\varpi)}^q & = \sum_{k \in K_\e^\varpi} \int_{\e Q_k \cap \O^\varpi} |\nabla W_\e|^q \dx +  \sum_{k \in K_\e'} \int_{\e Q_k \cap   \O^\varpi} |\nabla W_\e|^q \dx \\ 
& \leq \sum_{k \in K_\e^\varpi} \int_{\e Q_k} |\nabla W_\e|^q \dx + \sum_{k \in K_\e'} \int_{\e Q_k} |\nabla W_\e|^q \dx \\
&= (\# K_\e^\varpi )  \int_{\e Q_0} |\nabla W_\e|^q \dx + ( \# K_\e') \int_{\e Q_0} |\nabla W_\e|^q \dx  \\
& \lesssim (\varpi + \e) \e^{-3} \int_{\e Q_0} |\nabla W_\e|^q \dx \\ 
& \lesssim \varpi  \e^{-3} \int_{\e Q_0} |\nabla W_\e|^q \dx.
\nn
\ea
Together with the observation
\ba
\|\nabla W_\e\|_{L^q(\Omega)}^q & \geq  \sum_{k \in K_\e } \int_{\e Q_k } |\nabla W_\e|^q \dx  \geq c \e^{-3}  \int_{\e Q_0} |\nabla W_\e|^q  \dx,
\nn
\ea
we thus conclude that
\begin{align*}
\|\nabla W_\e\|_{L^q(\O^\varpi)}^q  \lesssim \varpi \|\nabla W_\e\|_{L^q(\Omega)}^q.
\end{align*}

Analogously, we have
\begin{align*}
\| W_\e - \Id\|_{L^q(\O^\varpi)}^q  \lesssim \varpi \|W_\e - \Id\|_{L^q(\Omega)}^q.
\end{align*}
}
\end{proof}
}

{\tc 
\begin{proof}[Proof of \eqref{We-Qe-eq}]
    The last conclusion \eqref{We-Qe-eq} has been shown in \cite{All90-1, All90-2}. More precisely, one can decompose $(-\Delta W_\e + \nabla Q_\e) = M_\e + \gamma_\e$, where $\gamma_\e$ is supported on $\partial \Omega_\e$, hence $-\Delta W_\e + \nabla Q_\e = M_\e$ in $W^{-1,2}(
\O_\e)$, see \cite[Eq. (3.4.36)]{All90-2}. Estimate $\eqref{We-Qe-eq}$ is then \cite[Eq. (3.4.40)]{All90-2}, which follows from  \cite[Lemma~2.4.3]{All90-1}.
\end{proof}
} The proof of Proposition~\ref{lem-local-1} is thus completed. 

\section*{Acknowledgements}
{\it Y. L. has been supported by the Recruitment Program of Global Experts of China. Y. L. is partially supported  by NSF of Jiangsu Province under Grant BK20240058 and by NSF of China under Grant 12171235.   F. O. has been supported by the Czech Academy of Sciences project L100192351. The Institute of Mathematics, CAS is supported by RVO:67985840.}



\begin{thebibliography}{000}
\bibitem{All89PhD} G. Allaire. \newblock Homog{\'e}n{\'e}isation des {\'e}quations de Stokes et de Navier-Stokes. \newblock{\em PhD thesis} (1989).

\bibitem{All90-1} G. Allaire. \newblock Homogenization of the Navier-Stokes equations in open sets perforated with tiny holes. I. Abstract framework, a volume
distribution of holes. \newblock {\em Arch. Ration. Mech. Anal.}  113 (3) (1990), 209--259.

\bibitem{All90-2} G. Allaire. \newblock Homogenization of the Navier-Stokes equations in open sets perforated with tiny holes. II. Noncritical sizes of the holes for a volume distribution and a surface distribution of holes. \newblock {\em Arch. Ration. Mech. Anal.}  113 (3) (1990), 261--298.

\bibitem{ABSG22} M. Anguiano, M. Bonnivard, F.J. Su\'arez-Grau. \newblock Carreau law for non-Newtonian fluid flow through a thin porous media. \newblock{\em Q. J. Mech. Appl. Math.} 75.1 (2022), 1-27.

\bibitem{ASG21} M. Anguiano, F.J. Su\'arez-Grau. \newblock Lower-dimensional nonlinear Brinkman’s law for non-Newtonian flows in a thin porous medium. \newblock{\em Mediterr. J. Math.} 18, (2021), 175.

\bibitem{ASG25} M. Anguiano, F.J. Su\'arez-Grau. \newblock Modeling non-Newtonian fluids in a thin domain perforated with cylinders of small diameter. \newblock{\em arXiv preprint  	arXiv:2508.04688} (2025).

\bibitem{BalaziAllaireOmnes24} L. Balazi, G. Allaire, P. Omnes. Sharp convergence rates for the homogenization of the Stokes equations in a perforated domain. \newblock{\em HAL-04541828} (2024).

\bibitem{BasaricChaud} D. Basari\'c, N. Chaudhuri. Low Mach number limit on perforated domains for the evolutionary Navier--Stokes--Fourier system. \newblock {\em Nonlinearity}  37 (6) (2024), Paper No. 065008.


\bibitem{BO22} P. Bella, F. Oschmann. Homogenization and low Mach number limit of compressible Navier-Stokes equations in critically
perforated domains. \newblock{\em J. Math. Fluid Mech.} 24 (3) (2022), 1--11.


\bibitem{BO23} P. Bella, F. Oschmann. Inverse of divergence and homogenization of compressible Navier-Stokes equations in randomly perforated domains. \newblock {\em Arch. Ration. Mech. Anal.} 247 (2) (2023), Paper No. 14.

\bibitem{BFO23} P. Bella, E. Feireisl, F. Oschmann. $\Gamma $-convergence for nearly incompressible fluids. \newblock {\em J. Math. Phys.} 64  (9) (2023), Paper No. 091507.



\bibitem{BOG}
M.~E. Bogovski\u{\i}.
\newblock Solution of some vector analysis problems connected with operators
  div and grad (in {R}ussian).
\newblock {\em Trudy Sem. S.L. Sobolev}, 80 (1) (1980), 5--40.

\bibitem{BM96}A. Bourgeat and A. Mikeli\'{c}.\newblock Homogenization of a polymer flow through a porous medium. \newblock {\em Nonlinear Analysis. Theory. Methods.}  26 (7) (1996), 1221--1253.



\bibitem{BGMS12}  M. Bul{\'i}{\v c}ek, P. Gwiazda, J. M\'alek, and A. \'Swierczewska-Gwiazda. On unsteady flows of implicitly constituted incompressible fluids. \newblock {\em SIAM J. Math. Anal.} 44 (4) (2012), 2756--2801.




\bibitem{DFL17} L. Diening, E. Feireisl, and Y. Lu. \newblock The inverse of the divergence operator on perforated domains with applications to homogenization problems for the compressible Navier--Stokes system. \newblock {\em ESAIM Control Optim. Calc. Var.} {23} (2017), 851--868.

\bibitem{DRW10} L. Diening, M. R\r{u}\v{z}i\v{c}ka, and J. Wolf. Existence of weak solutions for unsteady motions of generalized Newtonian fluids. \newblock {\em  Ann. Sc. Norm. Super. Pisa CI. Sci.} 9 (5) (2010), 1--46.



\bibitem{FL15} E. Feireisl and Y. Lu. Homogenization of stationary Navier--Stokes equations in domains with tiny holes. {\em J. Math. Fluid Mech.} 17 (2015), 381--392.

\bibitem{FNN16}
E.~Feireisl, Y.~Namlyeyeva, and {\v S.}~Ne{\v c}asov{\' a}.
\newblock Homogenization of the evolutionary {N}avier--{S}tokes system.
\newblock {\em Manuscripta Math.} {149} (2016), 251--274.

\bibitem{F-N-book} E. Feireisl and A. Novotn\'y. {\em Singular Limits in Thermodynamics of Viscous Fluids.} Birkh\"auser Verlag, Basel, 2009.


\bibitem{FNT10} E. Feireisl, A. Novotn\'y, and  T. Takahashi. \newblock Homogenization and singular limits for the complete Navier--Stokes--Fourier system.
\newblock {\em J. Math. Pures Appl.} {94} (2010), 33--57.


\bibitem{Gal94}
G.~P. Galdi.
\newblock {\em An introduction to the mathematical theory of the {N}avier--{S}tokes equations, I.}
\newblock Springer-Verlag, New York, 1994.

\bibitem{Hoefer23}  R. M. H\"ofer. Homogenization of the Navier--Stokes equations in perforated domains in the inviscid limit. \newblock {\em Nonlinearity} 36 (11) (2023), 6020.

\bibitem{HKS21}  R. M. H\"ofer, K. Kowalczyk, and S. Schwarzacher. Darcy's law as low Mach and homogenization limit of a compressible fluid in perforated domains. \newblock {\em Math. Models Methods Appl. Sci.} 31 (9) (2021), 1787-1819.


\bibitem{HNO24}  R. M. H\"ofer, \v{S}. Ne\v{c}asov\'a, and F. Oschmann. Quantitative homogenization of the compressible Navier-Stokes equations towards Darcy's law. \newblock {\em Annales de l'Institut Henri Poincar{\'e} C} (2025).

\bibitem{JLP25}  W. Jing, Y. Lu, and C. Prange. Unified quantitative analysis of the Stokes equations in dilute perforated domains via layer potentials. \newblock {\em Multiscale Modeling \& Simulation} 23 (3) (2025), 1145-1182.

 \bibitem{Hor97} U. Hornung (Ed.). Homogenization and Porous Media. Interdisciplinary Applied Mthematics Series, vol. 6, Springer-Verlag, New York, 1997.


\bibitem{Kato83} T. Kato. Remarks on zero viscosity limit for nonstationary Navier-Stokes flows with boundary. {\it Seminar on nonlinear partial differential equations (Berkeley, Calif., 1983)}. 85--98, Math. Sci. Res. Inst. Publ., 2, Springer, New York, 1983.


\bibitem{Lad69}  O. A. Ladyzhenskaya. The Mathematical Theory of Viscous Incompressible Flow. Gordon and Breach, New York, 1969.






\bibitem{Lu21} Y. Lu. Uniform estimates for Stokes equations in a domain with a small hole and applications in homogenization problems. \newblock {\em Calc. Var. Partial Differential Equations} {60} (6) (2021), Paper No. 228, 31 pp.


\bibitem{LQ23}  Y. Lu and Z. Qian. Homogenization of some evolutionary non-Newtonian flows in porous media. \newblock {\em J. Differential Equations} {411} (2024), 619--639.

\bibitem{LS18}Y. Lu and S. Schwarzacher.  Homogenization of the compressible Navier--Stokes equations in domains with very tiny holes. {\em J. Differential Equations} {265} (4) (2018), 1371--1406.


  \bibitem{Lu20} Y. Lu. Homogenization of Stokes equations in perforated domains: a unified approach. \newblock{\em J. Math. Fluid Mech.} 22 (3) (2020), Paper No. 44.


\bibitem{YL23} Y. Lu, P. Yang. Homogenization of evolutionary incompressible Navier-Stokes system in
perforated domains. \newblock{\em J. Math. Fluid Mech.} 25 (1) (2023), Paper No. 4.

\bibitem{MNRR} J. M\'alek, J. Ne\v{c}as, M. Rokyta, M. R\r{u}\v{z}i\v{c}ka. Weak and Measure-valued Solutions to Evolutionary PDEs. \newblock {\em Vol. 13. CRC Press.} (1996).


\bibitem{MPM96}
 E. Maru\u si\'c-Paloka and A. Mikeli\'c. An error estimate for correctors in the homogenization of the Stokes and the Navier-Stokes equations in a porous medium.
 \newblock{\em Boll. Un. Mat. Ital. A (7).}  10.3 (1996), pp. 661--671.


\bibitem{Mas02} N. Masmoudi. Homogenization of the compressible Navier-Stokes equations in a porous medium. \newblock {\em ESAIM Control Optim. Calc. Var}. 8 (2002), 885-906.



\bibitem{Mas04} N. Masmoudi. Some uniform elliptic estimates in a porous medium. {\em C. R. Math. Acad. Sci. Paris}, {339}(12) (2004),  849-854.


\bibitem{Mik91} A. Mikeli\'{c}. Homogenization of nonstationary Navier-Stokes equations in a domain with a grained boundary.
  \newblock {\em Ann. Mat. Pura Appl.}  158 (1991), 167-179.
  
  


\bibitem{NO23}  {\v S}. Ne{\v c}asov{\'a}, F. Oschmann. Homogenization of the two-dimensional evolutionary compressible Navier-Stokes equations. \newblock {\em Calc. Var. Partial Differential Equations.} 62 (6) (2023), Paper No. 184.

\bibitem{OP23} F. Oschmann, M. Pokorn{\' y}. Homogenization of the unsteady compressible Navier-Stokes equations for adiabatic exponent $\gamma> 3$.  \newblock {\em J. Differential Equations.} 377 (2023), 271-296.

\bibitem{She22} Z. Shen.  Sharp convergence rates for Darcy’s law. \newblock{\em Communications in Partial Differential Equations.}
47.6 (2022), pp. 1098–1123.



\bibitem{Tar80} L. Tartar. Incompressible fluid flow in a porous medium: convergence of the homogenization process, in\newblock {\em Nonhomogeneous media and vibration theory,} edited by E. S\'anchez-Palencia, (1980), 368-377.

\bibitem{Sohr01} H. Sohr. The Navier-Stokes Equations. An Elementary Functional Analytic Approach. \newblock {\em Birkh\"auser Advanced Texts.} Basel-Boston-Berlin (2001).


\bibitem{Temam79} R. Temam. Navier-Stokes Equations. North-Holland, Amsterdam, 1979.





 \end{thebibliography}
\end{document}